\documentclass[final,leqno]{amsart}
\usepackage{amsmath,amssymb,amsfonts,latexsym,stmaryrd}
\usepackage{graphicx,float,epsfig} 
\usepackage{mathrsfs} 
\usepackage{color}

\DeclareGraphicsExtensions{ .eps,.ps} \usepackage{subfig}
\usepackage{multirow}
\usepackage{url}
\usepackage[linesnumbered,algoruled,boxed]{algorithm2e}
\usepackage[colorlinks=false, linkcolor=blue,  citecolor=OliveGreen,  pdfstartview={}{}]{hyperref} 
\usepackage[dvipsnames]{xcolor}

\def\ds{\displaystyle}

\def\O{\Omega}

\def\t{\theta}

\def\mV{\mathcal{V}}

\renewcommand\sp{\mathop{\mathrm{Sp}}\nolimits}
\newtheorem{remark}{Remark}[section]
\newtheorem{lemma}{Lemma}[section]
\newtheorem{theorem}{Theorem}[section]

\newtheorem{corollary}{Corollary}[section]

\newcommand{\set}[1]{\lbrace #1 \rbrace}

\newcommand{\norm}[1]{\lVert#1\rVert}

\usepackage{booktabs}
\usepackage{float}

\newcommand\bu{\boldsymbol{u}}
\newcommand\bv{\boldsymbol{v}}

\def\hdel{\widehat{\delta}}
\def\CM{\mathcal{X}}
\def\CN{\mathcal{Y}}

\newcommand\bI{\boldsymbol{I}}

\newcommand\bT{\boldsymbol{T}}

\newcommand\0{\mathbf{0}}

\newcommand\bxi{\boldsymbol{\xi}}

\def\CT{{\mathcal T}}

\newcommand{\dd}{\texttt{d}}

\newcommand\bsig{\boldsymbol{\sigma}}
\newcommand\btau{\boldsymbol{\tau}}

\newcommand\R{\mathbb{R}}

\renewcommand\H{\mathrm{H}}
\renewcommand\L{\mathrm{L}}

\renewcommand\O{\Omega}
\newcommand\DO{\partial\O}

\newcommand\bdiv{\mathop{\mathbf{div}}\nolimits}

\renewcommand\div{\mathop{\mathrm{div}}\nolimits}

\newcommand\tr{\mathop{\mathrm{tr}}\nolimits}

\renewcommand\sp{\mathop{\mathrm{sp}}\nolimits}

\renewcommand\t{\mathtt{t}}

\newcommand\LO{\L^2(\O)}

\newcommand\HsO{\H^s(\O)}
\newcommand\nxn{n\times n}

\renewcommand\t{\mathtt{t}}

\newcommand{\vertiii}[1]{{\left\vert\kern-0.25ex\left\vert\kern-0.25ex\left\vert #1 
    \right\vert\kern-0.25ex\right\vert\kern-0.25ex\right\vert}}

\begin{document}

\title[Mixed methods for the Stokes spectral problem]
{Mixed methods for the velocity-pressure-pseudostress formulation of the Stokes eigenvalue problem}


\author{Felipe Lepe}
\address{GIMNAP-Departamento de Matem\'atica, Universidad del B\'io - B\'io, Casilla 5-C, Concepci\'on, Chile.}
\email{flepe@ubiobio.cl}
\thanks{The first author was partially supported by
ANID-Chile through FONDECYT project 11200529 (Chile).}

\author{Gonzalo Rivera}
\address{Departamento de Ciencias Exactas,
Universidad de Los Lagos, Casilla 933, Osorno, Chile.}
\email{gonzalo.rivera@ulagos.cl}
\thanks{The second author was supported by
ANID-Chile through FONDECYT project 11170534 (Chile).}

\author{Jesus Vellojin}
\address{Departamento de Matem\'atica,
Universidad T\'ecnica Federico Santa Mar\'ia, Valpara\'iso, Chile.}
\email{jesus.vellojinm@usm.cl}
\thanks{The third author was partially supported by ANID-Chile through FONDECYT project 1181098 (Chile).}




\begin{abstract}
In two and three dimensional domains, we analyze mixed finite element methods for a velocity-pressure-pseudostress formulation
of the Stokes eigenvalue problem. The methods consist in  two schemes:  the velocity and pressure are approximated with piecewise polynomial and for the pseudostress we consider two classic families of finite elements for $\H(\div)$ spaces: the Raviart-Thomas and the Brezzi-Douglas Marini elements.
 With the aid of the classic spectral theory for compact operators, we prove that our method does not introduce spurious modes. Also, we obtain convergence and error estimates for the proposed methods. In order to assess the performance of the schemes, we report numerical results to compare the accuracy and robustness between both numerical schemes.
\end{abstract}

\maketitle

\section{Introduction}\label{sec:intro}
The Stokes problem is a system of equations that describes the motion of a certain fluid. For a given domain $\O\subset\mathbb{R}^n$, where $n\in\{2,3\}$ with Lipschitz boundary, we are interested in the Stokes eigenvalue problem.
\begin{equation}\label{def:stokes_source}
\left\{
\begin{array}{rcll}
-2\mu\Delta\bu+\nabla p & = & \lambda\bu&  \text{ in } \quad \Omega, \\
\div\bu & = & 0 & \text{ in } \quad \Omega, \\
\bu & = & \boldsymbol{0} & \text{ on } \quad \partial\Omega,
\end{array}
\right.
\end{equation}
 where is $\mu$ is the kinematic viscosity,  $\bu$ is the velocity and $p$ is the pressure. 
 
 It is well known that mixed formulations are a suitable alternative to analyze different problems, since the introduction of
 additional unknowns with physical meaning, allows to obtain more information for certain phenomenons. Hence, the design 
 of finite element approximations has been an important subject of study for mathematicians and engineers, where several 
 families of mixed elements have been developed. For a complete state of art about mixed methods we resort to \cite{MR3097958}.
 
 In particular, mixed formulations for eigenvalue problems has been well developed in the past years and the theory to
 study these problems can be found in \cite{MR2652780,MR606505}, just for mention the more classic references. On the other hand, concrete applications for mixed formulations
 in spectral problems can be found in different contexts as, for instance,  \cite{MR2220929,MR2559736,MR4071826,MR3864690,MR3962898,MR3451491}, where several tools 
 have been implemented as DG methods, VEM methods, FEM, and a posteriori analysis.
 
In the present work, we consider  a tensorial formulation for the Stokes spectral problem. This type of formulation naturally arise when we are interested in
the computation of the stress.

More precisely, we will study the Stokes eigenvalue problem with the mixed formulation proposed in \cite{MR2594823} for the source problem where, not only the velocity and the pressure are the unknowns as in \eqref{def:stokes_source}, but also the pseudostress tensor (see \cite{cai2010} for further details related to this tensor). 
With these formulation, clearly the computational costs for the numerical methods increment compared with the classic velocity-pressure formulation, since we need to approximate each component of the pseudostress, each component of the velocity an the scalar 
associated to the pressure. However, this tensor is an interesting unknown since allows to compute other quantities of interest. For example, in the Stokes flow problems, the pseudostress  relates the classic stress and the gradient of the velocity. Hence, with an accurate approximation of the pseudostress we are able to obtain accurate values for these other relevant unknowns.

One of the motivations to analyze the mixed formulation of our work is that  allows to deal with eigenfunctions that present a poor regularity compared with those of the classic velocity-pressure formulation, which is a clear advantage when functions of these nature are presented in real applications. Moreover, mixed formulations are flexible in the choice of finite element families to approximate the space $\H(\div)$. In our case, we will consider two families: 
the Raviart-Thomas elements and the Brezzi-Douglas-Marini elements (see for instance \cite{MR799685,MR0483555}). The aim is to compare the accuracy of these inf-sup stable finite elements to approximate the eigenfunctions and eigenvalues of \eqref{def:stokes_source}. It is well known that BDM schemes are more expensive than RT schemes, which immediately give us the computational costs as first main difference. However, in eigenvalue problems, the orders of convergence and accuracy for the approximation of the spectrum of the solution operators can benefited with more expensive elements. Also, in the present work we perform a theoretical and computational analysis for high order mixed methods for both families of finite elements, which becomes an important feature to compare.

The paper is organized as follows: in Section \ref{sec:model_problem} we introduce the Stokes eigenvalue problem, 
for two and three dimensions, together with the pseudostress tensor. With suitable Hilbert spaces we derive a variational formulation for \eqref{def:stokes_source} where the main unknowns are the pseudostress, the velocity and the pressure. We introduce the corresponding solution operators and present an additional regularity result for the eigenfunctions. Finally, a spectral characterization is deduced. Section \ref{sec:mixed_fem} is the core of our paper, where we introduce the finite element schemes of our analysis. We prove the stability for the discrete eigenvalue problem. Also we introduce the discrete solution operator. In Section \ref{sec:conv} we perform the spectral analysis, where convergence and error results for the eigenfunctions and eigenvalues  are proved. Finally, in Section \ref{sec:numerics} we report a series of numerical tests where we confirm our theoretical results, together with a comparison between the mixed schemes of our paper. 

We end this section with some of the notations that we will  use below. Given
any Hilbert space $X$, let $X^2$ and $\mathbb{X}$ denote, respectively,
the space of vectors and tensors  with
entries in $X$. In particular, $\mathbb{I}$ is the identity matrix of
$\R^{n\times n}$, and $\mathbf{0}$ denotes a generic null vector or tensor. 
Given $\btau:=(\tau_{ij})$ and $\bsig:=(\sigma_{ij})\in\R^{n\times n}$, 
we define, as usual, the transpose tensor $\btau^{\t}:=(\tau_{ji})$, 
the trace $\tr\btau:=\sum_{i=1}^n\tau_{ii}$, the deviatoric tensor 
$\btau^{\dd}:=\btau-\frac{1}{n}\left(\tr\btau\right)\mathbb{I}$, and the
tensor inner product $\btau:\bsig:=\sum_{i,j=1}^n\tau_{ij}\sigma_{ij}$. 

Let $\O$ be a polygonal Lipschitz bounded domain of $\R^n$ with
boundary $\DO$. For $s\geq 0$, $\norm{\cdot}_{s,\O}$ stands indistinctly
for the norm of the Hilbertian Sobolev spaces $\HsO$, $\HsO^n$ or
$\mathbb{H}^s(\O)$ for scalar, vectorial and tensorial fields, respectively, with the convention $\H^0(\O):=\LO$, $[\H^0(\O)]^{n}=[\L^2(\O)]^n$ and $\mathbb{H}^0(\O):=\mathbb{L}^2(\O)$. We also define for
$s\geq 0$ the Hilbert space 
$\mathbb{H}^{s}(\bdiv;\O):=\set{\btau\in\mathbb{H}^s(\O):\ \bdiv\btau\in[\HsO]^n}$, whose norm
is given by $\norm{\btau}^2_{\mathbb{H}^s(\bdiv;\O)}
:=\norm{\btau}_{s,\O}^2+\norm{\bdiv\btau}^2_{s,\O}$.

The relation $\texttt{a} \lesssim \texttt{b}$ indicates that $\texttt{a} \leq C \texttt{b}$, with a positive constant $C$ which is 
independent of $\texttt{a}$ and  $\texttt{b}$.

\section{The model problem}
\label{sec:model_problem}

Let $\O$ be a bounded simply connected polygonal domain with boundary $\partial\O$. We introduce the pseudostress tensor $\bsig:=2\mu\nabla\bu-p\mathbb{I}$. Hence, system \eqref{def:stokes_source} is rewritten as follows:
\begin{equation}
\label{def:stokes_eigen}
\left\{
\begin{array}{rccc}
\bdiv\bsig&=&-\lambda\bu&\quad\text{in}\,\O,\\
\bsig-2\mu\nabla\bu+p\mathbb{I}&=&0&\quad\text{in}\,\O,\\
\bdiv\bu&=&0&\quad\text{in}\,\O, \\
\bu&=&\boldsymbol{0}&\quad\text{on}\,\partial\O,
\end{array}
\right.
\end{equation}
where $\mu$ is the kinematic viscosity, $\mathbb{I}\in\mathbb{R}^{n\times n}$ is the identity matrix, and $\bdiv$ must be understood as the divergence of any tensor applied 
along on each row. As is commented in \cite{MR2594823}, the pressure and the pseudostress tensor are related
through the following identity
\begin{equation*}
\ds p=-\frac{1}{n}\tr(\bsig)\quad\text{in}\,\O.
\end{equation*}
This identity holds by the incompresibility condition, together with the identity $\tr(\nabla\bu)=\div\bu$. Hence, problem \eqref{def:stokes_eigen} can be rewritten as the following system
\begin{equation}
\label{def:stokes_eigen2}
\left\{
\begin{array}{rccc}
\bdiv\bsig&=&-\lambda\bu&\quad\text{in}\,\O,\\
\bsig-2\mu\nabla\bu+p\mathbb{I}&=&0&\quad\text{in}\,\O,\\
\displaystyle p+\frac{1}{n}\tr(\bsig)&=&0&\quad\text{in}\,\O, \\
\bu&=&\boldsymbol{0}&\quad\text{on}\,\partial\O.
\end{array}
\right.
\end{equation}
For the analysis of problem \eqref{def:stokes_eigen2}, we are interested in the following variational formulation: Find $\lambda\in\mathbb{R}$ and the triplet $\boldsymbol{0}\neq(\bsig,p,\bu)\in\mathbb{H}(\bdiv,\O)\times \L^2(\O)\times[\L^2(\O)]^n$ such that
\begin{align}
\label{eq01}\frac{1}{2\mu}\int_{\O}\bsig^{\dd}:\btau^{\dd}+\frac{1}{\mu}\int_{\O}\left(p+\frac{1}{n}\tr(\bsig)\right)&\left(q+\frac{1}{n}\tr(\btau)\right)\\\nonumber
+\int_{\O}\bu\cdot\bdiv\btau.&=0\,\,\,\,\quad\forall (\btau,q)\in\mathbb{H}(\bdiv,\O)\times \L^2(\O),\\
\label{eq02}\int_{\O}\bv\cdot\bdiv\bsig&=-\lambda(\bu,\bv)\quad\forall\bv\in [\L^2(\O)]^n,
\end{align}


It is important to remark that problem \eqref{eq01}--\eqref{eq02} has a solution, but the uniqueness is no satisfied. To avoid this circumvent, and inspired by \cite[Section 2]{MR2594823}, we consider the following decomposition
of $\mathbb{H}(\bdiv,\O)=\mathbb{H}_0\oplus\mathbb{R}\mathbb{I}$, where

\begin{equation*}
\mathbb{H}_0:=\left\{\btau\in\mathbb{H}(\bdiv,\O)\,:\,\int_{\O}\tr(\btau)=0\right\}.
\end{equation*}

In order to simplify the presentation of the material, we define $\mathbb{H}:=\mathbb{H}_0\times \L^2(\O)$ and $\mathbf{Q}^{\bu}:=[\L^2(\O)]^n$. The bilinear forms $a:\mathbb{H}\times \mathbb{H}\rightarrow\mathbb{R}$ and $b:\mathbb{H}(\bdiv,\O)\times \mathbf{Q}^{\bu}\rightarrow\mathbb{R}$ are defined as follows:
\begin{equation*}
a((\bxi,r),(\btau,q)):=\frac{1}{2\mu}\int_{\O}\bxi^{\dd}:\btau^{\dd}+\frac{\gamma}{\mu}\left(p+\frac{1}{n}\tr(\bxi)\right)\left(q+\frac{1}{n}\tr(\btau)\right),
\end{equation*}
and 
\begin{equation*}
b(\bxi,\bv):=\int_{\O}\bv\cdot\bdiv\bxi.
\end{equation*}

According to \cite[Lemma 2.2]{MR2594823}, we have that any solution of problem  \eqref{eq01}--\eqref{eq02}  with $\bsig\in \mathbb{H}_0 $ is also solution of:
\begin{align}
\label{eq1}a((\bsig,p),(\btau,q))+b(\btau,\bu)&=0\,\,\,\,\quad\quad\quad\forall (\btau,q)\in\mathbb{H},\\
\label{eq2}b(\bsig,\bv)&=-\lambda(\bu,\bv)\quad\forall\bv\in \mathbf{Q}^{\bu},
\end{align}
with $\gamma=1$, and  also, any solution of \eqref{eq1}--\eqref{eq2} is also solution of problem \eqref{eq01}--\eqref{eq02}.

It is possible to consider an alternative reduced formulation for our problem \eqref{eq1}--\eqref{eq2}, which only depends on the stress tensor and the velocity. 

With the space $\mathbb{H}_0$ at hand, we consider the following problem: find $\lambda\in\mathbb{R}$ and $\0 \neq(\bsig,\bu)\in \mathbb{H}_0\times \mathbf{Q}^{\bu}$ such that
\begin{align}
\label{eq1_reduced}a_0(\bsig,\btau)+b(\btau,\bu)&=0\,\,\,\,\quad\quad\quad\forall \btau\in \mathbb{H}_0,\\
\label{eq2_reduced}b(\bsig,\bv)&=-\lambda(\bu,\bv)\quad\forall\bv\in \mathbf{Q}^{\bu},
\end{align}
where $a_0: \mathbb{H}_0\times \mathbb{H}_0\rightarrow\mathbb{R}$ is a bounded bilinear form defined by
\begin{equation*}
\displaystyle a_0(\bxi,\btau):=\frac{1}{2\mu}\int_{\O}\bxi^{\texttt{d}}:\btau^{\texttt{d}}\quad\forall (\bxi,\btau)\in \mathbb{H}_0\times \mathbb{H}_0.
\end{equation*}
We remark that the pressure can be recovered with the third equation of system \eqref{def:stokes_eigen2}.
\begin{remark}
\label{remark_1}
 It is easy to check that  if $(\lambda,\bsig,p,\bu)\in \mathbb{R}\times  \mathbb{H}\times \mathbf{Q}$ is a solution of  problem \eqref{eq1}--\eqref{eq2} if and only if $(\lambda,\bsig,\bu)\in \mathbb{R}\times  \mathbb{H}_{0}\times \mathbf{Q}$ is a solution of problem \eqref{eq1_reduced}--\eqref{eq2_reduced} and $p=-\dfrac{1}{n}\tr(\bsig)$ (see \cite[Lemma 2.3]{MR2594823}). 
\end{remark}
For the analysis of the mixed problem \eqref{eq1_reduced}--\eqref{eq2_reduced} we invoke the following result (see \cite[Ch. 4, Proposition 3.1]{MR3097958})
\begin{equation}
\label{eq:des_deviator}
\|\btau\|_{0,\O}^2\lesssim \|\btau^{\dd}\|_{0,\O}^2+\|\bdiv\btau\|_{0,\O}^2.
\end{equation}

Let us introduce the kernel of the operator induced by $b(\cdot,\cdot)$
\begin{equation*}
\mathcal{V}:=\{\btau\in \mathbb{H}_0\,:\, b(\btau,\bv)=\boldsymbol{0}\,\,\,\forall \bv\in \mathbf{Q}\}=\{\btau\in \mathbb{H}_0\,:\,\, \bdiv\btau=\boldsymbol{0}\}.
\end{equation*}

With the aid of \eqref{eq:des_deviator} it is easy to check that $a_0(\cdot,\cdot)$ is coercive in $\mathcal{V}$ (see \cite[Subsection 2.3]{MR2594823})
\begin{equation*}
a_0(\btau,\btau)\geq \alpha\|\btau\|_{\div}^2\quad\forall\btau\in\mathcal{V},
\end{equation*}
where it can be proved that $\alpha=C/2\mu$, with $C$ being the  positive constant provided by \eqref{eq:des_deviator}.

On the other hand, there exists a positive constant $\beta$ such that the following inf-sup condition for $b(\cdot,\cdot)$ holds (see \cite[Theorem 2.1]{MR2594823})
\begin{equation*}
\displaystyle\sup_{\boldsymbol{0}\neq\btau\in \mathbb{H}_{0}}\frac{b(\btau,\bv)}{\|\btau\|_{\bdiv,\O}}\geq\beta\|\bv\|_{0,\O}\quad\forall\bv\in \mathbf{Q}^{\bu}.
\end{equation*}

With this results at hand, we are in position to introduce the solution operator
\begin{align*}
\bT:\mathbf{Q}^{\bu}&\rightarrow \mathbf{Q}^{\bu},\\
           \boldsymbol{f}&\mapsto \bT\boldsymbol{f}:=\widehat{\bu}, 
\end{align*}
where the pair $(\widehat{\bsig}, \widehat{\bu})$ is the solution of the following source problem
\begin{align}
\label{eq1_source}a_0(\widehat{\bsig}, \btau)+b(\btau,\widehat{\bu})&=0\,\,\,\,\quad\quad\quad\forall \btau\in \mathbb{H}_0,\\
\label{eq2_source}b(\widehat{\bsig},\bv)&=-(\boldsymbol{f},\bv)\quad\forall\bv\in \mathbf{Q}^{\bu},
\end{align}

Notice that $\bT$ is well defined due the Babu\^{s}ka-Brezzi theory and we have the follows estimate
\begin{equation*}
\|\widehat{\bsig}\|_{\bdiv,\O}+\|\widehat{\bu}\|_{0,\O} \lesssim\|\boldsymbol{f}\|_{0,\O}.
\end{equation*}
Moreover, it is easy to check that  $\bT$ is self-adjoint respect the $[\L^2]^{n}$- inner product. 
Indeed, given
$\boldsymbol{f},\widehat{\boldsymbol{f}}\in\mathbf{Q}^{\bu}$, let $(\widehat{\bsig},\widehat{\bu})\in\mathbb{H}_{0}\times\mathbf{Q}^{\bu}$ and $(\widetilde{\bsig},\widetilde{\bu})\in\mathbb{H}_{0}\times\mathbf{Q}^{\bu}$  be the solutions to problem \eqref{eq1_source}--\eqref{eq2_source} with right hand sides $\boldsymbol{f}$ and $\widehat{\boldsymbol{f}}$, respectively. Assume that that $\bT \boldsymbol{f}=\widehat{\bu}$ and $\bT \widehat{\boldsymbol{f}}=\widetilde{\bu}$. The symmetry of $a(\cdot,\cdot)$ and $(\cdot,\cdot)_{0,\O}$ implies that
\begin{equation*}
(\boldsymbol{f},\bT \widetilde{\boldsymbol{f}})_{0,\O}=(\boldsymbol{f},\widetilde{\bu})_{0,\O}=-\big(a(\widehat{\bsig},\widetilde{\bu})+b(\widetilde{\bu}, \widehat{\bsig})+b(\widehat{\bsig},\widetilde{\bu}) \big)=
(\widetilde{\boldsymbol{f}}, \widehat{\bu})_{0,\O}=(\bT \boldsymbol{f}, \widetilde{\boldsymbol{f}})_{0,\O}.
\end{equation*}

We observe that $(\lambda, (\bsig,\bu))\in\mathbb{R}\times \mathbb{H}_0\times \mathbf{Q}^{\bu}$ solves \eqref{eq1_reduced}--\eqref{eq2_reduced} if and only if $(\kappa,\bu)$ is an eigenpair of $\bT$, i.e.
\begin{equation*}
\ds \bT\bu=\kappa\bu\quad\text{with}\,\,\kappa:=\frac{1}{\lambda}.
\end{equation*}

The next step in our analysis is to obtain an additional regularity result for our eigenfunctions. To do this task, we consider the following problem: given $\boldsymbol{f}\in \L^2(\O)^n$, let $(\widetilde{\bu},\widetilde{\boldsymbol{\sigma}})\in [\H^1(\Omega)]^n\times \mathbb{H}(\div;\Omega)$ be the solution
of the following problem
\begin{equation*}
\left\{
\begin{array}{rcll}
-\bdiv \widetilde{\bsig} & = & \boldsymbol{f}&  \text{ in } \quad \Omega, \\
\dfrac{1}{2\mu}\widetilde{\bsig}^{d} & = &\nabla\widetilde{\bu} & \text{ in } \quad \Omega, \\
\widetilde{\bsig} & = & \mathbf{0} & \text{ on } \quad \partial\Omega,\\
\widetilde{\bu} & = & \mathbf{0} & \text{ on } \quad \partial\Omega.
\end{array}
\right.
\end{equation*}

 Now, using the relation between incompressible elasticity and the Stokes problem and according to  \cite{MR851383}  we conclude that: there exists $s\in (0,1)$  such that $\widehat{\bu}\in[\H^{1+s}(\Omega)]^n$ and
\begin{equation}
\label{eq:reg_u_if}
||\widehat{\bsig}\|_{s,\O}+\|\widehat{\bu}\|_{1+s,\Omega}\lesssim\|\boldsymbol{f}\|_{0,\Omega}.
\end{equation}

Hence, the compactness of $\bT$ is a direct consequence of the previous regularity result.
Finally, we have the following spectral characterization.

\begin{theorem}
\label{thrm:spec_char_T}
The spectrum of $\bT$ satisfies $\sp(\bT)=\{0\}\cup\{\mu_k\}_{k\in\mathbb{N}}$, where \\$\{\mu_k\}_{k\in\mathbb{N}}\in (0,1)$
is a sequence of real positive eigenvalues which converges to zero, repeated according their respective multiplicities.  In addition, the following additional regularity result holds true for eigenfunctions
\begin{equation*}
\label{eq:reg_u_infty}
||\bsig\|_{s,\O}+\|\bu\|_{1+s,\Omega}\lesssim\|\bu\|_{0,\Omega}.
\end{equation*}
\end{theorem}

We have from Remark \ref{remark_1} that problems  \eqref{eq1}--\eqref{eq2} and  \eqref{eq1_reduced}--\eqref{eq2_reduced} are equivalent. However, the finite element discretizations for these problems are not equivalent (cf. Section \ref{sec:mixed_fem}).   Hence, to obtain error estimates for our
methods, we need to consider an additional  solution operator associated with the problem \eqref{eq1}--\eqref{eq2}. 

Now, let $\widetilde{\bT}$ be the solution operator defined by
 \begin{align*}
\widetilde{\bT}:\mathbf{Q}^{\bu}&\rightarrow \mathbf{Q}^{\bu},\\
           \boldsymbol{f}&\mapsto \widetilde{\bT}\widetilde{\boldsymbol{f}}:=\widetilde{\bu}, 
\end{align*}
where $(\widetilde{\bsig}, \widetilde{p},\widetilde{\bu})$ is the solution of the following source problem
\begin{align*}
a((\widetilde{\bsig},\widetilde{p}),(\btau,q))+b(\btau,\widetilde{\bu})&=0\,\,\,\,\quad\quad\quad\forall (\btau,q)\in\mathbb{H},\\
b(\widetilde{\bsig},\bv)&=-(\widetilde{\boldsymbol{f}},\bv)\quad\forall\bv\in \mathbf{Q}^{\bu}.
\end{align*}
Thanks to the Remark \ref{remark_1} and the fact that the operator $\bT$ is well-defined, we have that also $\widetilde{\bT}$ is well defined (see \cite[Theorem 2.2]{MR2594823}) and there holds
\begin{equation*}
\|\widetilde{\bsig}\|_{\bdiv,\O}+\|\widetilde{p}\|_{0,\O}+\|\widetilde{\bu}\|_{0,\O}\lesssim\|\widetilde{\boldsymbol{f}}\|_{0,\O}.
\end{equation*}

Moreover, it is easy to check that  $\widetilde{\bT}$ is self-adjoint respect the $[\L^2(\O)]^n$- inner product and we observe that $(\lambda, (\bsig,p,\bu))\in\mathbb{R}\times \mathbb{H}\times \mathbf{Q}^{\bu}$ solves \eqref{eq1}--\eqref{eq2} if and only if $(\kappa,\bu)$ is an eigenpair of $\widetilde{\bT}$. Hence, $\widetilde{\bT}$ is compact and $\sp(\widetilde{\bT})=\sp(\bT)=\{0\}\cup\{\mu_k\}_{k\in\mathbb{N}}$

\section{The mixed finite element method}
\label{sec:mixed_fem}
The present section deals with the finite element approximation for the eigenvalue 
problem.  To do this task, we begin by introducing a regular family of triangulations of $\O$ denoted by $\{\CT_h\}_{h>0}$. Let $h_T$ the diameter of a triangle $T$ of the triangulation and let us define $h:=\max\{h_T\,:\, T\in \CT_h\}$.

Given an integer $\ell\geq 0$ and a subset $D$ of $\mathbb{R}^n$, we denote by $\mathbb{P}_\ell(S)$ the space of polynomials of degree at most $\ell$ defined in $D$.

\subsection{The finite element spaces}
In our study, we consider two numerical schemes that only differ in the space that approximates the pseudostress. Hence, we 
only refer to $\mathbb{H}_h^{\bsig}$ to the finite element space related to the approximation of $\bsig$. For the velocity field
we consider the space $\mathbf{Q}_h^{\bu}$ and for the pressure $Q_h^p$. In what follows we specify each of these 
finite dimensional spaces.

For $k\geq 0$ we define the local Raviart-Thomas space of order $k$
 as follows  (see \cite{MR3097958})
 \begin{equation*}
 \mathbb{RT}_k(T)=[\mathbb{P}_k(T)]^n\oplus\mathbb{P}_k(T)\boldsymbol{x},
 \end{equation*}
 where if $\texttt{t}$ denotes the transpose operator,  $\boldsymbol{x}^{\texttt{t}}$ represents a generic vector of $\mathbb{R}^n$. Hence, the global Raviart-Thomas is defined by
 \begin{equation*}
 \mathbb{RT}_k(\CT_h):=\{\btau\in\mathbb{H}(\bdiv,\O)\,:\, \btau|_T^{\texttt{t}}\in\mathbb{RT}_k(T),\,\,\forall T\in\CT_h\}.
 \end{equation*}
 
 More precisely, in the definition above $\btau|_T^{\texttt{t}}$ must be understood as $(\tau_{i1},\tau_{i2})^{\texttt{t}}\in\mathbb{RT}_k(T)$ for all $i\in\{1,2\}$ when $n=2$, and $(\tau_{j1},\tau_{j2}, \tau_{j3})^{\texttt{t}}\in\mathbb{RT}_k(T)$ for all $j\in\{1,2,3\}$ when $n=3$.
  
 On the other hand, we define the space of piecewise polynomials of degree at most $k$
 \begin{equation*}
 \mathbb{P}_k(\CT_h):=\{v\in\L^2(\O)\,:\, v|_T\in\mathbb{P}_k(T)\,\,\forall T\in\CT_h\}.
 \end{equation*}
 In addition, we  introduce the  Brezzi-Douglas-Marini finite element space \cite{MR799685},
 \begin{equation*}
  \mathbb{BDM}_{k}:=[\mathbb{P}_k(\CT_h)]^{n}\cap\mathbb{H} \text{ with } k\geq 1.
 \end{equation*}
 It is well known from the literature that  $\mathbb{RT}_{k-1}\subset \mathbb{BDM}_{k}\subset\mathbb{RT}_k$ for all $k\geq 1$ (see \cite[Section 2]{MR3097958}).
Moreover,  the number of degrees of freedom per edge is the same for both finite elements, however, the number of internal degrees of freedom of  Brezzi-Douglas-Marini ($\mathbb{BDM}_{k}$) elements is  less than that of standard finite elements of the same order such as Raviart-Thomas ($\mathbb{RT}_k$) . A count of the internal degrees of freedom for $n=2$ gives 
 \begin{equation*}
\mathbb{BDM}_k:2(k-1)(k+1).\qquad \mathbb{RT}_k:2k(k+1),
 \end{equation*}
 and for $n=3$ 
 \begin{equation*}
\mathbb{BDM}_k:\dfrac{3}{2}(k-1)(k+1)(k+2).\qquad \mathbb{RT}_k:\dfrac{3}{2}k(k+1)(k+2),
 \end{equation*}

\subsection{Approximation errors} 
In the following, some approximation results for discrete spaces are presented. To make matters precise, since we consider two spaces to approximate the pseudostress tensor, we need to introduce suitable interpolators for each finite element space, namely, Raviart-Thomas and BDM spaces. We begin with the classical approximation property for piecewise polynomials   (see \cite{MR3507271}). Let $\mathcal R_h:[\LO]^n\rightarrow [\mathbb{P}_{k}(\CT_h)]^n$. The following estimate is true.
\begin{equation}
\label{eq:errorl2.1}
\|\bv-\mathcal R_h\bv\|_{0,\O}\lesssim h^{\min\{t,k+1\}}\|\bv\|_{t,\O}\qquad\forall t\in[\H^t(\O)]^{n}\cap[\LO]^n.
 \end{equation}
%

For the Raviart-Thomas spaces, we have the following   approximation results:
let $\Pi_h^{\mathbb{RT}}: [\mathbb{H}^t(\O)]^{\nxn} \to\mathbb{RT}_k$ be the
tensorial version of the Raviart-Thomas  interpolation operator , 
which satisfies the following classical error estimate,
see \cite{MR1115205,MR1115239},
\begin{equation}\label{asympRT}
 \norm{\btau - \Pi_h^{\mathbb{RT}} \btau}_{0,\O} \lesssim h^{\min\{t, k+1\}} \norm{\btau}_{t,\O} \qquad \forall \btau \in [\mathbb{H}^t(\O)]^{\nxn}, \quad t\geq 1.
\end{equation}
Also,  thanks to the commutative diagram, if $\div\btau\in[\mathbf{H}^r(\O)]^{n}$ with $r\geq 0$ we have the following result 
\begin{equation}\label{asympdivRT}
 \norm{\bdiv (\btau - \Pi_h^{\mathbb{RT}} \btau) }_{0,\O}  
 \lesssim h^{\min\{r, k+1\}} \norm{\bdiv\btau}_{r,\O}.
\end{equation}
Moreover, $\Pi_h^{\mathbb{RT}}$ can also be defined as  $\Pi_h^{\mathbb{RT}}: [\mathbb{H}^t(\O)]^{\nxn}\cap \mathbb{H}(\bdiv,\O) \to\mathbb{RT}_k$ for all $t\in(0,1]$, and  we have the following estimate
\begin{equation}\label{asymp00RT}
 \norm{\btau - \Pi_h^{\mathbb{RT}} \btau}_{0,\O} \lesssim h^t (\norm{\btau}_{t,\O}
 + \norm{\bdiv \btau}_{0,\O})  \quad \forall \btau \in [\mathbb{H}^t(\O)]^{\nxn}\cap \mathbb{H}(\bdiv,\O)\quad t\in (0, 1].
\end{equation}

For the BDM spaces, we have the following properties: let $\ell\geq 1$ and let $\Pi_h^{\mathbb{BDM}}: [\mathbb{H}^t(\O)]^{\nxn} \to\mathbb{BDM}_\ell$ be the
tensorial version of the BDM-interpolation operator , 
which satisfies the following classical error estimate,
see \cite[Theorem 3.16]{MR2009375},
\begin{equation}\label{asymp0}
 \norm{\btau - \Pi_h^{\mathbb{BDM}}\btau}_{0,\O} \lesssim h^{\min\{t, \ell+1\}} \norm{\btau}_{t,\O} \qquad \forall \btau \in [\mathbb{H}^t(\O)]^{\nxn}, \quad t>1/2.
\end{equation}
Also,  for less regular tensorial fields we have the following estimate
\begin{equation}\label{asymp00}
 \norm{\btau - \Pi_h^{\mathbb{BDM}}\btau}_{0,\O} \lesssim h^t (\norm{\btau}_{t,\O}
 + \norm{\btau}_{\bdiv,\O})  \;\ \forall \btau \in [\mathbb{H}^t(\O)]^{\nxn}\cap \mathbb{H}(\bdiv,\O), \;\ t\in (0, 1/2].
\end{equation}

Moreover, the following commuting diagram property holds true:
\begin{equation}\label{asympdiv}
 \norm{\bdiv (\btau - \Pi_h^{\mathbb{BDM}}\btau) }_{0,\O} = \norm{\bdiv \btau - \mathcal R_h \bdiv \btau }_{0,\O} 
 \lesssim h^{\min\{t, \ell\}} \norm{\bdiv\btau}_{t,\O},
\end{equation}
for  $\bdiv \btau \in [\H^t(\O)]^{n}$ and  $\mathcal R_h$
being  the $[\LO]^n$-orthogonal projection onto \\$[\mathbb{P}_{\ell-1}(\CT_h)]^n$.

We conclude this section by introducing the following  notations
\begin{equation*}
\mathbb{H}^{\bsig}_{0,h}:=\left\{\btau\in\bar{\mathbb{H}}_h\,:\,\,\int_{\Omega}\tr(\btau)=0 \right\},
\end{equation*}
where $\bar{\mathbb{H}}_h\in\{\mathbb{RT}_k,\mathbb{BDM}_{k+1}\}$. Also, we define $Q_h^p:=\mathbb{P}_k(\CT_h)$,  $\mathbf{Q}_h^{\bu}:=[\mathbb{P}_k(\CT_h)]^n.$ and $\mathbb{H}_h:=\mathbb{H}^{\bsig}_{0,h}\times Q_h^p $.

Therefore, as a consequence of \eqref{eq:errorl2.1}--\eqref{asympdiv}, we have 
the following approximation properties for $k\geq 0$:
For each $t>0$ and for each $\btau\in\mathbb{H}^t(\O)\cap \mathbb{H}_{0}$ with $\div\btau\in [\H^{t}(\O)]^{n}$ there exists  $\btau_{h}\in\mathbb{H}_{0}^{\sigma}$ such that
\begin{equation}
\label{eq:H0}
\|\btau-\btau_{h}\|_{\div,\O} \lesssim h^{\min\{t,k+1\}}\left(\|\btau\|_{t,\O}+\|\div\btau\|_{t,\O}\right).
\end{equation}
For $q\in Q^p $ there exists $q_{h}\in Q_h^p$ such that 
\begin{equation}
\label{eq:Qp}
\|q-q_{h}\|_{0,\O} \lesssim h^{\min\{t,k+1\}}\|q\|_{t,\O}.
\end{equation}
For $\bv\in [\H^{t}(\O)]^{n}$ there exists $\bv_{h}\in\mathbf{Q}_h^{\bu}$ such that
\begin{equation}
\label{eq:Qbu}
\|\bv-\bv_{h}\|_{0,\O} \lesssim h^{\min\{t,k+1\}}\|\bv\|_{t,\O}.
\end{equation}

\subsection{The discrete eigenvalue problems}
As we claim in Section \ref{sec:model_problem}, discrete counterparts of problems \eqref{eq1}--\eqref{eq2} and \eqref{eq1_reduced}--\eqref{eq2_reduced} are not equivalent (see\cite[Lemma 3.1]{MR2594823} for further details ). Hence, we need to analyze each discrete eigenvalue problem by separated.

With the discrete spaces defined above, we are in position to introduce the discretization of 
problem \eqref{eq1}--\eqref{eq2}: Find $\lambda_h\in\mathbb{R}$ and $\boldsymbol{0}\neq(\bsig_h, p_h,\bu_h)\in\mathbb{H}_h\times\mathbf{Q}_h^{\bu}$ such that
\begin{align}
\label{eq1h}a((\bsig_h,p_h),(\btau_h,q_h))+b(\btau_h,\bu_h)&=0\,\,\,\,\quad\quad\quad\forall (\btau_h,q_h)\in\mathbb{H}_h,\\
\label{eq2h}b(\bsig_h,\bv_h)&=-\lambda_h(\bu_h,\bv_h)\quad\forall\bv_h\in \mathbf{Q}_h^{\bu}.
\end{align}

Similarly as in the continuous case, it is possible to consider a reduced formulation for the discrete eigenvalue problem. These reduced discrete problem reads as follows: find $\lambda_h\in\mathbb{R}$ and $\boldsymbol{0}\neq (\bsig_h,\bu_h)\in \mathbb{H}^{\bsig}_{0,h}\times \mathbf{Q}_h^{\bu}$ such that
\begin{align}
\label{reduced_disc_source1}a_0(\bsig_h,\btau_h)+b(\btau_h,\bu_h)&=0\,\,\,\,\quad\quad\quad\forall \btau_h\in \mathbb{H}^{\bsig}_{0,h},\\
\label{reduced_disc_source2}b(\bsig_h,\bv_h)&=-\lambda(\bu_h,\bv_h)\quad\forall\bv_h\in \mathbf{Q}_h^{\bu}.
\end{align}

It has been proved in \cite[Lemma 3.2]{MR2594823} that there exists a positive constant $\beta$, independent of $h$, such that the following inf-sup condition holds
\begin{equation*}
\displaystyle\sup_{\boldsymbol{0}\neq\btau_h\in \mathbb{H}_{0,h}^{\bsig}}\frac{b(\btau_h,\bv_h)}{\|\btau_h\|_{\bdiv,\O}}\geq\beta\|\bv_h\|_{0,\O}\quad\forall\bv_h\in \mathbf{Q}_h^{\bu}.
\end{equation*}

On the other hand, the discrete kernel of $b(\cdot,\cdot)$ (namely, the kernel of the operator induced by $b(\cdot,\cdot)$) is defined by
\begin{equation*}
\mV_h:=\{\btau\in \mathbb{H}_{0,h}^{\sigma}\,:\, b(\btau,\bv)=\boldsymbol{0}\,\,\forall\bv\in\mathbf{Q}_h^{\bu}\}=\{\btau\in \mathbb{H}_{0,h}^{\sigma}\,:\,\bdiv\btau=\boldsymbol{0}\,\,\,\text{in}\,\,\O\}.
\end{equation*}
In \cite[Theorem 3.1]{MR2594823} the authors have stated that $a_0(\cdot,\cdot)$ is coercive in $\mV_h$ and that $b(\cdot,\cdot)$ satisfies the corresponding discrete inf-sup condition.

With these ingredients at hand, we are in position to introduce the discrete solution operator associated to \eqref{reduced_disc_source1}-- \eqref{reduced_disc_source2}
\begin{align*}
\bT_h:\mathbf{Q}^{\bu}&\rightarrow \mathbf{Q}^{\bu}_h,\\
           \boldsymbol{f}&\mapsto \bT_h\boldsymbol{f}:=\widehat{\bu}_h, 
\end{align*}
where $(\widehat{\bsig}_h,\widehat{\bu}_h)$ is the solution of the following source problem
\begin{align*}
a_{0}(\widehat{\bsig}_h,\btau_h)+b(\btau_h,\widehat{\bu}_h)&=0\,\,\,\,\quad\quad\quad\forall \btau_h\in\mathbb{H}_{0,h}^{\sigma}\\
b(\widehat{\bsig}_h,\bv_h)&=-(\boldsymbol{f},\bv_h)\quad\forall\bv_h\in \mathbf{Q}_h^{\bu},
\end{align*}
which according to the Babu\^{s}ka-Brezzi theory, is well posed (see \cite{MR3097958}) and the following estimate holds.
\begin{equation*}
\|\widehat{\bsig}_{h}\|_{\bdiv,\O}+\|\widehat{\bu}_{h}\|_{0,\O} \lesssim\|\boldsymbol{f}\|_{0,\O},
\end{equation*}
where the hidden constant is independent of $h$.

As presented in \cite[Lemma 3.1]{MR2594823}, a necessary condition for discrete problem \eqref{eq1h}--\eqref{eq2h} and problem \eqref{reduced_disc_source1}---\eqref{reduced_disc_source2} to be equivalent, is that 
$\tr(\mathbb{H}_{0,h}^{\sigma})\subset Q_h^p$ and since in this case, this condition does not hold,  we need to define the following discrete solution operator $\widetilde{\bT}_h$ associated with the problem \eqref{eq1h}--\eqref{eq2h}.
\begin{align*}
\widetilde{\bT}_h:\mathbf{Q}^{\bu}&\rightarrow\mathbf{Q}_h^{\bu}\\
           \widetilde{\boldsymbol{f}}&\mapsto \widetilde{\bT}_h\widetilde{\boldsymbol{f}}:=\widetilde{\bu}_h, 
\end{align*}
where the triplet $(\widetilde{\bsig}_h, \widetilde{p}_h, \widetilde{\bu}_h)$ is the solution of the following source problem
\begin{align*}
a((\widetilde{\bsig}_h,\widetilde{p}_h),(\btau_h,q_h))+b(\btau_h,\widetilde{\bu}_h)&=0\,\,\,\,\quad\quad\quad\forall (\btau_h,q_h)\in\mathbb{H}_h,\\
b(\widetilde{\bsig}_h,\bv_h)&=-(\widetilde{\boldsymbol{f}},\bv_h)\quad\forall\bv_h\in \mathbf{Q}_h^{\bu}.
\end{align*}
Observe that \cite[Theorem 3.3]{MR2594823} guarantees that $\widetilde{\bT}_h$ is well-defined and 
\begin{equation*}
\|\widetilde{\bsig}_{h}\|_{\bdiv,\O}+\|\widetilde{p}\|_{0,\O}+\|\widetilde{\bu}_{h}\|_{0,\O}\lesssim\|\widetilde{\boldsymbol{f}}\|_{0,\O},
\end{equation*}
where the hidden constant is independent of $h$.

We are in position to establish the following approximation result

\begin{lemma}
\label{lm:app}
Let $\boldsymbol{f}\in\mathbf{Q}^{\bu}$. The following best approximation estimates hold
\begin{equation*}
\|(\bT-\bT_h)\boldsymbol{f}\|_{0,\O}\lesssim \inf_{\btau_{h}\in\mathbb{H}^{\bsig}_{0,h}}\|\widehat{\bsig}-\btau_{h}\|_{\bdiv,\O}+\inf_{\bv_{h}\in\mathbf{Q}^{\bu}_{h}}\|\widehat{\bu}-\bv_{h}\|_{0,\O},
\end{equation*}
and 
\begin{equation*}
\|(\widetilde{\bT}-\widetilde{\bT}_h)\boldsymbol{f}\|_{0,\O}\lesssim \inf_{\btau_{h}\in\mathbb{H}^{\bsig}_{0,h}}\|\widetilde{\bsig}-\btau_{h}\|_{\bdiv,\O}+\inf_{q_{h}\in Q^{p}_{h}}\|\widetilde{p}-q_{h}\|_{0,\O}+\inf_{\bv_{h}\in\mathbf{Q}^{\bu}_{h}}\|\widetilde{\bu}-\bv_{h}\|_{0,\O},
\end{equation*}
where the hidden constant is independent of $h$. 
\end{lemma}
\begin{proof}
Let $\boldsymbol{f}\in\mathbf{Q}^{\bu}$ be such that  $\bT\boldsymbol{f}=\widehat{\bu}$ and $\bT_h\boldsymbol{f}=\widehat{\bu}_h$ where $\widehat{\bu}$ is the solution of \eqref{eq1_source}--\eqref{eq2_source} and $\widehat{\bu}_h$ is the solution of \eqref{reduced_disc_source1}--\eqref{reduced_disc_source2}.  We remark that $\widehat{\bu}_h$ is the finite element  approximation of $\widehat{\bu}$ through the scheme  $\mathbb{H}_{0,h}^{\sigma}\times \mathbf{Q}^{\bu}_{h}$.

 Hence, applying \cite[Theorem 3.1]{ MR2594823} we have immediately that
\begin{align*}
\|(\bT-\bT_h)\boldsymbol{f}\|_{0,\O}&=\|\widehat{\bu}-\widehat{\bu}_{h}\|_{0,\O}\lesssim \inf_{\btau_{h}\in\mathbb{H}^{\bsig}_{0,h}}\|\widehat{\bsig}-\btau_{h}\|_{\bdiv,\O}+\inf_{\bv_{h}\in\mathbf{Q}^{\bu}_{h}}\|\widehat{\bu}-\bv_{h}\|_{0,\O},
\end{align*}
where the hidden constant is independent of $h$. 

For the approximation error $\|(\widetilde{\bT}-\widetilde{\bT}_h)\boldsymbol{f}\|_{0,\O}$ the derivation is similar as the previous estimate. This concludes the proof.
\end{proof}

We remark that Lemma \ref{lm:app} is a general result where the choice of the finite element scheme has not 
influence. If we are more specific in the numerical scheme, the lemma above becomes into an error estimate for each scheme. 

Since we are dealing with two numerical schemes and two discrete eigenvalue problems, as corollaries,  we 
derived the following results. The first corresponds to the approximation error between $\bT$ and $\bT_h$.
\begin{corollary}[Approximation between $\bT$ and $\bT_h$]
\label{cor:app_TT1}
Let $\boldsymbol{f}\in\mathbf{Q}^{\bu}$. If the approximation scheme $[\mathbb{P}_k]^{n}\text{-}\mathbb{RT}_k$  is considered, then there holds
\begin{equation*}
\|(\bT-\bT_h)\boldsymbol{f}\|_{0,\O}\lesssim h^{s}\|\boldsymbol{f}\|_{0,\O}.
\end{equation*}
Otherwise, if the scheme is $[\mathbb{P}]^{n}_k\text{-}\mathbb{BDM}_{k+1}$, there holds
\begin{equation*}
\|(\bT-\bT_h)\boldsymbol{f}\|_{0,\O}\lesssim h^{s}\|\boldsymbol{f}\|_{0,\O}.
\end{equation*}
where, in each estimate, the hidden constant is independent of $h$.
\end{corollary}
\begin{proof}
The proof  follows from \eqref{eq:reg_u_if}, the first estimate of Lemma \ref{lm:app},   the approximation properties \eqref{eq:H0} and \eqref{eq:Qbu} 
\end{proof}
Now we present the analogous of Corollary \ref{cor:app_TT1}, but for the error between $\widetilde{\bT}$ and $\widetilde{\bT}_h$. The proof is follows the same arguments of corollary above, so we skip the details.
\begin{corollary}[Approximation between $\widetilde{\bT}$ and $\widetilde{\bT}_h$]
Let $\boldsymbol{f}\in\mathbf{Q}^{\bu}$. If the approximation scheme $[\mathbb{P}_k]^{n}\text{-}\mathbb{P}_k\text{-}\mathbb{RT}_k$  is considered, then there holds
\begin{equation*}
\|(\widetilde{\bT}-\widetilde{\bT}_h)\boldsymbol{f}\|_{0,\O}\lesssim h^{s}\|\boldsymbol{f}\|_{0,\O}.
\end{equation*}
Otherwise, if the scheme is $[\mathbb{P}]^{n}_k\text{-}\mathbb{P}_k\text{-}\mathbb{BDM}_{k+1}$, there holds
\begin{equation*}
\|(\widetilde{\bT}-\widetilde{\bT}_h)\boldsymbol{f}\|_{0,\O}\lesssim h^{s}\|\boldsymbol{f}\|_{0,\O}.
\end{equation*}
where, in each estimate, the hidden constant is independent of $h$.
\end{corollary}

\section{Convergence and Error estimates}
\label{sec:conv}
In this section we will analyze the convergence of the mixed method and derive error estimates for the eigenvalues and eigenfunctions.
We remark that, for both of the numerical schemes considered in our paper, these results are valid. Hence,  to simplify the presentation  of our results, we will prove our results in a reduced formulation.
 Due the compactness of
$\bT$ and $\widetilde{\bT}$, the convergence of the eigenvalues is obtained by means of the classic theory.  From now on we concentrate on the case associated with the reduced formulation.  The following result is a consequence of the convergence in norm between $\bT$ and $\bT_h$, and states that the method does not introduce spurious eigenvalues.

\begin{theorem}
\label{thm:spurious_free}
Let $V\subset\mathbb{C}$ be an open set containing $\sp(\bT)$. Then, there exists $h_0>0$ such that $\sp(\bT_h)\subset V$ for all $h<h_0$.
\end{theorem}

 We recall the definition of the resolvent operator of $\bT$ and $\bT_h$ respectively:
\begin{gather*}
	(z\bI-\bT)^{-1}\,:\, \mathbf{Q}^{\bu} \to \mathbf{Q}^{\bu}\,, \quad z\in\mathbb{C}\setminus \sp(\bT), \\
		(z\bI-\bT_h)^{-1}\,:\, \mathbf{Q}_h^{\bu} \to \mathbf{Q}_h^{\bu}\,, \quad z\in\mathbb{C}\setminus\sp(\bT_h) .
\end{gather*}

 As a consequence of Corollary \ref{cor:app_TT1}, if $\kappa \in (0,1)$ is an isolated eigenvalue of $\bT$ with multiplicity $m$, and $\mathcal{E}$ its associated eigenspace, then, there exist $m$
eigenvalues $\kappa_{h}^{(1)},...,\kappa_{h}^{(m)}$ of $\bT_{h}$, repeated according to their respective multiplicities, which converge to $\kappa$. Let $\mathcal{E}_{h}$ be the direct sum of their corresponding associated eigenspaces (see \cite{MR0203473}).

We recall the definition of the \textit{gap} $\hdel$ between two closed
subspaces $\CM$ and $\CN$ of $\L^2(\O)$:
$$
\hdel(\CM,\CN)
:=\max\big\{\delta(\CM,\CN),\delta(\CN,\CM)\big\},
$$
where
$$
\delta(\CM,\CN)
:=\sup_{x\in\CM:\ \left\|x\right\|_{0,\O}=1}
\left(\inf_{y\in\CN}\left\|x-y\right\|_{0,\O}\right).
$$

\begin{theorem}
\label{millar2015}
There exists strictly positive constant C, such that 
\begin{equation*}
	\widehat{\delta} ( \mathcal{E}, \mathcal{E}_{h} )\lesssim\,h^{ \min\{ s,k+1 \} } \quad \mbox{and} \quad | \mu-\mu_{h}(i) | \lesssim\,h^{ \min\{s,k+1 \} }.
\end{equation*}
\begin{proof}
	 As consequence of Corollary \ref{cor:app_TT1}, $\bT_{h}$  converges in norm to $\bT$ as $h$ goes to zero. Then, the proof follows as a direct consequence of  
	\cite[Theorem 7.1 and Theorem 7.3]{MR1115235} and  using the 
regularity from Theorem \ref{thrm:spec_char_T}.
\end{proof}
\end{theorem}

The next result provides a double order of convergence for the eigenvalues. 
\begin{theorem}
There exists a strictly positive constant $h_0$ such that, for $h<h_0$ there holds
\begin{equation*}
|\lambda-\lambda_h|\lesssim h^{2\min\{s,k+1\}},
\end{equation*}
where the hidden constant is independent of $h$.
\end{theorem}
\begin{proof}
Let $(\lambda,\bsig,\bu)$ solution of problem \eqref{eq1_source}--\eqref{eq2_source}  and let $(\lambda_{h},\bsig_{h},\bu_{h})$ be solution of problem \eqref{reduced_disc_source1}--\eqref{reduced_disc_source2} with $\|\bu_{h}\|_{0,\O}=1$. For simplicity, we define 
$$\mathbf{U}:=(\bsig,\bu),\,\,\, \mathbf{U}_h:=(\bsig_h,\bu_h),\,\,\,
\mathbf{V}:=(\btau,\bv),\,\,\,\mathbf{V}_h:=(\btau_h,\bv_h).$$
Then, we rewrite  problems \eqref{eq1_source}--\eqref{eq2_source} and   \eqref{reduced_disc_source1}--\eqref{reduced_disc_source2} as follows: Find  $(\lambda,\mathbf{U})$ and $(\lambda_h,\mathbf{U}_h)$  solutions of the following eigenvalue problems
\begin{equation*}
A(\mathbf{U},\mathbf{V})=\lambda(\bu,\bv),\quad
A(\mathbf{U}_h,\mathbf{V}_h)=\lambda_h(\bu,\bv_h),
\end{equation*}
where $A(\cdot,\cdot)$ is the symmetric and bounded bilinear forms defined by 
$$A(\mathbf{U},\mathbf{V})=A((\bsig,\bu),(\btau,\bv)):=a_{0}(\bsig,\btau)+b(\btau,\bu)+b(\bsig,\bu).$$
Moreover, we have the following classic identity
	\begin{equation*}
		(\lambda-\lambda_{h})(\bu_{h}, \bu_{h})=A(\boldsymbol{U}-\boldsymbol{U}_{h}, \boldsymbol{U}-\boldsymbol{U}_{h})+\lambda\, (\bu-\bu_{h}, \bu-\bu_{h}).
	\end{equation*}
The proof is completed by taking absolute value on both sides of the above equation, the triangular inequality, and the approximation properties \eqref{eq:H0}--\eqref{eq:Qbu}, together with the additional regularity provided by Theorem \ref{thrm:spec_char_T}.
%

\end{proof}
\section{Numerical experiments}
\label{sec:numerics}

In this section we report some numerical tests in order to assess the performance of the proposed mixed element method, in the computation of the eigenvalues of problem \eqref{eq1h}--\eqref{eq2h}. In all our experiments we consider the boundary condition $\bu=0$ and $\mu=1/2.$ 

We have implemented the discrete eigenvalue problem in a FEniCS code and the orders of convergence have been computed with a least-square fitting. 

The schemes are performed in different domains as bidimensional convex and non convex domains and a three dimensional domain. For all the geometric configurations we compute the lowest eigenvalues and convergence orders. For the two dimensional domains we prove the schemes with polynomials degrees $k=0,1,2$ and for the 3-D domain only for $k=0$ due to the 
machine memory. With the computed results at hand, we compare the schemes that only differ on the $\H(\div)$ finite element space. In particular, in the first test we show that the reduced scheme \eqref{reduced_disc_source1}--\eqref{reduced_disc_source2} gives the same numerical results as \eqref{eq1h}--\eqref{eq2h}. This allows to perform the rest of the tests by choosing only one of them, that  in our case is \eqref{eq1h}--\eqref{eq2h}.  

In each test we also report plots of the associated eigenfunctions, in particular the velocity fields and pressure fluctuations. Moreover, in several experiments we consider the relative errors $e_{\lambda_i}\,i=1,2,3,4$ for different choices of $k$, where
$$
e_{\lambda_i}:=\frac{\vert \lambda_{h_i}-\lambda_{extr_i}\vert}{\vert \lambda_{extr_i}\vert}.
$$

Finally, we  denote by $e_{\lambda_i}(\mathbb{RT})$ and $e_{\lambda_i}(\mathbb{BDM})$ the relative errors obtained using $[\mathbb{P}_k]^{n}\text{-}\mathbb{P}_k\text{-}\mathbb{RT}_k$   and $[\mathbb{P}_k]^{n}\text{-}\mathbb{P}_k\text{-}\mathbb{BDM}_{k+1}$   schemes, respectively.
\subsection{Test 1: Square}

In this test we consider as computational domain the square $\O_S:=(-1,1)^2$ and the meshes for the following tests are  like the presented in Figure \ref{meshes_square}.
\begin{figure}[H]
	\begin{center}
		\begin{minipage}{10cm}
			\centering\includegraphics[height=4.6cm, width=3.3cm]{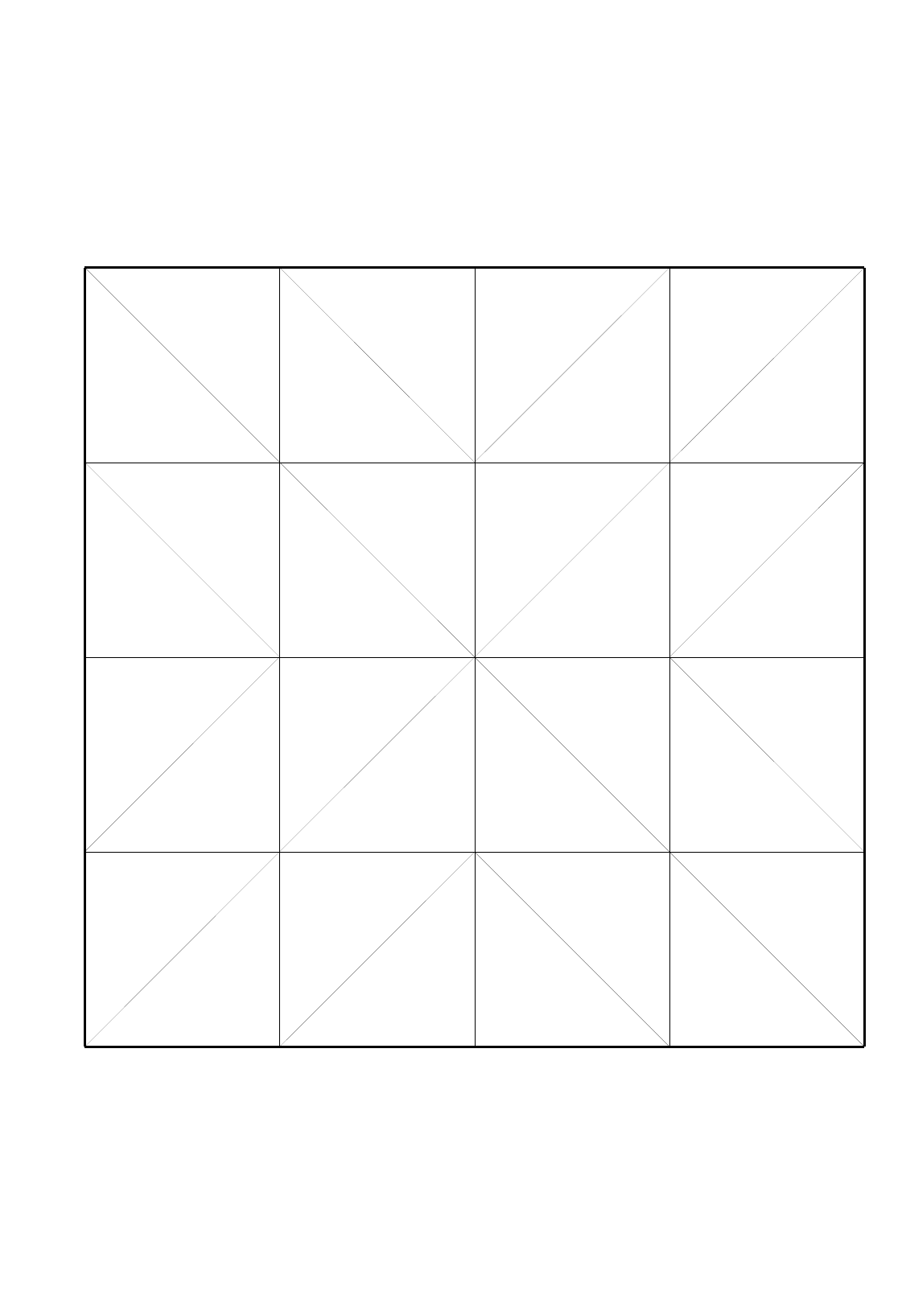}
			\centering\includegraphics[height=4.6cm, width=3.3cm]{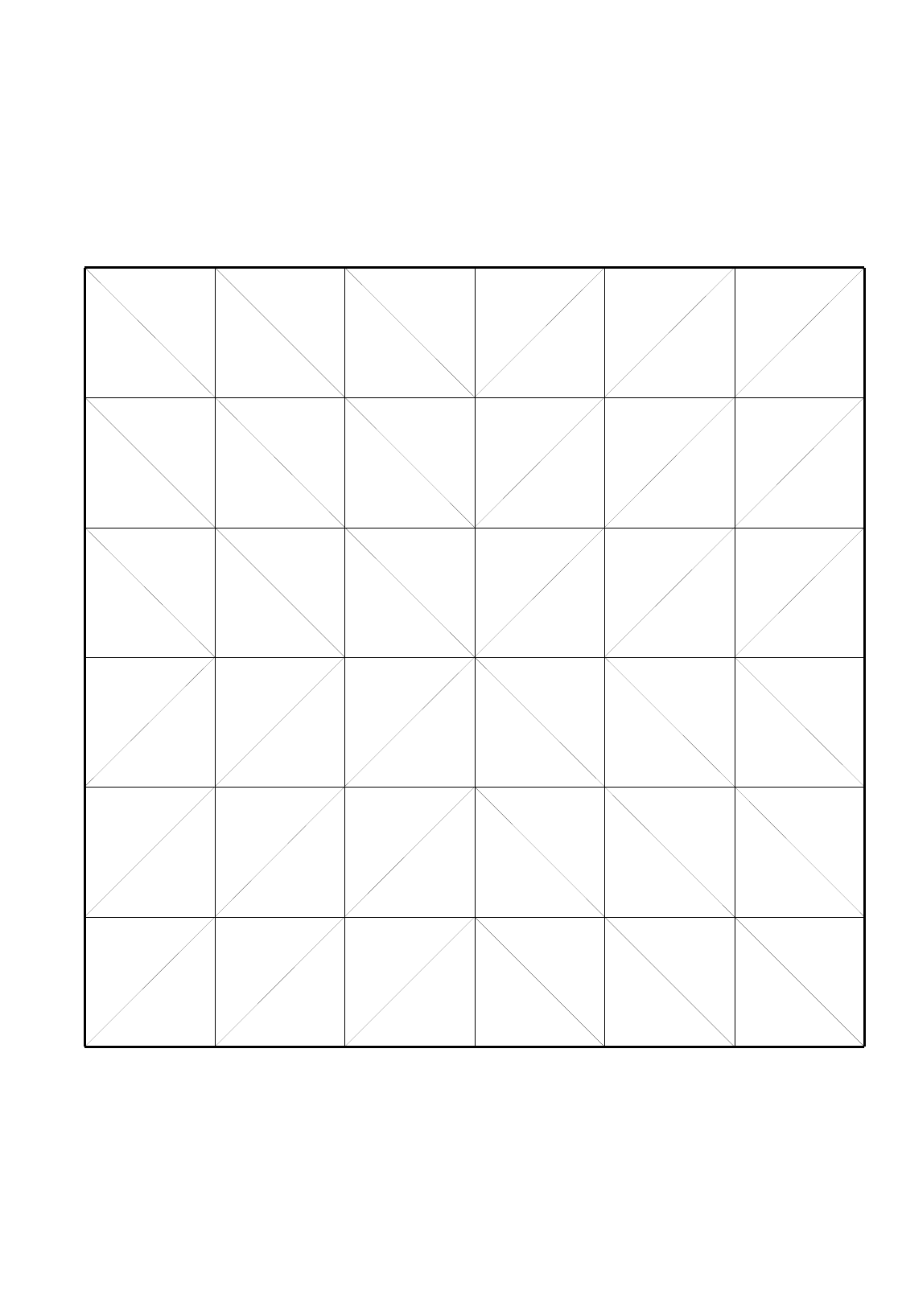}
                   \end{minipage}
		\caption{Examples of the meshes used in the unit square. The left figure represents a mesh for $N=4$ and the right one for $N=6$.}
		\label{meshes_square}
	\end{center}
\end{figure}

We observe that  the convexity of this domain is convex, leads to the sufficiently smooth eigenfunctions for the Stokes eigenvalue problem. This fact implies that the order of convergence will be optimal. For this test, we have considered as polynomial degrees $k=0, 1, 2$ together with the studied numerical schemes.

\begin{table}[H]
{\footnotesize
\begin{center}
\begin{tabular}{c |c c c c |c| c| c| c}
\toprule
 $k $        & $N=10$             &  $N=20$         &   $N=30$         & $N=40$ & Order & $\lambda_{extr}$&\cite{MR3335223} &\cite{MR2473688} \\ 
 \midrule
 & 12.61618& 12.96634  &   13.03637   & 13.05313 & 2.00& 13.08484 &13.0860&13.086    \\
 & 21.08840& 22.63791  &   22.85202   & 22.93446 & 2.40 &22.99702 &23.0308&23.031   \\
 \multirow{2}{0.7cm}{0}
 & 21.33183& 22.69036  &   22.88083   & 22.93810&  2.46&22.99245&23.0308&23.031    \\
 & 27.96811& 31.27226   &   31.70100  & 31.81983&  2.59& 31.93357&32.0443&32.053   \\
 & 33.42538& 37.66786   &   38.18478  & 38.32443&  2.69&38.44565 &38.5252&38.532   \\

          \hline

 &13.08698&13.08620&13.08617 &13.08617&  4.56&13.08617& 13.0860&13.086      \\
 &23.04310&23.03182&23.03123 &23.03114&  4.04&23.03109& 23.0308&23.031     \\
\multirow{2}{0.7cm}{1} 
 &23.04310&23.03182&23.03123 &23.03114&  4.04&23.03109& 23.0308&23.031    \\
 &32.07944&32.05400&32.05270 &32.05249&  4.07&32.05239& 32.0443&32.053     \\
 &38.60095&38.53594&38.53227 &38.53165&  3.92&38.53134& 38.5252&38.532     \\

\hline

 & 13.08528 &13.08616&13.08617&13.08617  &  5.84&13.08617& 13.0860&13.086   \\
 & 23.03116 &23.03109&23.03109&23.03109  &  6.00 &23.03109& 23.0308&23.031      \\
\multirow{2}{0.7cm}{2}   
 & 23.03116 &23.03109&23.03109&23.03109 &   6.00 &23.03109 & 23.0308&23.031     \\
 & 32.05268 &32.05239&32.05239&32.05239 &   6.00 &32.05239& 32.0443&32.053    \\
 & 38.53256 &38.53138&38.53136&38.53136 &   5.97 &38.53136& 38.5252&38.532    \\

\bottomrule             
\end{tabular}
\end{center}}
\caption{Lowest computed eigenvalues for polynomial degrees $k=0, 1, 2$ with the $[\mathbb{P}_k]^{n}\text{-}\mathbb{P}_k\text{-}\mathbb{RT}_k$   scheme. }
\label{tabla:square}
\end{table}

\begin{table}[H]
	{\footnotesize
		\begin{center}
			\begin{tabular}{c |c c c c |c| c| c| c}
				\toprule
				$k $        & $N=10$             &  $N=20$         &   $N=30$         & $N=40$ & Order & $\lambda_{extr}$&\cite{MR3335223} &\cite{MR2473688} \\ 
				\midrule
				& 13.18205& 13.10744   &   13.09534   & 13.09127 & 2.21 & 13.08688 &13.0860&13.086    \\
				& 22.59086& 22.92419   &   22.98366   & 23.00442 & 2.06 & 23.02944 &23.0308&23.031   \\
				\multirow{2}{0.7cm}{0}
				& 22.59086& 22.92419   &   22.98366  & 23.00442  & 2.06 & 23.02944 &23.0308&23.031    \\
				& 31.52148& 31.92201   &   31.99384  & 32.01930  & 2.04 & 32.05042 &32.0443&32.053   \\
				& 36.97903& 38.18216   &   38.37946  & 38.44657  & 2.19 & 38.51958 &38.5252&38.532   \\
				
				\hline
				
				&13.08698&13.08620&13.08617 &13.08617&  4.56&13.08617& 13.0860&13.086      \\
				&23.04310&23.03182&23.03123 &23.03114&  4.04&23.03122& 23.0310&23.031     \\
				\multirow{2}{0.7cm}{1} 
				&23.04310&23.03182&23.03123 &23.03114& 4.04 &23.03109& 23.0308&23.031    \\
				&32.07944&32.05400&32.05270 &32.05249& 4.07 & 32.05239& 32.0443&32.053     \\
				&38.60095&38.53594&38.53227 &38.53165& 3.92 &38.53134& 38.5252&38.532     \\
				
				\hline
				
				& 13.08615 &13.08617&13.08617&13.08617  &  4.75 &13.08617& 13.0860&13.086   \\
				& 23.03116 &23.03109&23.03109&23.03109  &  6.00 &23.03109& 23.0308&23.031      \\
				\multirow{2}{0.7cm}{2}   
				& 23.03116 &23.03109&23.03109&23.03109 &  6.00  &23.03110 & 23.0308&23.031     \\
				& 32.05268 &32.05239&32.05239&32.05239 &  6.00 &32.05239& 32.0443&32.053    \\
				& 38.53256 &38.53138&38.53136&38.53136 &  5.92 &38.53136& 38.5252&38.532    \\

				\bottomrule             
			\end{tabular}
	\end{center}}
	\caption{Lowest computed eigenvalues for polynomial degrees $k=0, 1, 2$ with the $[\mathbb{P}_k]^{n}\text{-}\mathbb{RT}_k$  scheme. }
	\label{tabla:square-scheme2}
\end{table}

\begin{table}[H]
	{\footnotesize
		\begin{center}
			\begin{tabular}{c |c c c c |c| c| c| c}
				\toprule
				$k $        & $N=10$             &  $N=20$         &   $N=30$         & $N=40$ & Order & $\lambda_{extr}$&\cite{MR3335223} &\cite{MR2473688} \\ 
				\midrule
				& 13.39520& 13.16477  &   13.12123   & 13.10591 & 1.97& 13.08574 &13.0860&13.086    \\
				& 23.74378& 23.22000  &   23.11593   & 23.07899 & 1.89 &23.02641 &23.0308&23.031   \\
				\multirow{2}{0.7cm}{0}& 24.19514& 23.32856  &   23.16384   & 23.10587&  1.96&23.02865&23.0308&23.031    \\
				& 33.73344& 32.50272   &   32.25523  & 32.16703& 1.87& 32.03920&32.0443&32.053   \\
				& 41.15209& 39.23059   &   38.84532  & 38.70858&  1.88&38.51262 &38.5252&38.532   \\
				
				\hline
				
				&13.08919&13.08636&13.08621 &13.08618&  3.99&13.08617& 13.0860&13.086      \\
				&23.04441&23.03195&23.03126 &23.03115&  3.96&23.03109& 23.0308&23.031     \\
				\multirow{2}{0.7cm}{1} &23.05331&23.03253&23.03138 &23.03118&  3.95&23.03109& 23.0308&23.031    \\
				&32.10055&32.0555&32.05302 &32.05259& 3.92& 32.05238& 32.0443&32.053     \\
				&38.61259&38.53671&38.53243 &38.53170&  3.92&38.53134& 38.5252&38.532     \\
				
				\hline
				
				& 13.08618 &13.08617&13.08617&13.08617  &  6.16&13.08617& 13.0860&13.086   \\
				& 23.03117 &23.03109&23.03109&23.03109  &  6.04 &23.03109& 23.0308&23.031      \\
				\multirow{2}{0.7cm}{2}   & 23.03128 &23.03110&23.03109&23.03109 &  6.01  &23.03109 & 23.0308&23.031     \\
				& 32.05303 &32.05240&32.05239&32.05239 &   6.02 &32.05239& 32.0443&32.053    \\
				& 38.53239 &38.53138&38.53136&38.53136 &   5.92 &38.53136& 38.5252&38.532    \\

				\bottomrule             
			\end{tabular}
	\end{center}}
	\caption{Lowest computed eigenvalues for polynomial degrees $k=1, 2, 3.$ with the $[\mathbb{P}_k]^{n}\text{-}\mathbb{P}_k\text{-}\mathbb{BDM}_{k+1}$   scheme. }
	\label{tabla:square-BDM}
\end{table}

\begin{table}[H]
	{\footnotesize
		\begin{center}
			\begin{tabular}{c |c c c c |c| c| c| c}
				\toprule
				$k $        & $N=10$             &  $N=20$         &   $N=30$         & $N=40$ & Order & $\lambda_{extr}$&\cite{MR3335223} &\cite{MR2473688} \\ 
				\midrule
				& 13.46029& 13.18088  &  13.12837  & 13.10993 & 1.98& 13.08589 &13.0860&13.086    \\
				& 24.18596& 23.32433  &  23.16178  & 23.10467 & 1.97& 23.02910 &23.0308&23.031   \\
				\multirow{2}{0.7cm}{0}
				& 24.18596& 23.32433  &  23.16178  & 23.10467 & 1.97& 23.02910 &23.0308&23.031    \\
				& 34.23489& 32.61702  &  32.30485  & 32.19470 & 1.94& 32.04581 &32.0443&32.053   \\
				& 41.75299& 39.35261  &  38.89728  & 38.73736 & 1.96& 38.52295 &38.5252&38.532   \\
				
				\hline
				
				&13.08997&13.08642&13.08622 &13.08618&  3.93&13.08617& 13.0860&13.086      \\
				&23.05092&23.03240&23.03135 &23.03118&  3.92&23.03109& 23.0308&23.031     \\
				\multirow{2}{0.7cm}{1} 
				&23.05092&23.03240&23.03135 &23.03118 & 3.92&23.03109& 23.0308&23.031    \\
				&32.10848&32.05619&32.05315 &32.05263 & 3.88&32.05237& 32.0443&32.053     \\
				&38.61788&38.53707&38.53250 &38.53172 & 3.92&38.53134& 38.5252&38.532     \\
				
				\hline
				
				& 13.08619 &13.08617&13.08617&13.08617 &  6.09 &13.08617& 13.0860&13.086   \\
				& 23.03128 &23.03110&23.03109&23.03109 &  5.97 &23.03109& 23.0308&23.031      \\
				\multirow{2}{0.7cm}{2}  
				& 23.03128 &23.03110&23.03109&23.03109 &  5.97 &23.03109 & 23.0308&23.031     \\
				& 32.05323 &32.05240&32.05239&32.05239 &  5.91 &32.05239& 32.0443&32.053    \\
				& 38.53245 &38.53138&38.53136&38.53136 &  5.92 &38.53136& 38.5252&38.532    \\

				\bottomrule             
			\end{tabular}
	\end{center}}
	\caption{Lowest computed eigenvalues for polynomial degrees $k=1, 2, 3.$ with the $[\mathbb{P}_k]^{n}\text{-}\mathbb{BDM}_{k+1}$   scheme. }
	\label{tabla:square-BDM-scheme2}
\end{table}

In Table \ref{tabla:square} we report the first five eigenvalues computed with the $[\mathbb{P}_k]^{n}\text{-}\mathbb{P}_k\text{-}\mathbb{RT}_k$   scheme, considering different meshes and polynomial degrees. In the column 
$\lambda_{extr}$ we report extrapolated values, obtained with a lest square fitting, which we compare with two very well known references that 
have deal with the same domain. We observe that our extrapolated values are close to those in \cite{MR3335223,MR2473688} and that the rates of convergence
are as we expect. In fact, we notice that for $k=0$, the order of approximation  is clearly $h^2$, meanwhile for $k>0$ the observed order is  
close to $h^{2(k+1)}$. This increased order is expectable for high order methods, as it happen, for example,  in DG methods (see for instance \cite{MR2324460, MR4077220}). In Table \ref{tabla:square-scheme2}, the results from using the $[\mathbb{P}_k]^{n}\text{-}\mathbb{RT}_k$   scheme are provided. We observe that similar results are obtained when compared with the results from Table \ref{tabla:square}.

On the other hand, Table \ref{tabla:square-BDM} shows the computed eigenvalues when using the $[\mathbb{P}_k]^{n}\text{-}\mathbb{P}_k\text{-}\mathbb{BDM}_{k+1}$   scheme,  where we observe that an optimal rate of convergence $O(h^{2(k+1)})$ is reached for high order elements. For instance, in Figure \ref{figura:error-cuadrado} we observe that, except for the noise present in the error slopes, the scheme allows to stay on the optimal rate of convergence. This is compared with Table \ref{tabla:square-BDM-scheme2}, where we have the computed eigenvalues from the $[\mathbb{P}_k]^{n}\text{-}\mathbb{BDM}_{k+1}$   scheme. As before, the results are considerably similar with those from Table \ref{tabla:square-BDM}. Hence, we conclude that, although the full and reduced numerical schemes studied are not equivalent, they yield to the same numerical results. To complete the experiment, we present in Figure \ref{figura:modos-cuadrado} the first, third and fourth lowest computed  eigenfunctions on the square domain, and in Figure \ref{figura:error-cuadrado} we present the error behavior on the chosen numerical schemes.

\begin{figure}[H]
\centering\begin{minipage}{0.32\linewidth}
	\centering\includegraphics[scale=0.098]{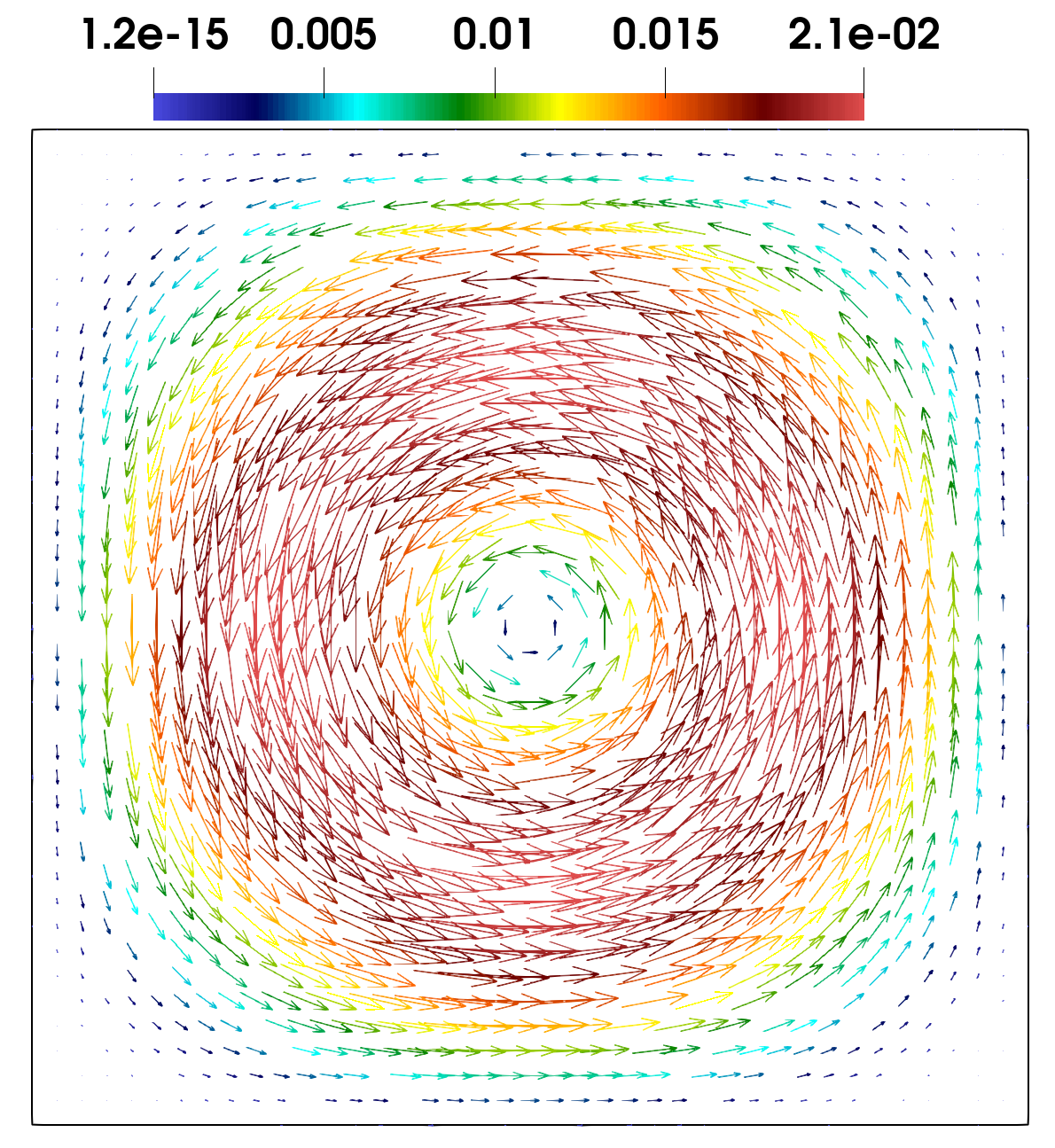}
\end{minipage}
\begin{minipage}{0.32\linewidth}
	\centering\includegraphics[scale=0.098]{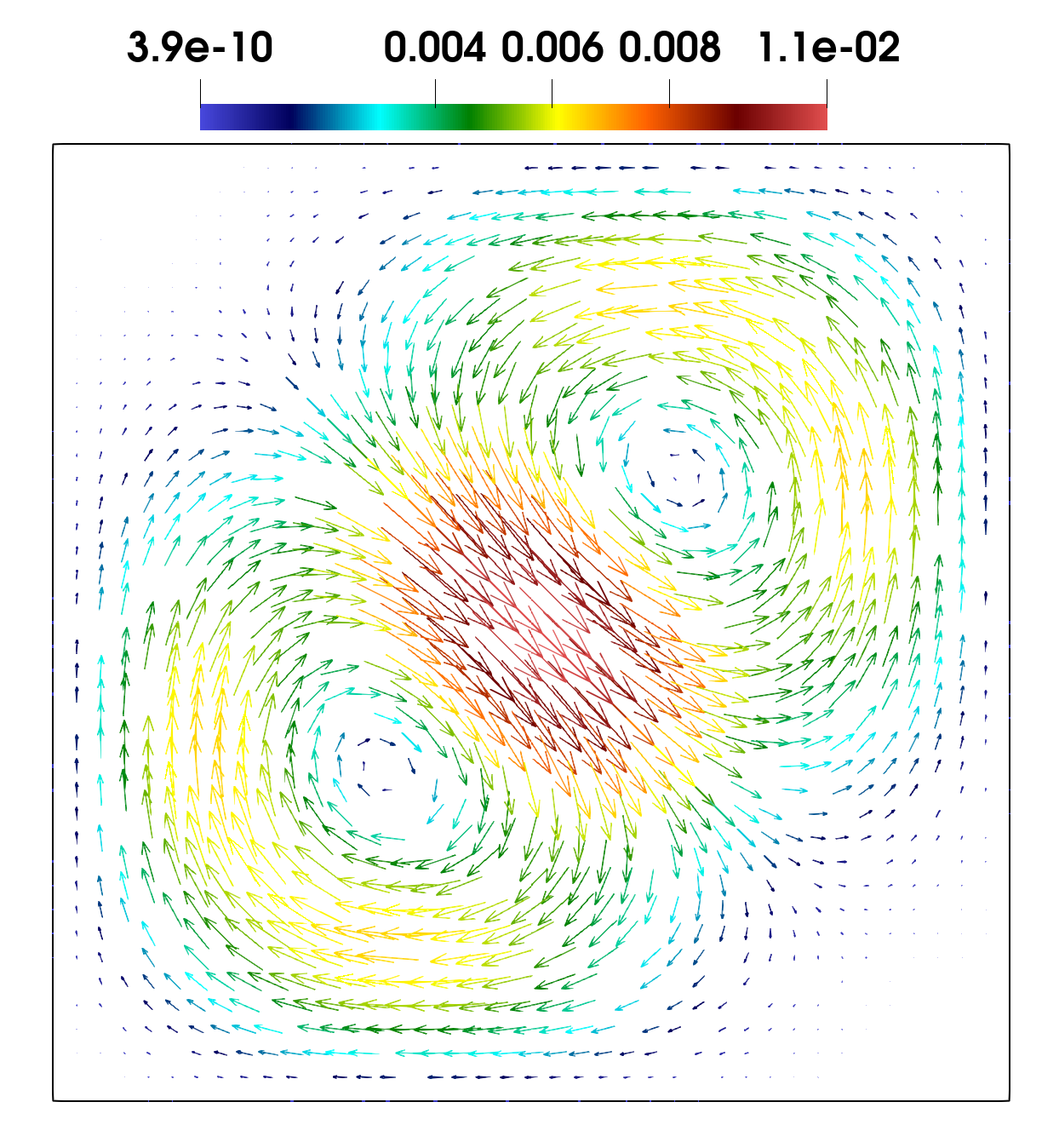}
\end{minipage}
\begin{minipage}{0.32\linewidth}
	\centering\includegraphics[scale=0.098]{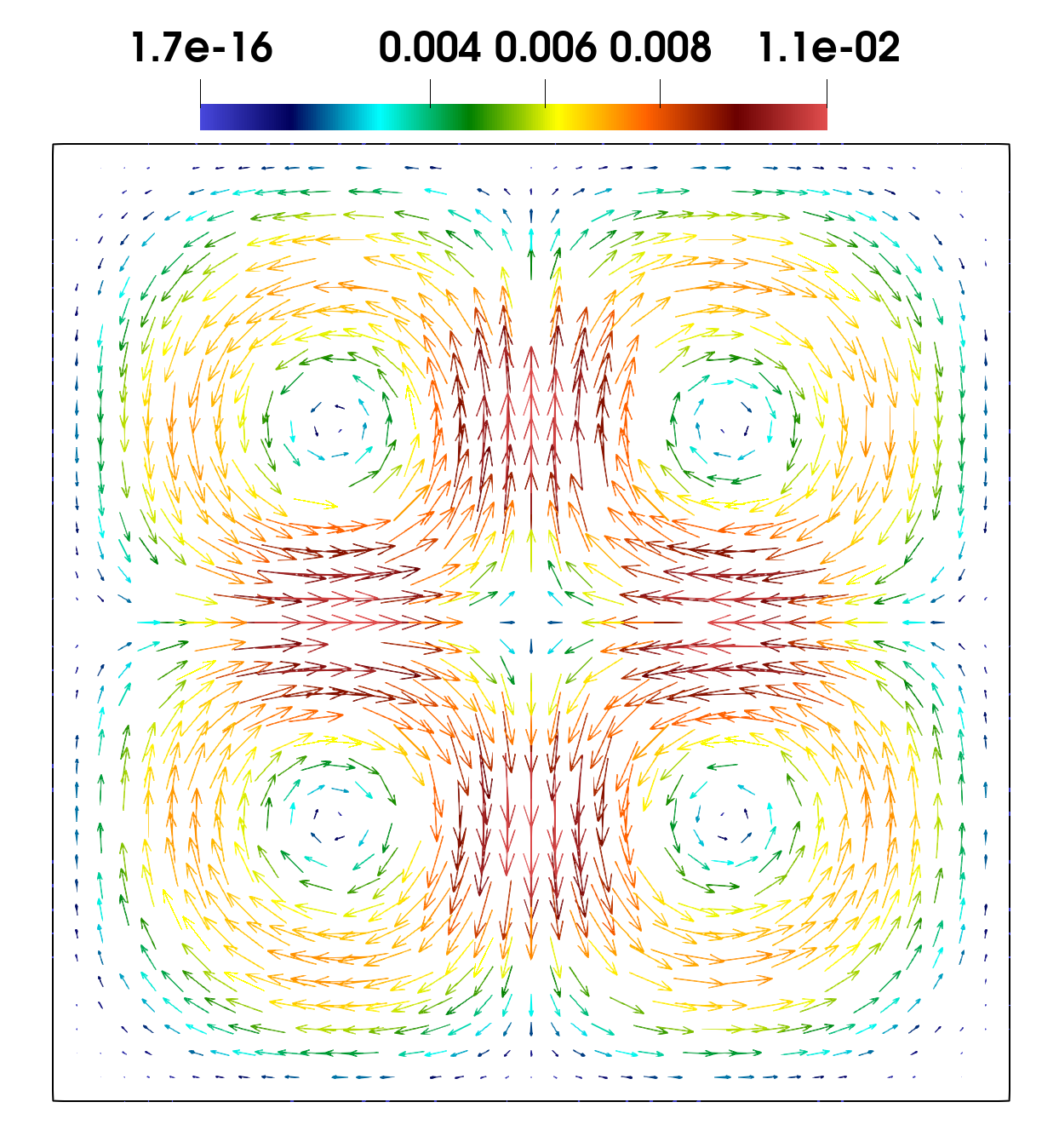}
\end{minipage}\\
\begin{minipage}{0.32\linewidth}
	\centering\includegraphics[scale=0.098]{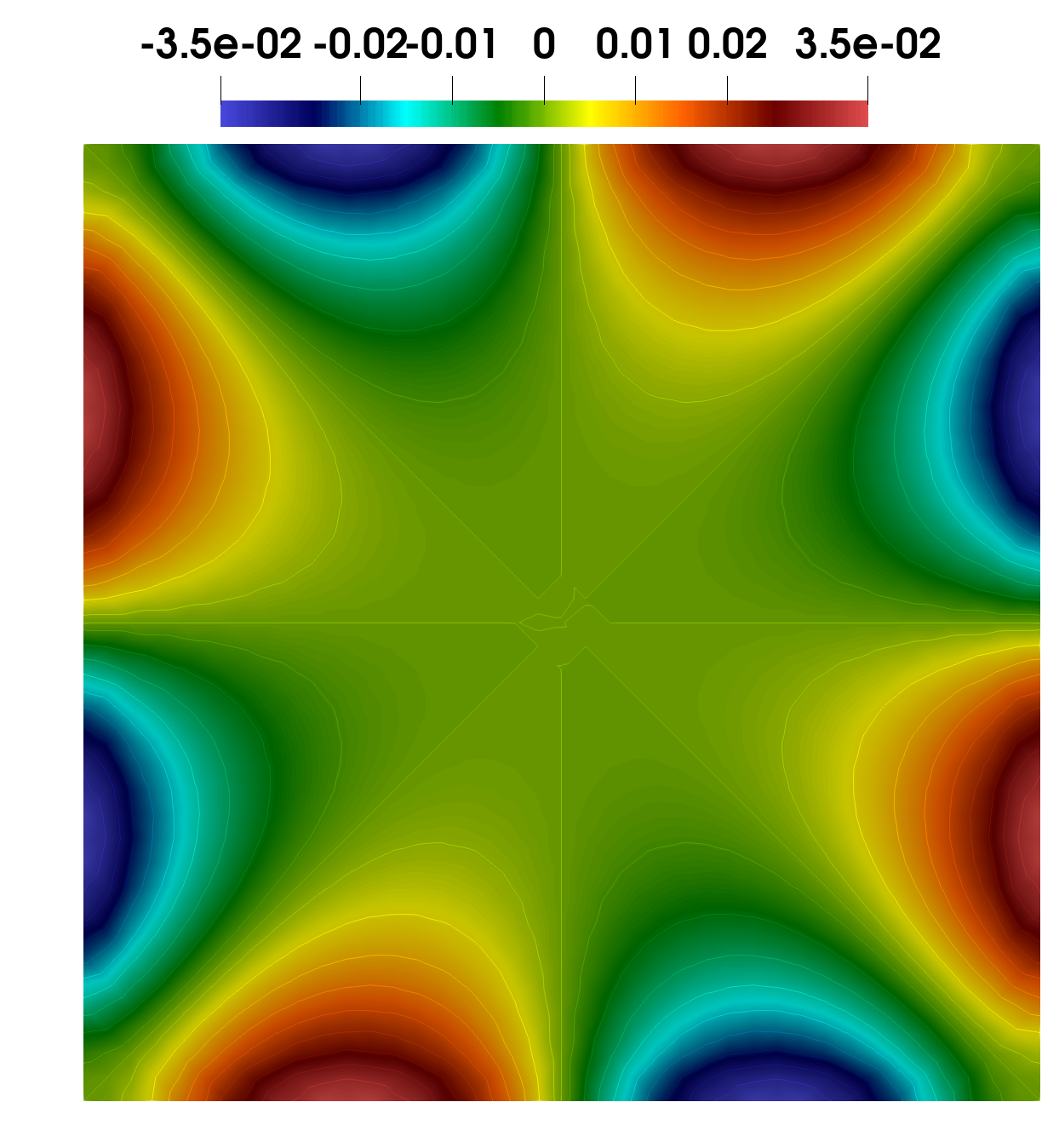}
\end{minipage}
\begin{minipage}{0.32\linewidth}
	\centering\includegraphics[scale=0.098]{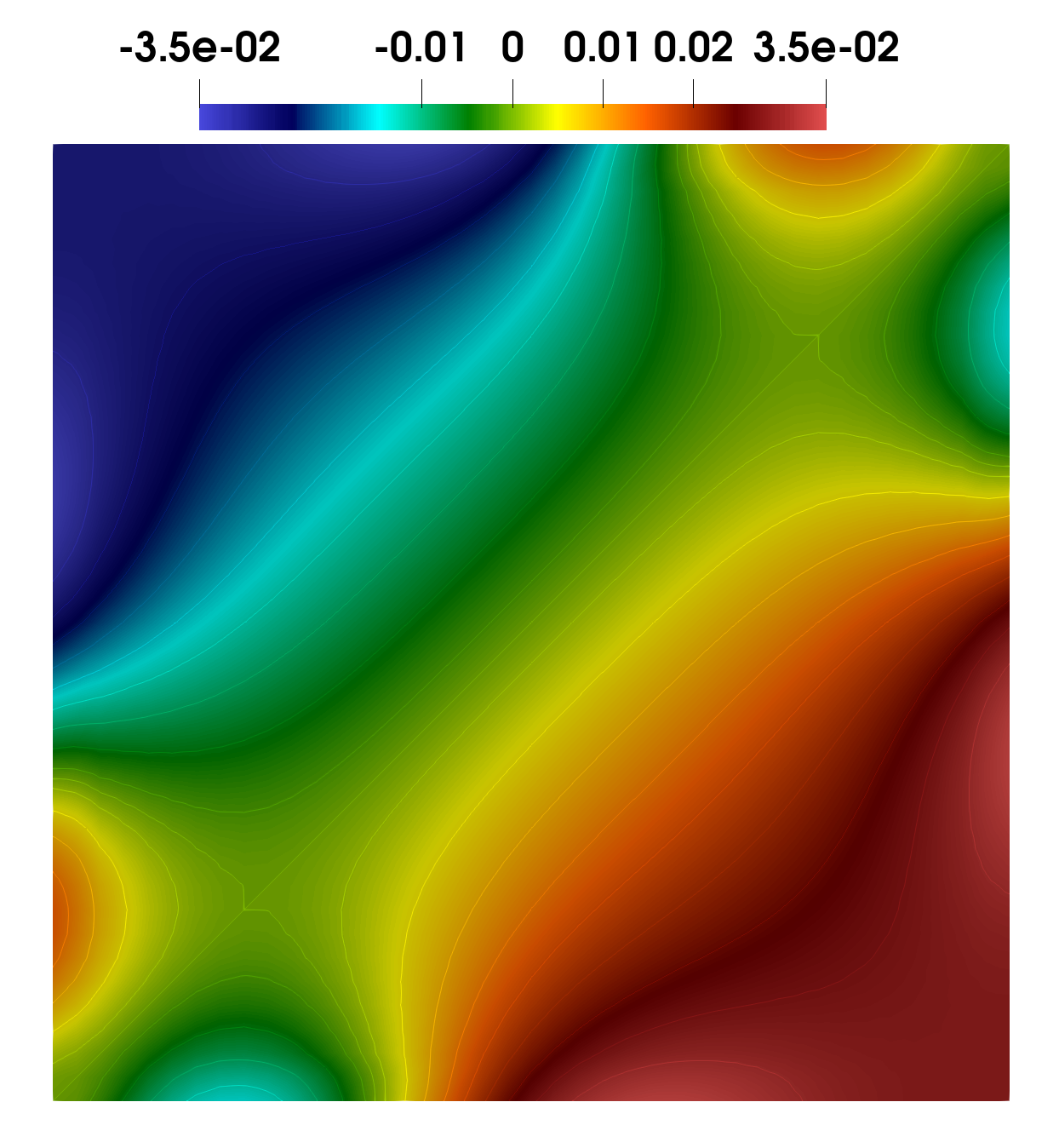}
\end{minipage}
\begin{minipage}{0.32\linewidth}
	\centering\includegraphics[scale=0.098]{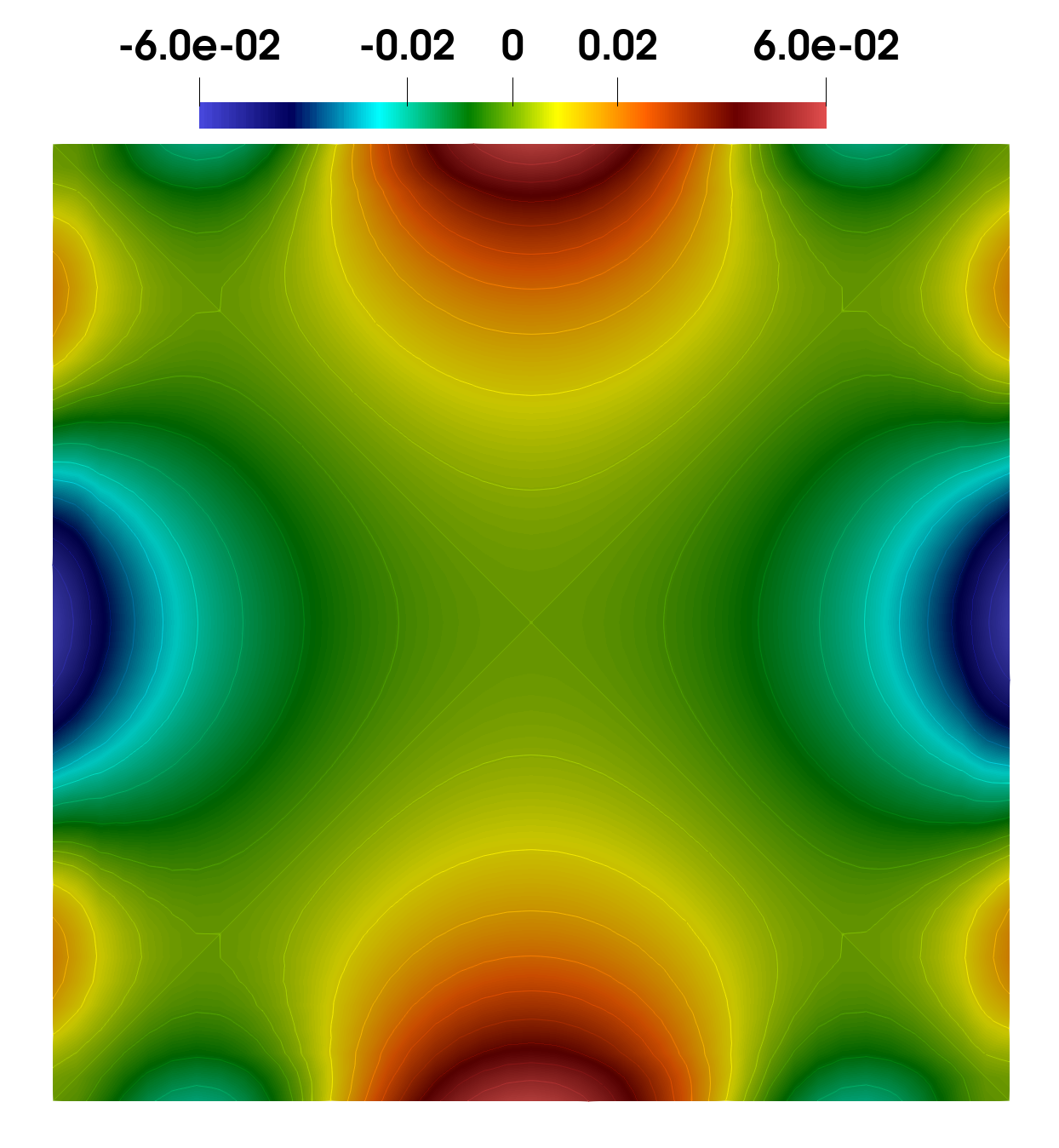}
\end{minipage}
\caption{Approximate velocity field $\bu_h$ (top row) and pressure $p_h$ (bottom row), corresponding to the first, third and fourth lowest eigenvalues in the square domain.}
\label{figura:modos-cuadrado}
\end{figure}

\begin{figure}[H]
	\begin{minipage}{0.49\linewidth}
		\centering\includegraphics[scale=0.42]{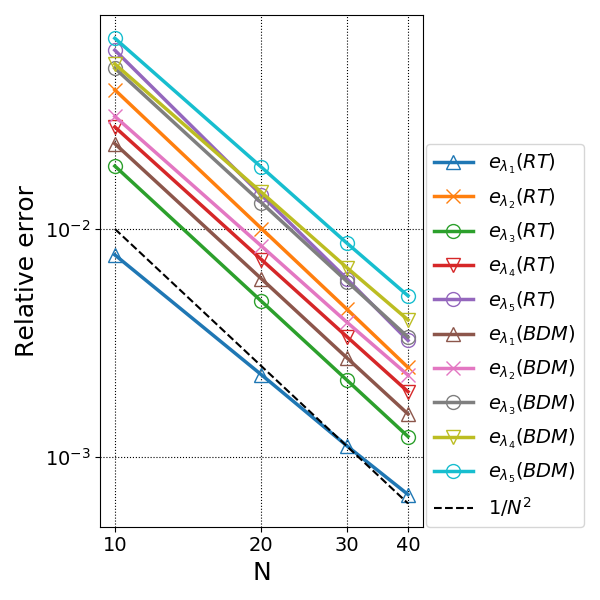}
	\end{minipage}
	\begin{minipage}{0.49\linewidth}
		\centering\includegraphics[scale=0.42]{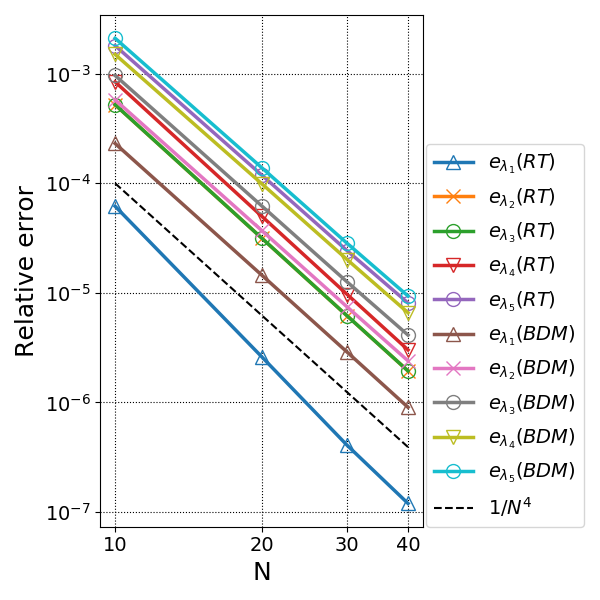}
	\end{minipage}
%
\caption{Comparison of the eigenvalues error behavior in the square domain when using $[\mathbb{P}_k]^{n}\text{-}\mathbb{P}_k\text{-}\mathbb{RT}_k$   and $[\mathbb{P}_k]^{n}\text{-}\mathbb{P}_k\text{-}\mathbb{BDM}_{k+1}$   schemes. The experiment considers polynomials of degree $k=0$ (left) and $k=1$ (right).}
\label{figura:error-cuadrado}
\end{figure}
%

 \subsection{Test 2: Circular domain}
 In this test we consider the unitary circle as computational domain, which we define by $\O_C:=\{(x,y)\in\mathbb{R}^2\,:\, x^2+y^2\leq 1\}$. The relevance of this experiment is that we are approximating a curved domain with triangular meshes, which lead to a variational crime.  In Figure \ref{meshes_circle} we present examples of the quasi-uniform triangular meshes considered to approximate the circular domain.
 
The fact that we are approximating a curved domain by means of a polygonal one is reflected in the numerical experiments presented below in Table \ref{tabla:circle} where, independent of the polynomial degree, the order of convergence is $\mathcal{O}(h^2)$ for all $k\geq 0$.
 
 \begin{figure}[H]
	\begin{center}
		\begin{minipage}{14cm}
			\centering\includegraphics[height=3.4cm, width=3.4cm]{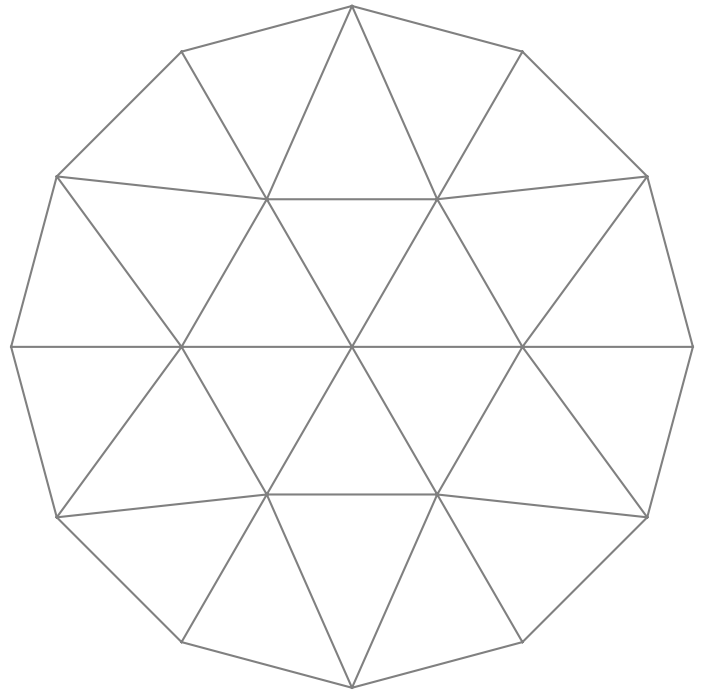}
			\centering\includegraphics[height=3.4cm, width=3.4cm]{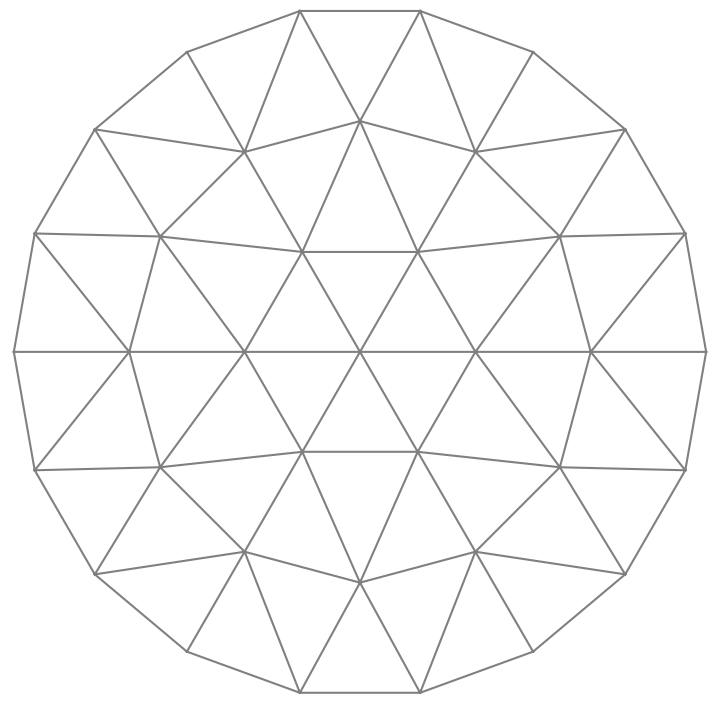}
			\centering\includegraphics[height=3.4cm, width=3.4cm]{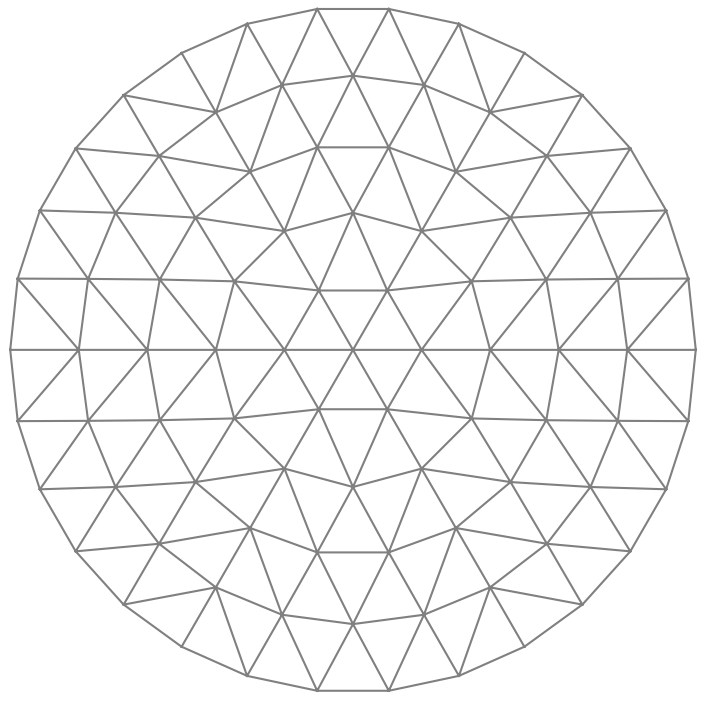}
                   \end{minipage}
		\caption{Meshes used in the circular domain.}
		\label{meshes_circle}
	\end{center}
\end{figure}
The results from using the $[\mathbb{P}_k]^{n}\text{-}\mathbb{P}_k\text{-}\mathbb{BDM}_{k+1}$  scheme are described in Table \ref{tabla:circle-BDM}, where similar rates of convergence are observed. We recall that, in both cases, $N$ represents the mesh resolution such that the number of elements is $6N^2$. In Figure \ref{figura:modos-circulo} we present the approximated eigenfunctions for  the  lowest frequencies. We further describe the results obtained in Figure \ref{figura:error-circulo}, where we observe the experimental rates obtained, which are in good agreement with those predicted by theory. 
 \begin{table}[H]
{\footnotesize
\begin{center}
\begin{tabular}{c |c c c c |c| c| c }
\toprule
 $k $        & $N=20$             &  $N=30$         &   $N=40$         & $N=50$ & Order & $\lambda_{extr}$ &\cite{ MR4077220} \\ 
 \midrule
                                    &  14.94827 & 14.79867 & 14.74712&  14.72354   &2.04&  14.68251  &14.68345  \\
                                    &  26.81747 & 26.56803 & 26.48211&   26.44329  &2.05&  26.37559  &26.37840   \\
 \multirow{2}{0.7cm}{0}&  26.81821 & 26.56845 & 26.48262&  26.44365   &2.06&  26.37683  &26.37862   \\
                                    &  41.32838 & 40.98177 & 40.85915&  40.80453  &2.01 & 40.70533   &40.71434    \\
                                    &  41.34096 & 40.98359 & 40.86093&  40.80487  &2.05 & 40.70809   &40.71606   \\

          \hline

                                    &  14.94196&  14.79448&  14.74448 & 14.72169&   2.08& 14.68323 &14.68345    \\
                                    &  26.84091&  26.57657&  26.48686 & 26.44594&   2.08& 26.37704 &26.37840     \\
\multirow{2}{0.7cm}{1} &  26.84099&  26.57662&  26.48687 & 26.44595&   2.08& 26.37703 &26.37862	\\
                                    &  41.42501&  41.01797&  40.87964 & 40.81652&   2.08& 40.71046 &40.71434	 \\
                                    &  41.42543&  41.01805&  40.87966 & 40.81654&   2.08& 40.71037 &40.71606	 \\

          \hline

                                      &  14.94315&  14.79487&  14.74464 & 14.72177 &  2.09 & 14.68361  & 14.68345  \\
                                      &  26.84301&  26.57727&  26.48715 & 26.44610 &  2.08 & 26.37680  &26.37840 	\\
\multirow{2}{0.7cm}{2}   &   26.84303&  26.57728&  26.48716 & 26.44610 &  2.08 & 26.37680 &26.37862	\\
                                      &  41.42807&  41.01900&  40.88008 & 40.81675 &  2.08 & 40.71012 &40.71434	 \\
                                      &  41.42814&  41.01902&  40.88008 & 40.81676 &  2.08 & 40.71010 &40.71606	 \\

\bottomrule             
\end{tabular}
\end{center}}
\caption{Lowest computed eigenvalues for polynomial degrees $k=0, 1, 2.$ using the  $[\mathbb{P}_k]^{n}\text{-}\mathbb{P}_k\text{-}\mathbb{RT}_k$   scheme.}
\label{tabla:circle}
\end{table}
 \begin{table}[H]
	{\footnotesize
		\begin{center}
			\begin{tabular}{c |c c c c |c| c| c }
				\toprule
				$k $        & $N=20$             &  $N=30$         &   $N=40$         & $N=50$ & Order & $\lambda_{extr}$ &\cite{ MR4077220} \\ 
				\midrule
				&  14.82469 & 14.71768 & 14.69784&  14.69090   &2.00&  14.68199  &14.68345  \\
				&  26.77392 & 26.47427 & 26.41889&  26.39951  &2.00&  26.37450  &26.37840   \\
\multirow{2}{0.7cm}{0}&  26.77392 & 26.47427 & 26.41889&  26.39951   &2.00&  26.37450  &26.37862   \\
				&  41.56881 & 40.92423 & 40.80343&  40.76105  &1.98 & 40.70545   &40.71434    \\
				&  41.56881 & 40.92423 & 40.80343&  40.76105  &1.98 & 40.70545   &40.71606   \\
				
				\hline
				
				&  14.70933&  14.68872&  14.68496 & 14.68365&   2.02& 14.68199 &14.68345    \\
				&  26.42481&  26.38682&  26.38000 & 26.37764&   2.05& 26.37473 &26.37840     \\
\multirow{2}{0.7cm}{1} & 26.42481&  26.38682&  26.38000 & 26.37764&   2.05& 26.37703 &26.37862	\\
				&  40.78741&  40.72552&  40.71483 & 40.71115&   2.11& 40.70686 &40.71434	 \\
				&  40.78741&  40.72552&  40.71483 & 40.71115&   2.11& 40.70686 &40.71606	 \\
				
				\hline
				
				&  14.70930&  14.68873&  14.68496 & 14.68365 &  2.02 & 14.68200  & 14.68345  \\
				&  26.42370&  26.38677&  26.38000 & 26.37764 &  2.02 & 26.37467  &26.37840 	\\
\multirow{2}{0.7cm}{2}   &   26.42370&  26.38677&  26.38000 & 26.37764 &  2.02 & 26.37467 &26.37862	\\
				&  40.78222&  40.72523&  40.71478 & 40.71113 &  2.02 & 40.70655 &40.71434	 \\
				&  40.78222&  40.72523&  40.71478 & 40.71113 &  2.02 & 40.70655 &40.71606	 \\

				\bottomrule             
			\end{tabular}
	\end{center}}
	\caption{Lowest computed eigenvalues for polynomial degrees $k=0, 1, 2,$ when using the $[\mathbb{P}_k]^{n}\text{-}\mathbb{P}_k\text{-}\mathbb{BDM}_{k+1}$   scheme. }
	\label{tabla:circle-BDM}
\end{table}

\begin{figure}[H]
	\begin{minipage}{0.31\linewidth}
		\centering\includegraphics[scale=0.095]{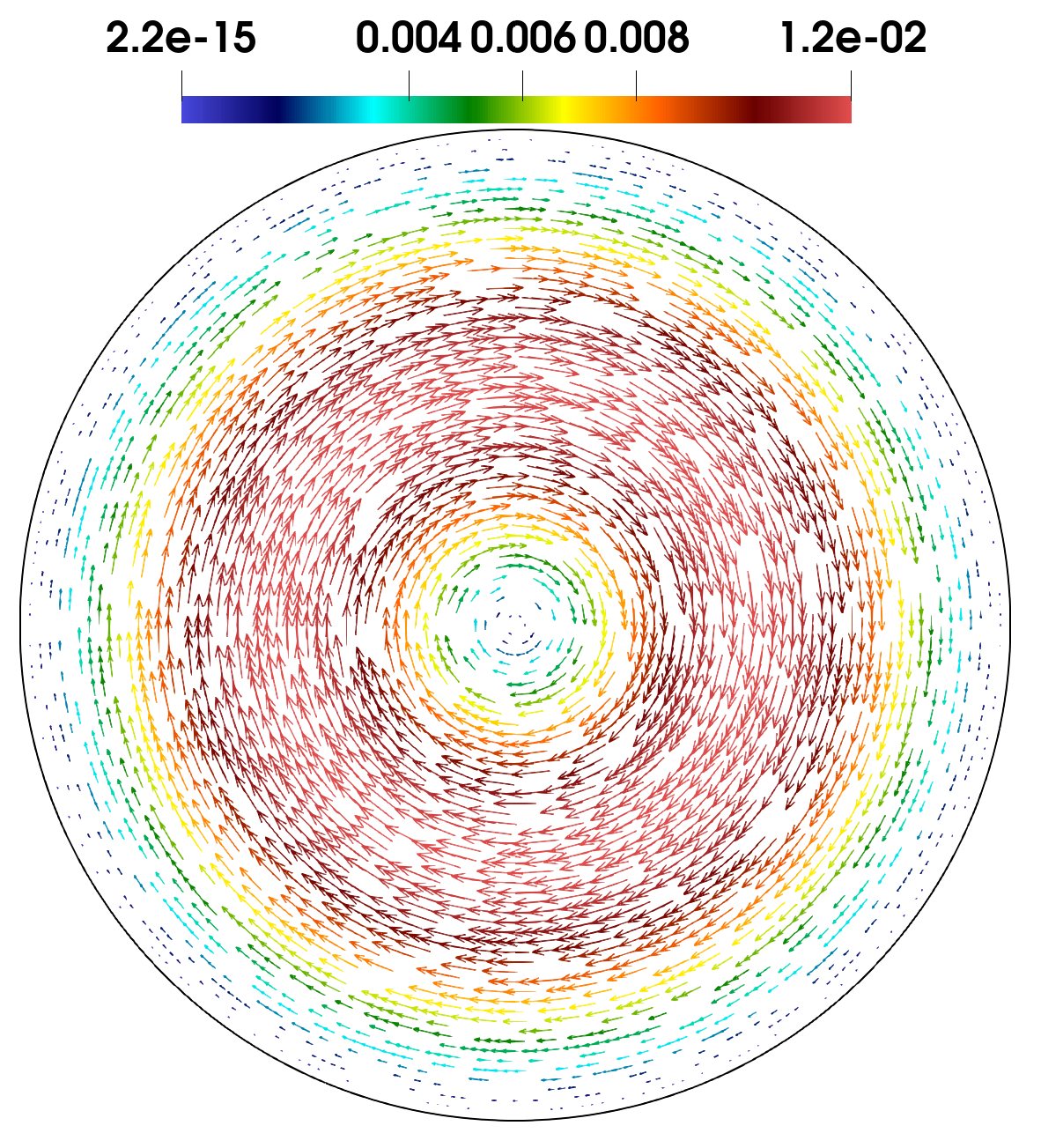}
	\end{minipage}
	\begin{minipage}{0.3\linewidth}
		\centering\includegraphics[scale=0.095]{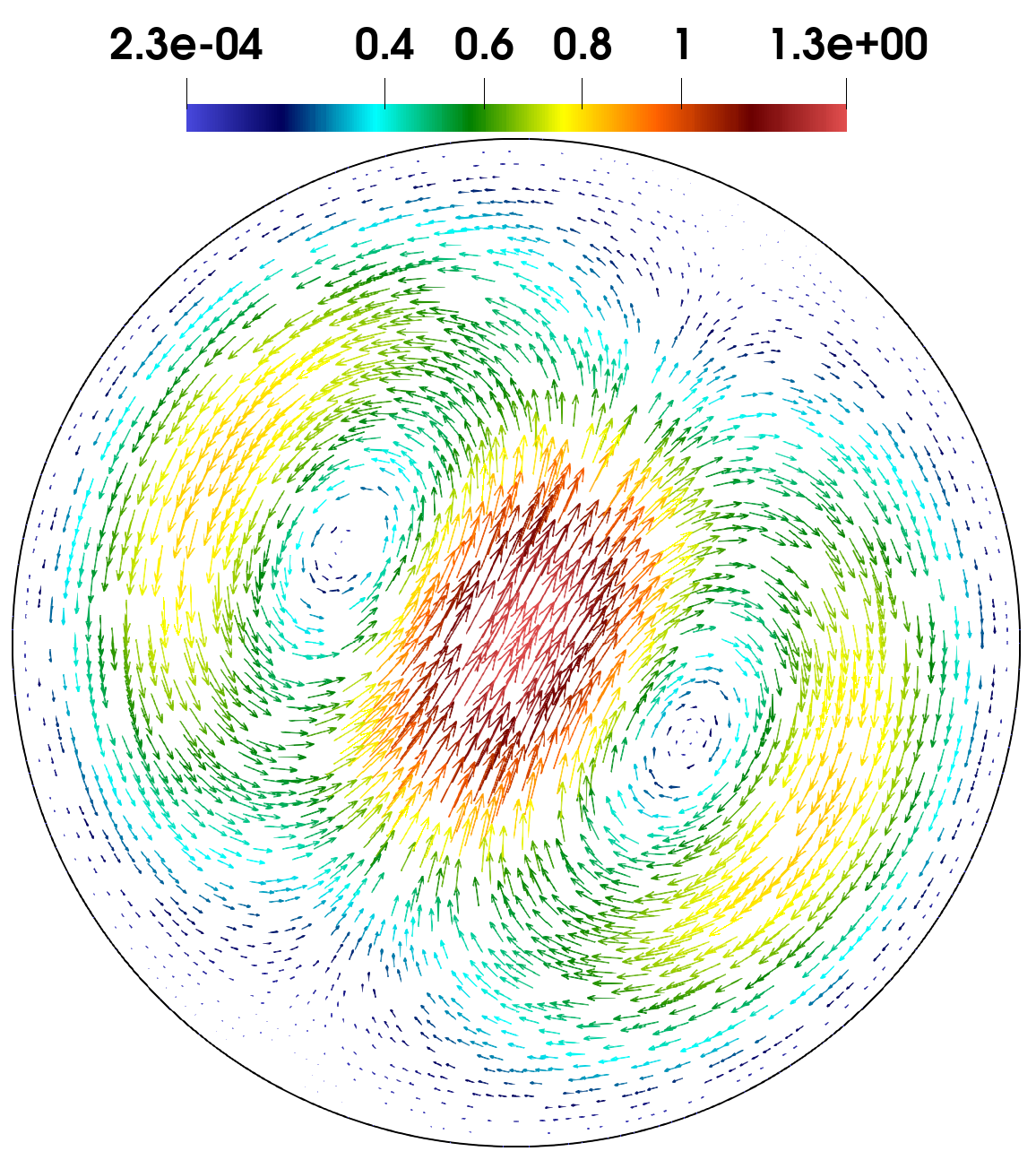}
	\end{minipage}
	\begin{minipage}{0.3\linewidth}
		\centering\includegraphics[scale=0.095]{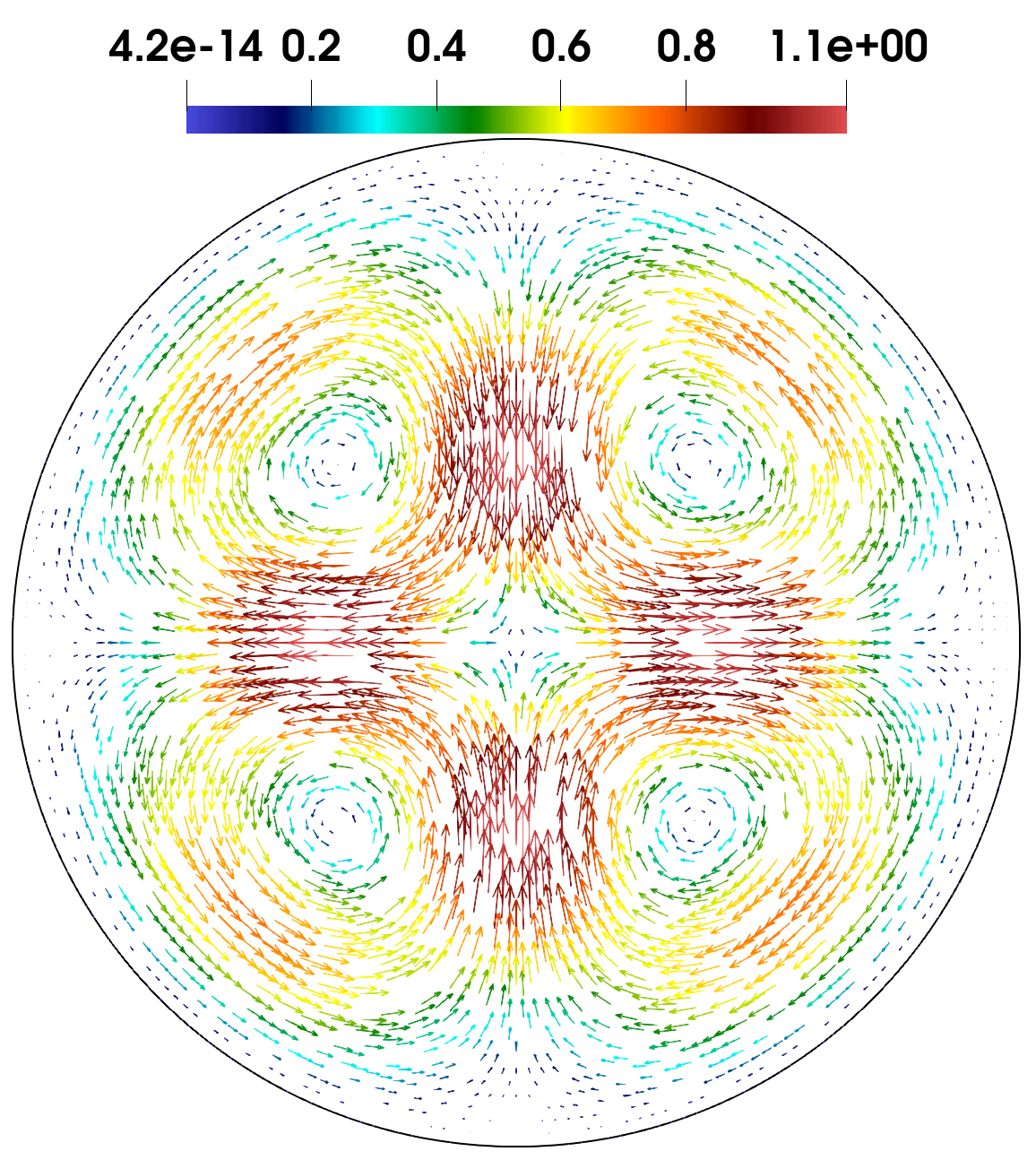}
	\end{minipage}\\
	\begin{minipage}{0.31\linewidth}
		\centering\includegraphics[scale=0.098]{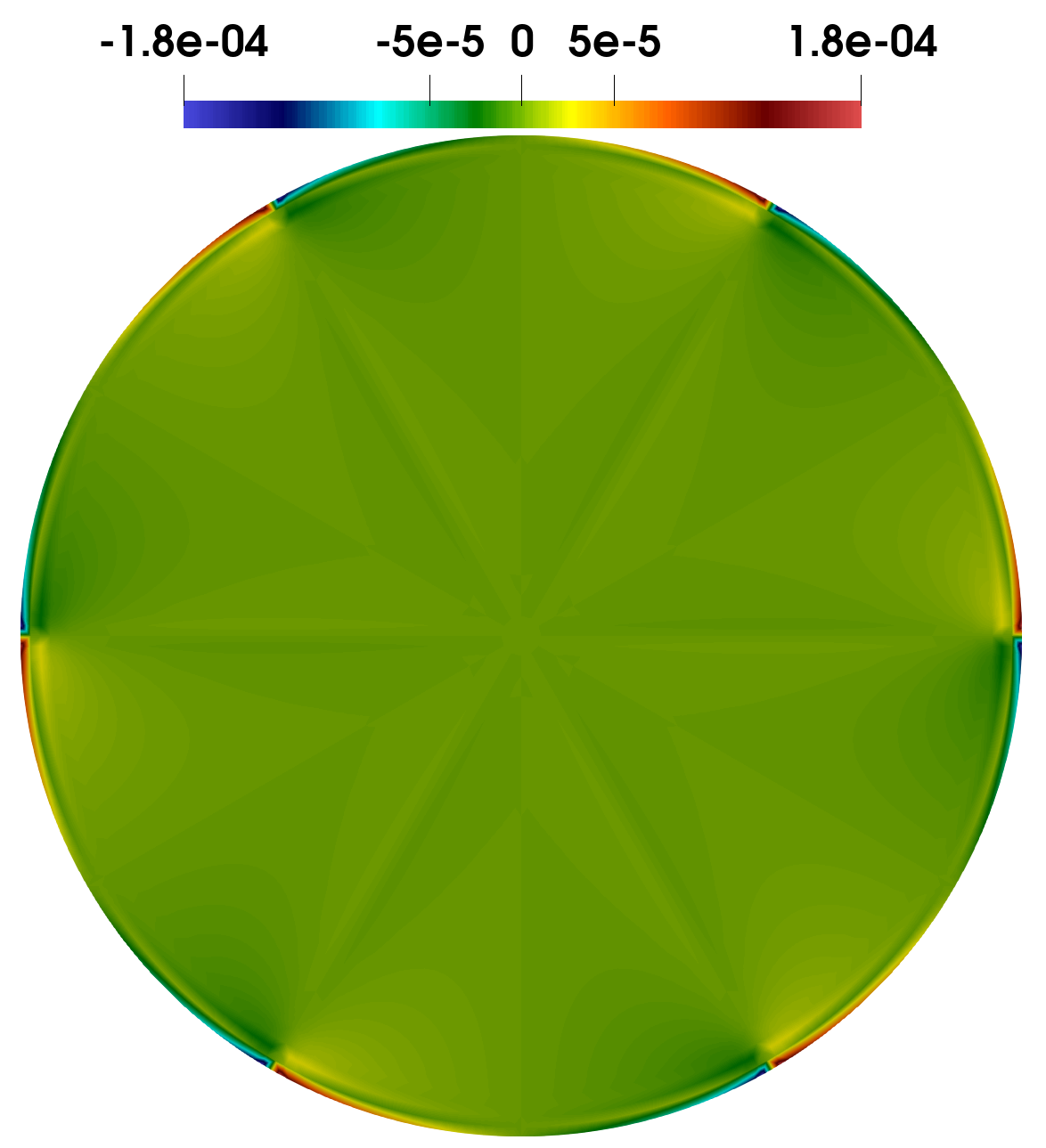}
	\end{minipage}
	\begin{minipage}{0.3\linewidth}
		\centering\includegraphics[scale=0.098]{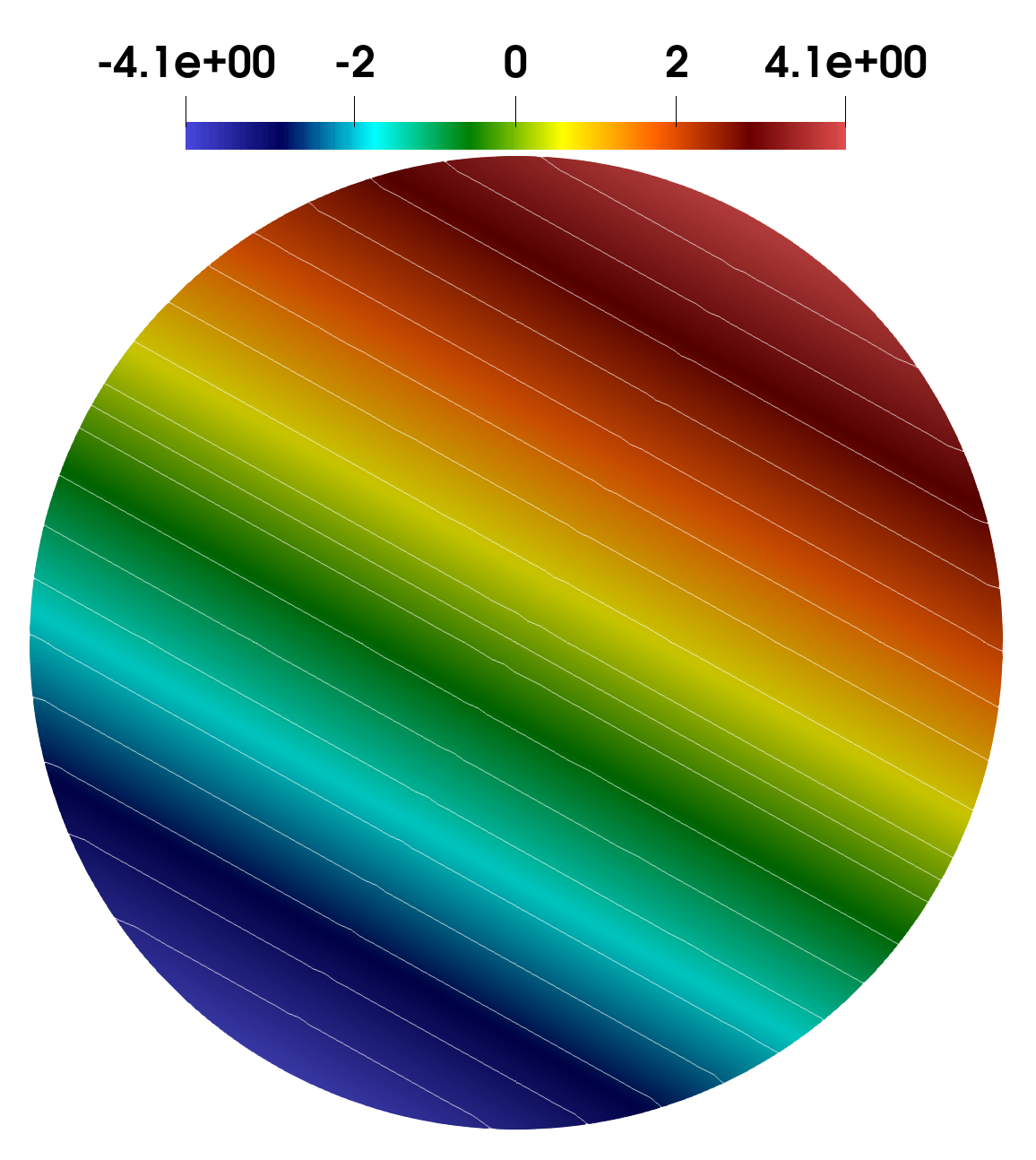}
	\end{minipage}
	\begin{minipage}{0.3\linewidth}
		\centering\includegraphics[scale=0.098]{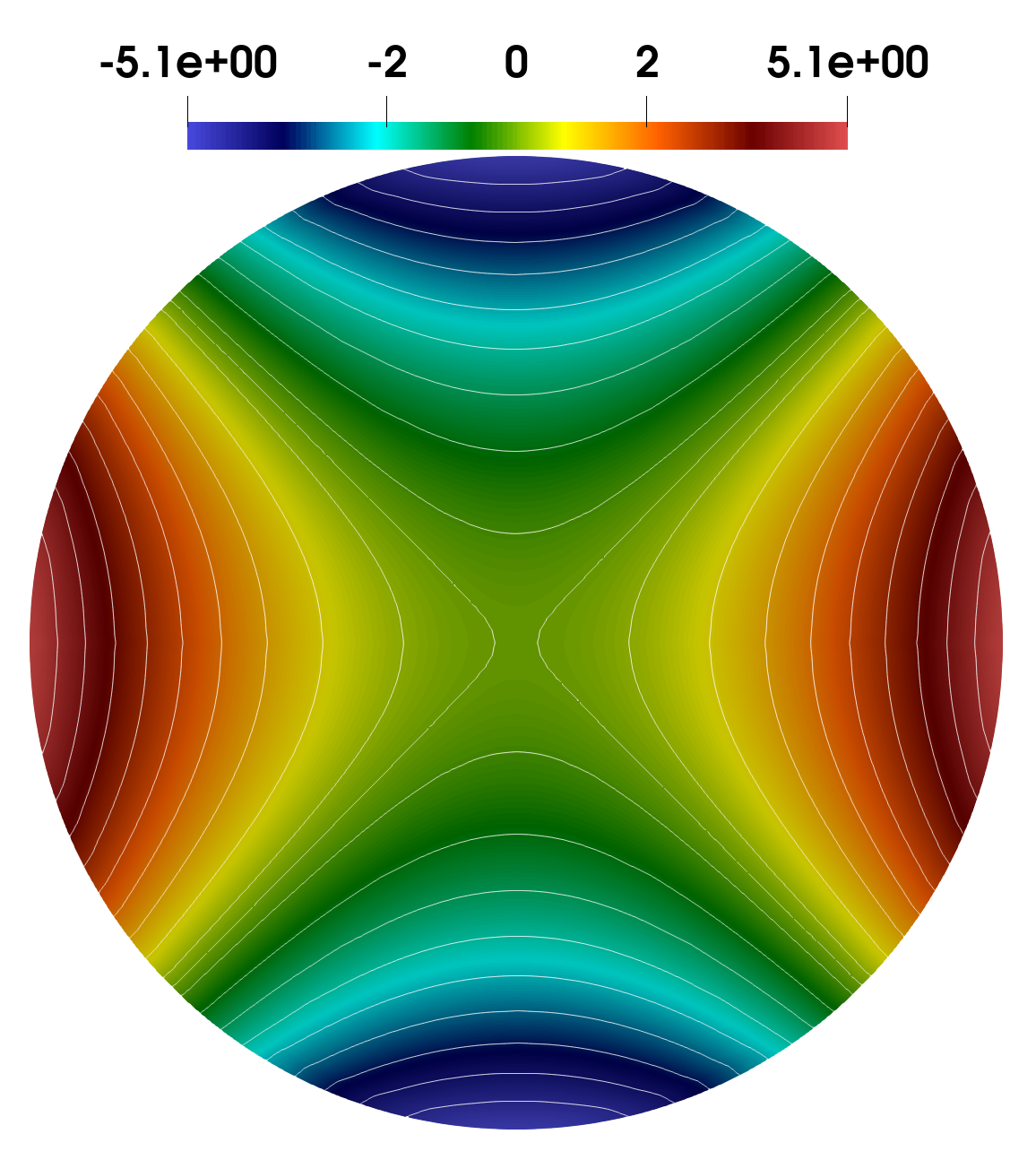}
	\end{minipage}
	\caption{Approximate velocity field $\bu_h$ (top row) and pressure $p_h$ (bottom row), corresponding to the first, third and fourth lowest eigenvalues in the unit circular domain.}
	\label{figura:modos-circulo}
\end{figure}

\begin{figure}[H]
	\begin{minipage}{0.49\linewidth}
		\centering\includegraphics[scale=0.42]{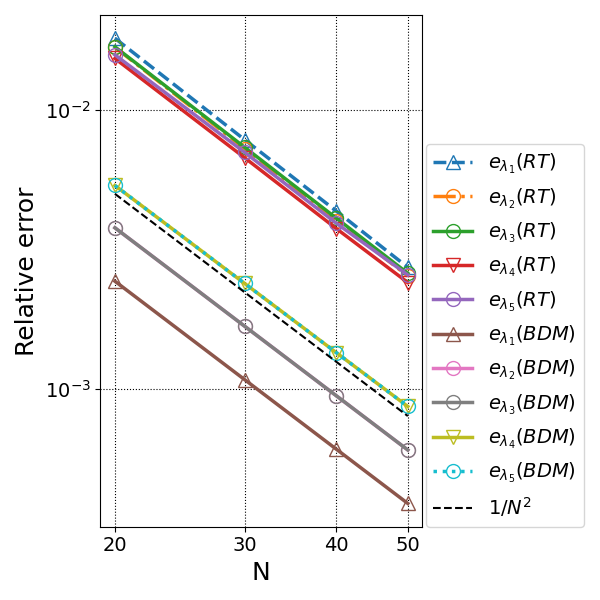}
	\end{minipage}
	\begin{minipage}{0.49\linewidth}
		\centering\includegraphics[scale=0.42]{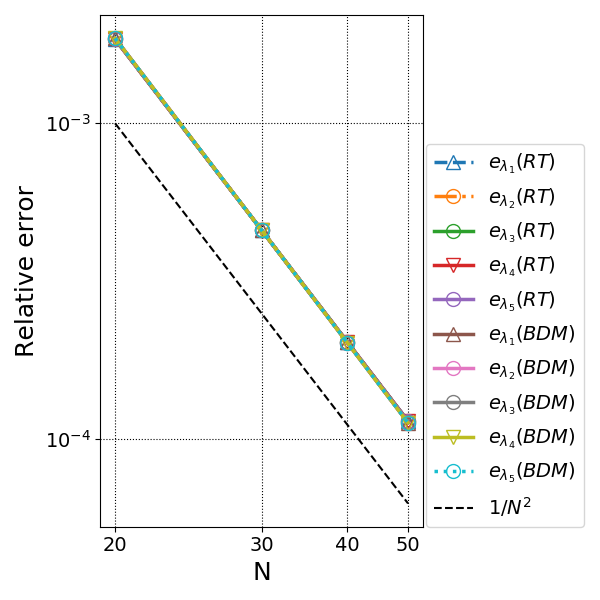}
	\end{minipage}
	\caption{Comparison of the eigenvalues error behavior in the circle when using $[\mathbb{P}_k]^{n}\text{-}\mathbb{P}_k\text{-}\mathbb{RT}_k$   and $[\mathbb{P}_k]^{n}\text{-}\mathbb{P}_k\text{-}\mathbb{BDM}_{k+1}$   schemes. The experiment considers polynomials of degree $k=0$ (left) and $k=2$ (right).}
	\label{figura:error-circulo}
\end{figure}

\subsection{Test 3. The L-Shape domain.} In this numerical test we consider an L-shape domain given by $\Omega_L:=(-1,1)\times(-1,1) \backslash [-1,0]\times [-1,0]$. In Table \ref{tabla:l-shape} we report the results when using $[\mathbb{P}_k]^{n}\text{-}\mathbb{P}_k\text{-}\mathbb{RT}_k$   and $[\mathbb{P}_k]^{n}\text{-}\mathbb{P}_k\text{-}\mathbb{BDM}_{k+1}$   schemes to solve the discrete eigenvalue problem. The table show the corresponding order of convergence together with the extrapolated values of the five lowest computed eigenvalues. Note that the singularity produced by the reentrant corner yields to a rate of convergence around 1.7 (see \cite{MR2473688} for instance), as can be seen in the lowest computed eigenvalue. In fact, we observe that the order of convergence is $s\approx 2\min\{r,k+1\}$, with $1.7\leq r\leq 2$, as is predictable in this geometry. For better visualization, we explore this result in the relative error plots in Figure \ref{figura:error-lshape}, where the slopes are compared with the $1/N^{\sqrt{3}}$, which is the best order possible with uniform refinement.  

We end this test reporting plots of  the velocity fields and pressure fluctuations in Figure \ref{figura:modos-l-shape}, where, as is expectable, high gradients around the singularity are observed.

 \begin{table}[H]
	{\footnotesize
		\begin{center}
			\begin{tabular}{c|c c c c |c| c }
				\toprule
				Scheme&$N=9$             &  $N=15$         &   $N=20$         & $N=35$ & Order & $\lambda_{extr}$  \\ 
				\midrule
				\multirow{5}{0.19\linewidth}{$[\mathbb{P}_k]^{n}\text{-}\mathbb{P}_k\text{-}\mathbb{RT}_k$  }&29.43565 &  30.83700 &  31.16193 &  31.62598  & 1.59 & 31.89457\\
				&34.98077 &  36.28132 &  36.50660 &  36.83669  & 2.03 & 36.94231\\
				&40.70064 &  41.43833 &  41.62290 &  41.83014  & 1.73 & 41.94524\\
				&46.83830 &  48.22776 &  48.47328 &  48.80875  & 2.07 & 48.91635\\
				&52.08483 &  53.96541 &  54.48404 &  55.02474  & 1.65 & 55.37238\\
				\midrule
				\multirow{5}{0.19\linewidth}{$[\mathbb{P}_k]^{n}\text{-}\mathbb{P}_k\text{-}\mathbb{BDM}_{k+1}$  }&32.59542 &  32.24970 &  32.14635 &  32.06144  & 1.75 & 32.00483\\
				&38.76953 &  37.57884 &  37.32081 &  37.11240  & 2.26 & 37.03276\\
				&44.76985 &  42.88018 &  42.46067 &  42.10765  & 2.19 & 41.96744\\
				&52.09587 &  50.19205 & 49.67367 & 49.20827  & 1.81 & 48.93475\\
				&58.84979 &  56.72442 &  56.20364 &  55.63553  & 1.79 & 55.33628\\
				\bottomrule             
			\end{tabular}
	\end{center}}
	\caption{Lowest computed eigenvalues for polynomial degrees $k=0$ in the L-shape domain.}
	\label{tabla:l-shape}
\end{table}
\begin{figure}[H]
	\centering\includegraphics[scale=0.4]{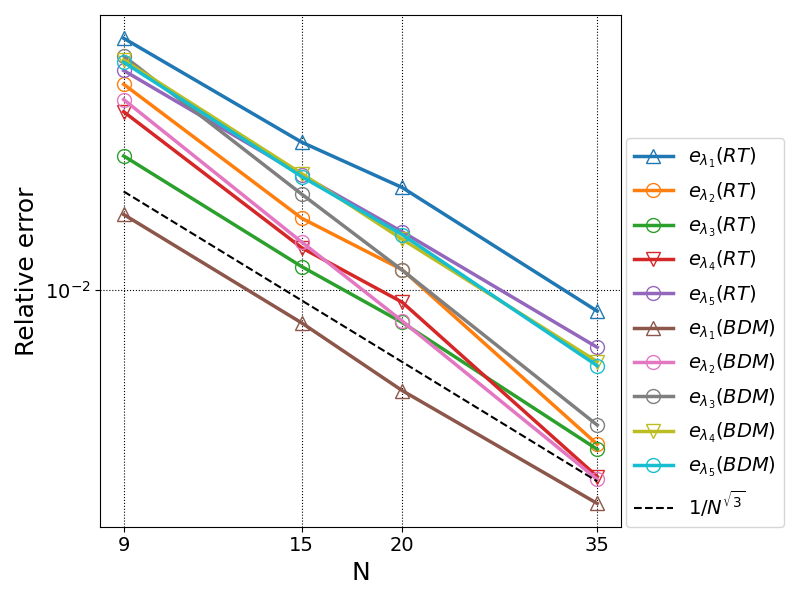}
	\caption{Comparison of the eigenvalues error behavior in the L-shape domain when using $[\mathbb{P}_k]^{n}\text{-}\mathbb{P}_k\text{-}\mathbb{RT}_k$   and $[\mathbb{P}_k]^{n}\text{-}\mathbb{P}_k\text{-}\mathbb{BDM}_{k+1}$   schemes. The experiment considers polynomials of degree $k=0$.}
	\label{figura:error-lshape}
\end{figure}
\begin{figure}[H]
	\begin{minipage}{0.3\linewidth}
		\centering\includegraphics[scale=0.091]{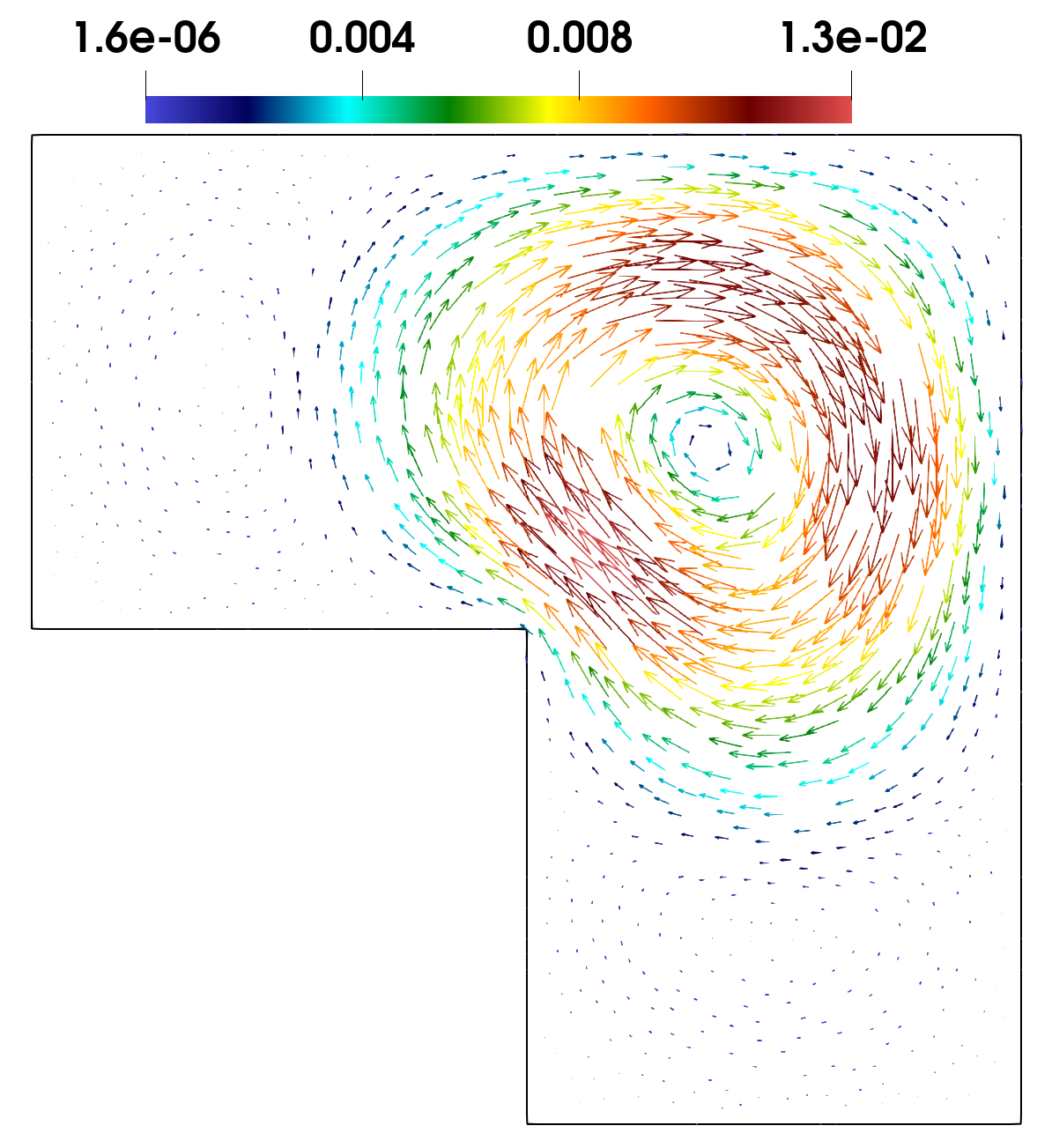}
	\end{minipage}
	\begin{minipage}{0.3\linewidth}
		\centering\includegraphics[scale=0.093]{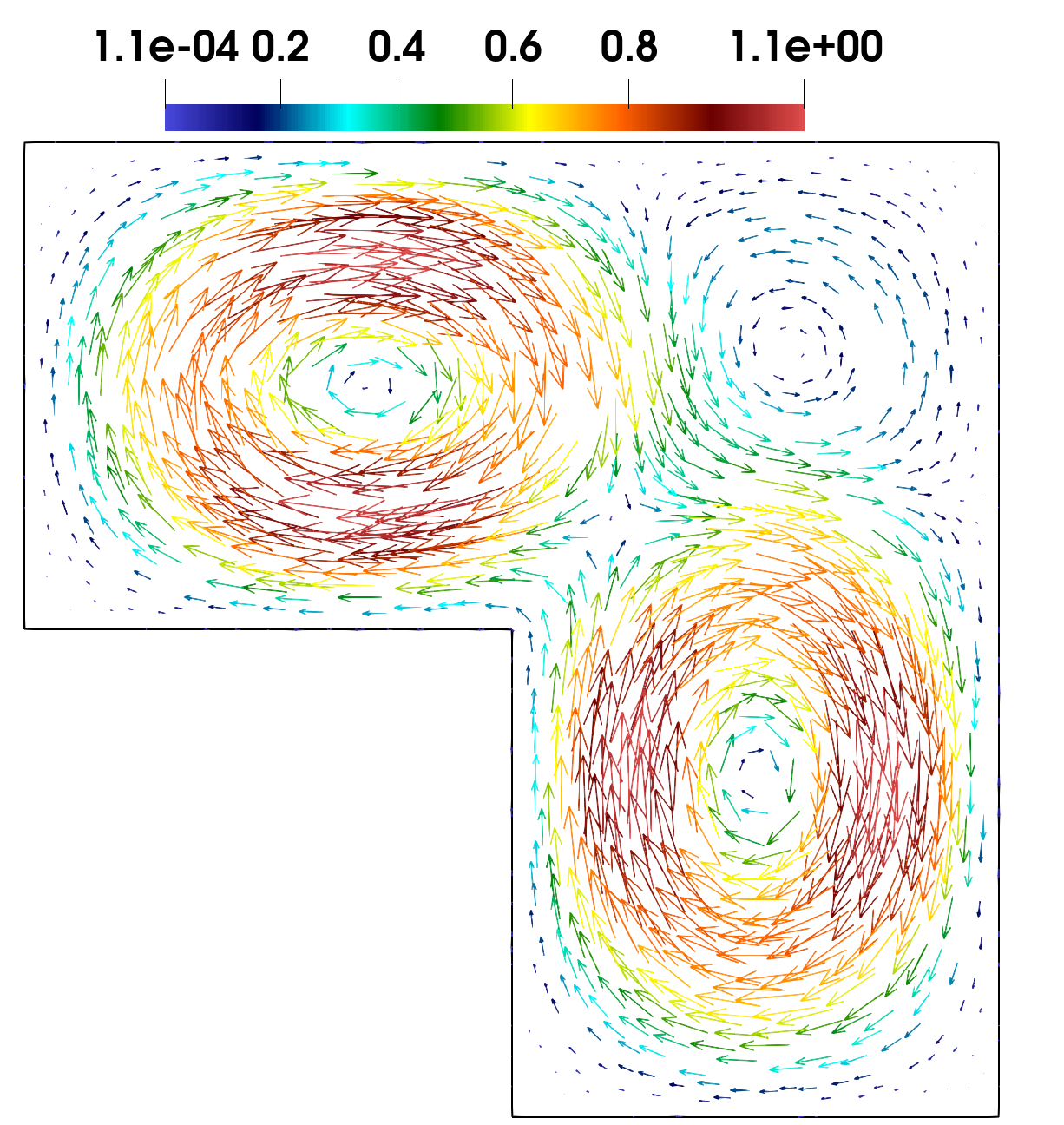}
	\end{minipage}
	\begin{minipage}{0.3\linewidth}
		\centering\includegraphics[scale=0.093]{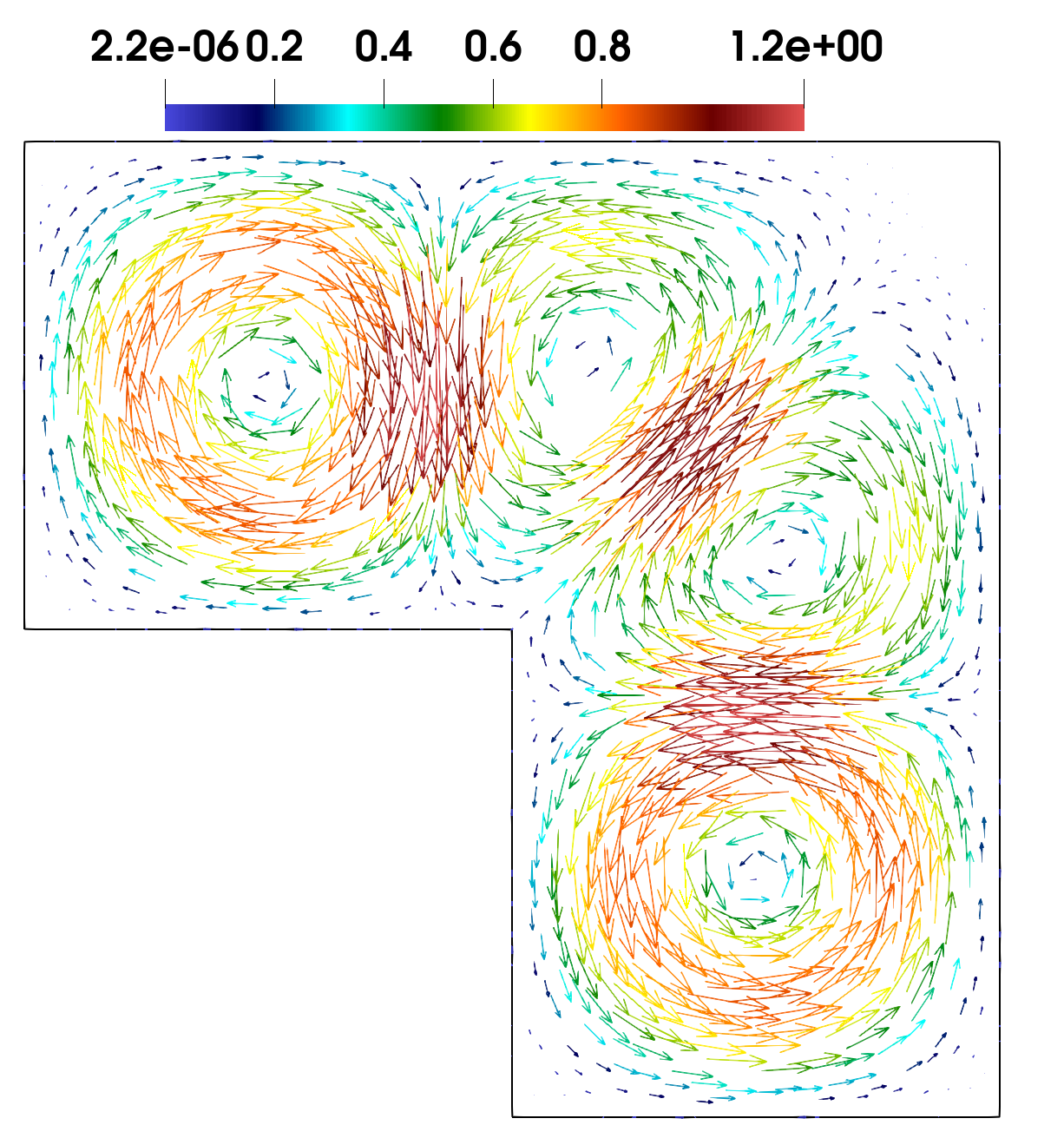}
	\end{minipage}\\
	\begin{minipage}{0.3\linewidth}
		\centering\includegraphics[scale=0.095]{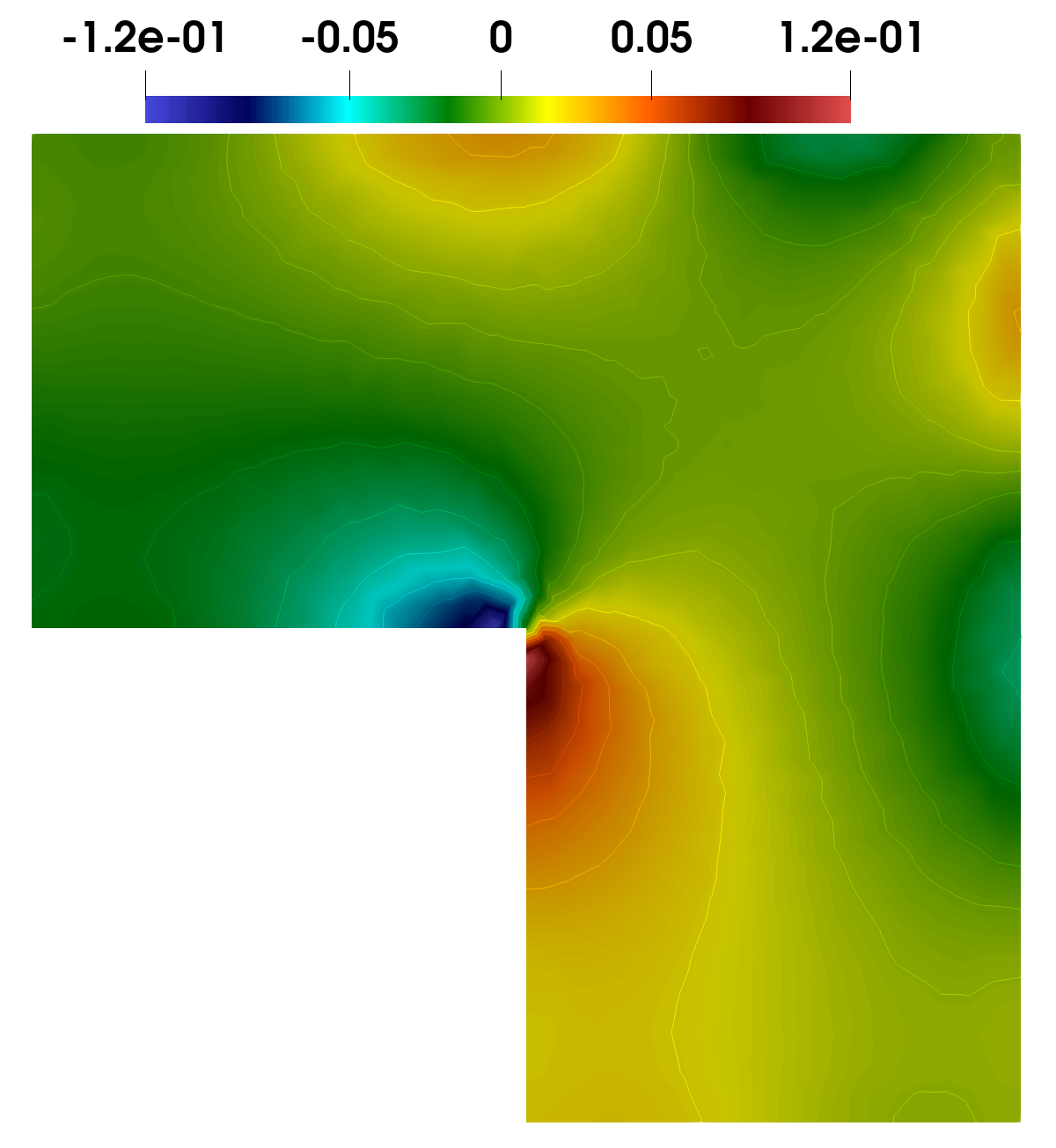}
	\end{minipage}
	\begin{minipage}{0.3\linewidth}
		\centering\includegraphics[scale=0.098]{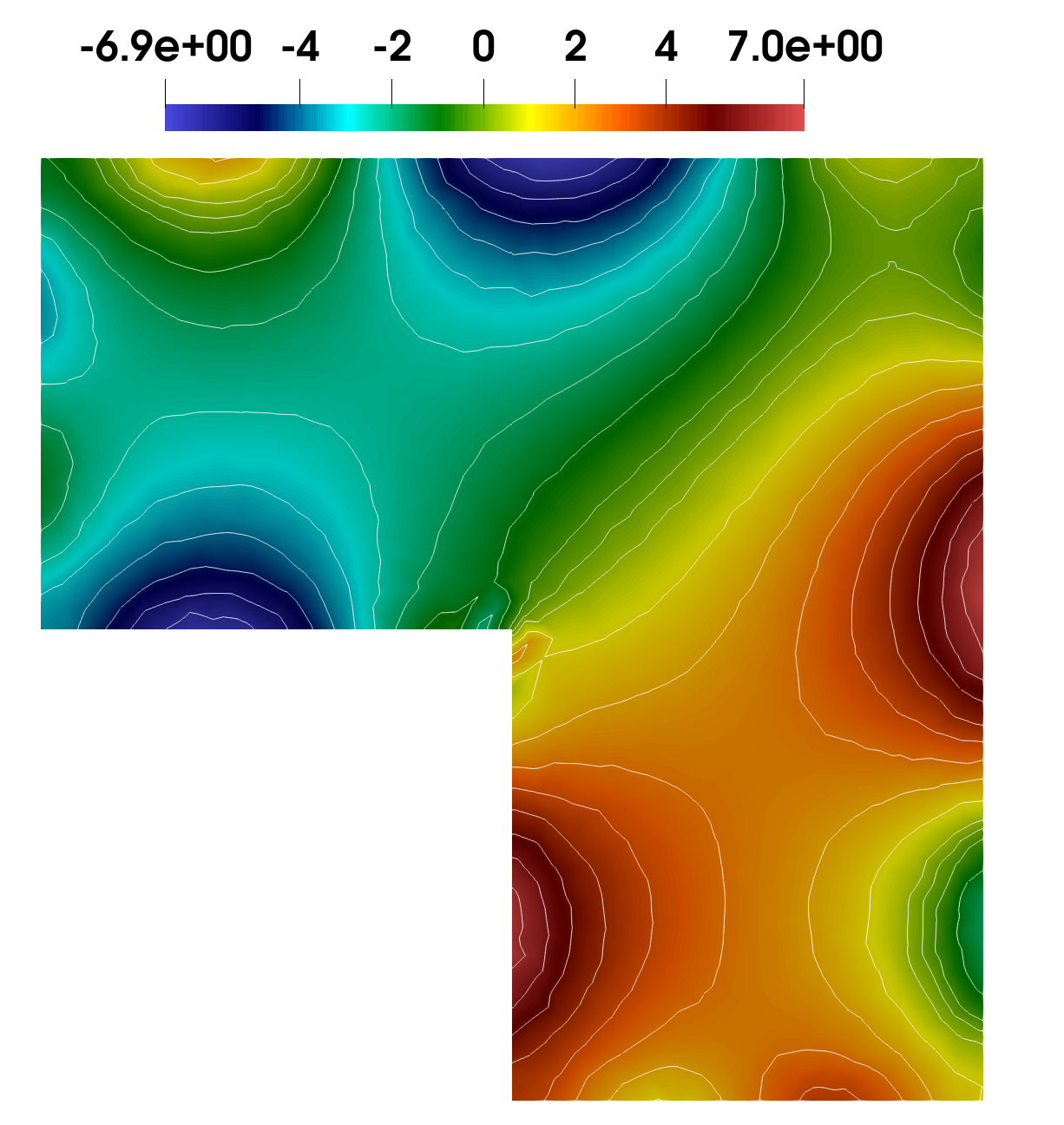}
	\end{minipage}
	\begin{minipage}{0.3\linewidth}
		\centering\includegraphics[scale=0.098]{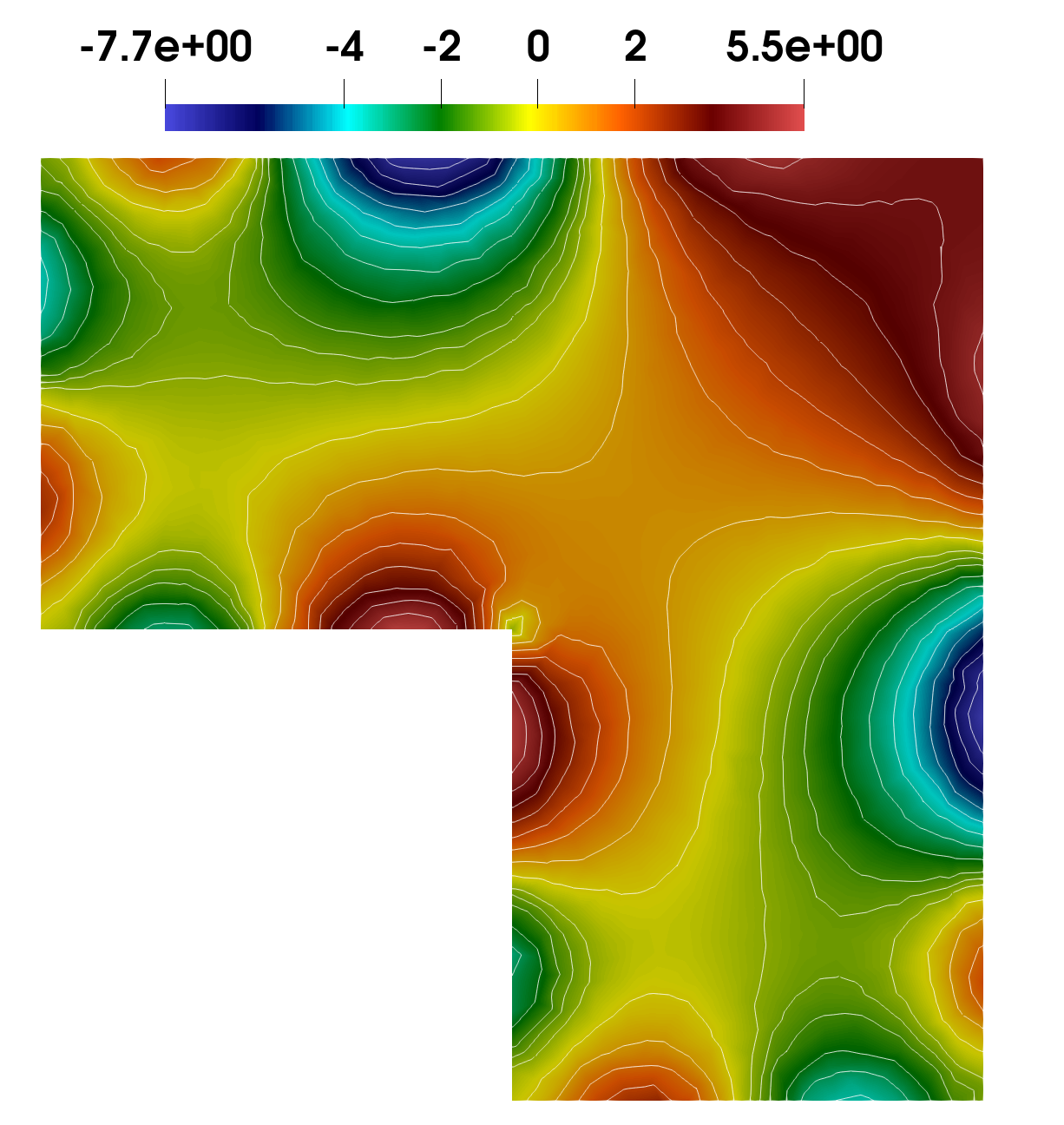}
	\end{minipage}
	\caption{Approximate velocity field $\bu_h$ (top row) and pressure $p_h$ (bottom row), corresponding to the first, third and fourth lowest eigenvalues in the L-shape domain.}
	\label{figura:modos-l-shape}
\end{figure}

\subsection{3-D test: Cubic and spherical domain.} In this test we further assess the proposed schemes by consider two different three dimensional scenarios. The first case considers cube in the region $\Omega=(0,1)^3$. Here, $N$ represents the number of cell per side such that the number of tetrahedron is $6(N+1)^3$. In Figure \ref{FIG:meshescube}
we present examples of the meshes used in the cube domain.
\begin{figure}[H]
	\begin{center}
		\begin{minipage}{14cm}
			\hspace{-1.1 cm}\centering\includegraphics[height=4.4cm, width=4.4cm]{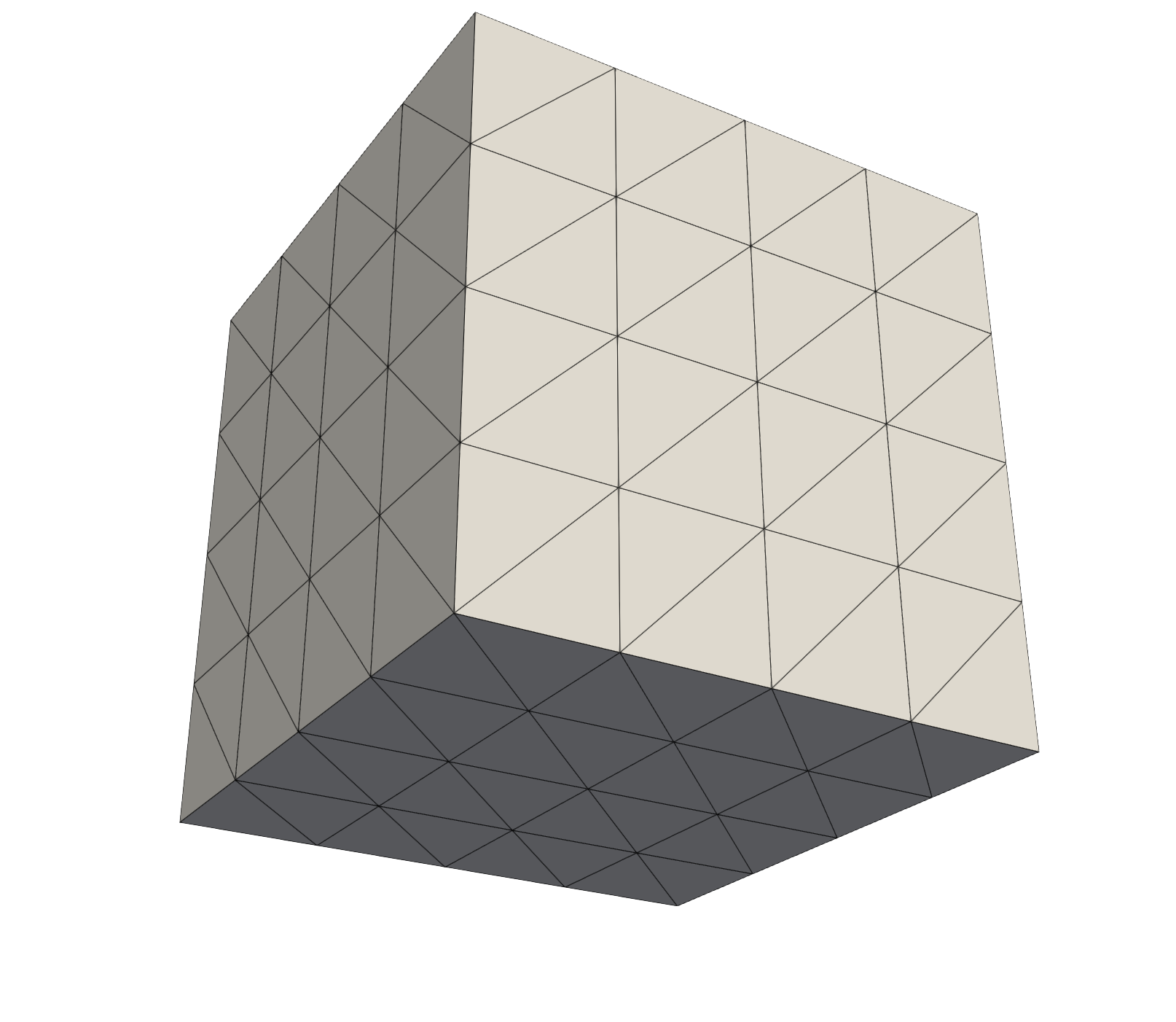}
			\centering\includegraphics[height=4.4cm, width=4.4cm]{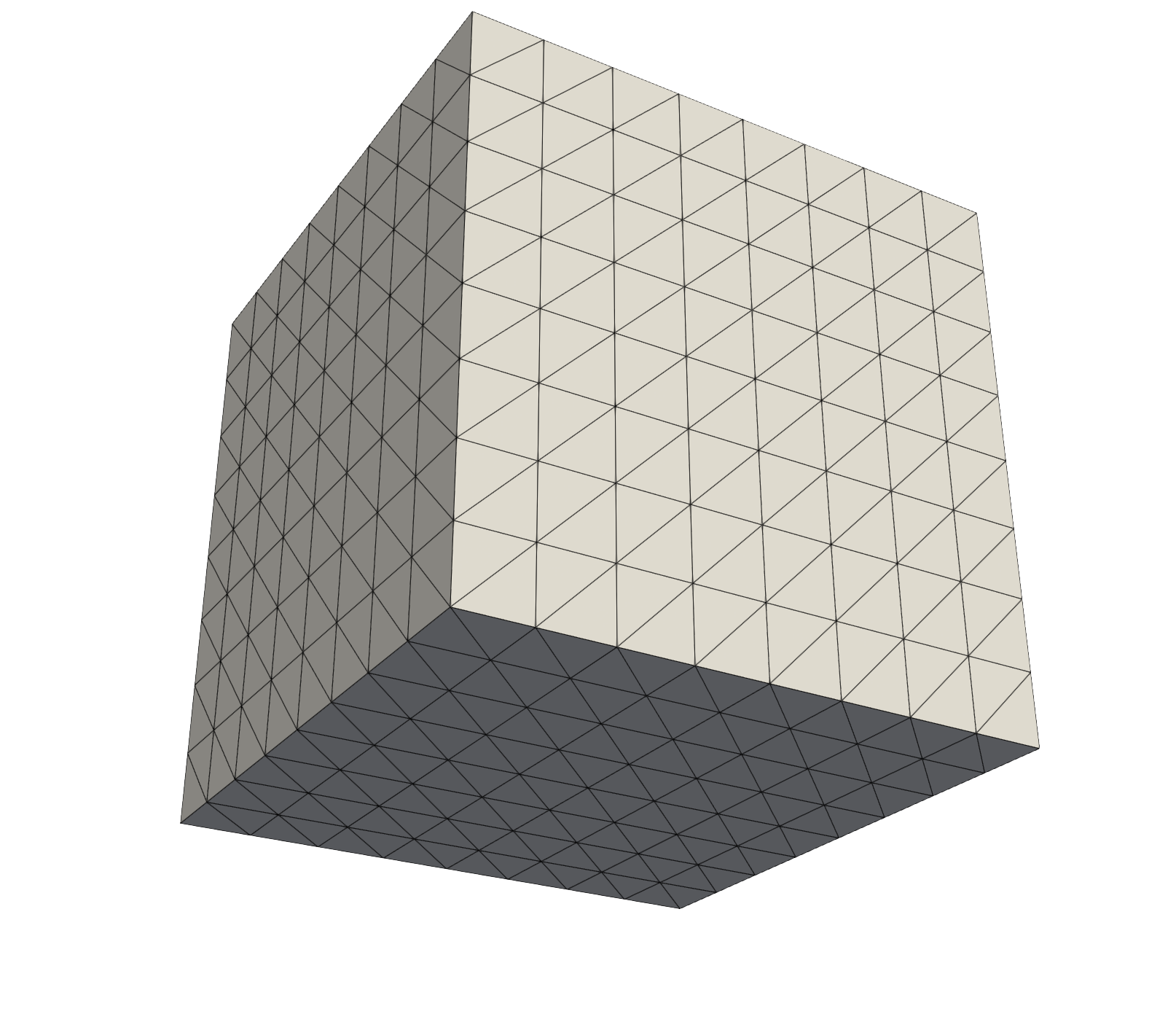}
			\centering\includegraphics[height=4.4cm, width=4.4cm]{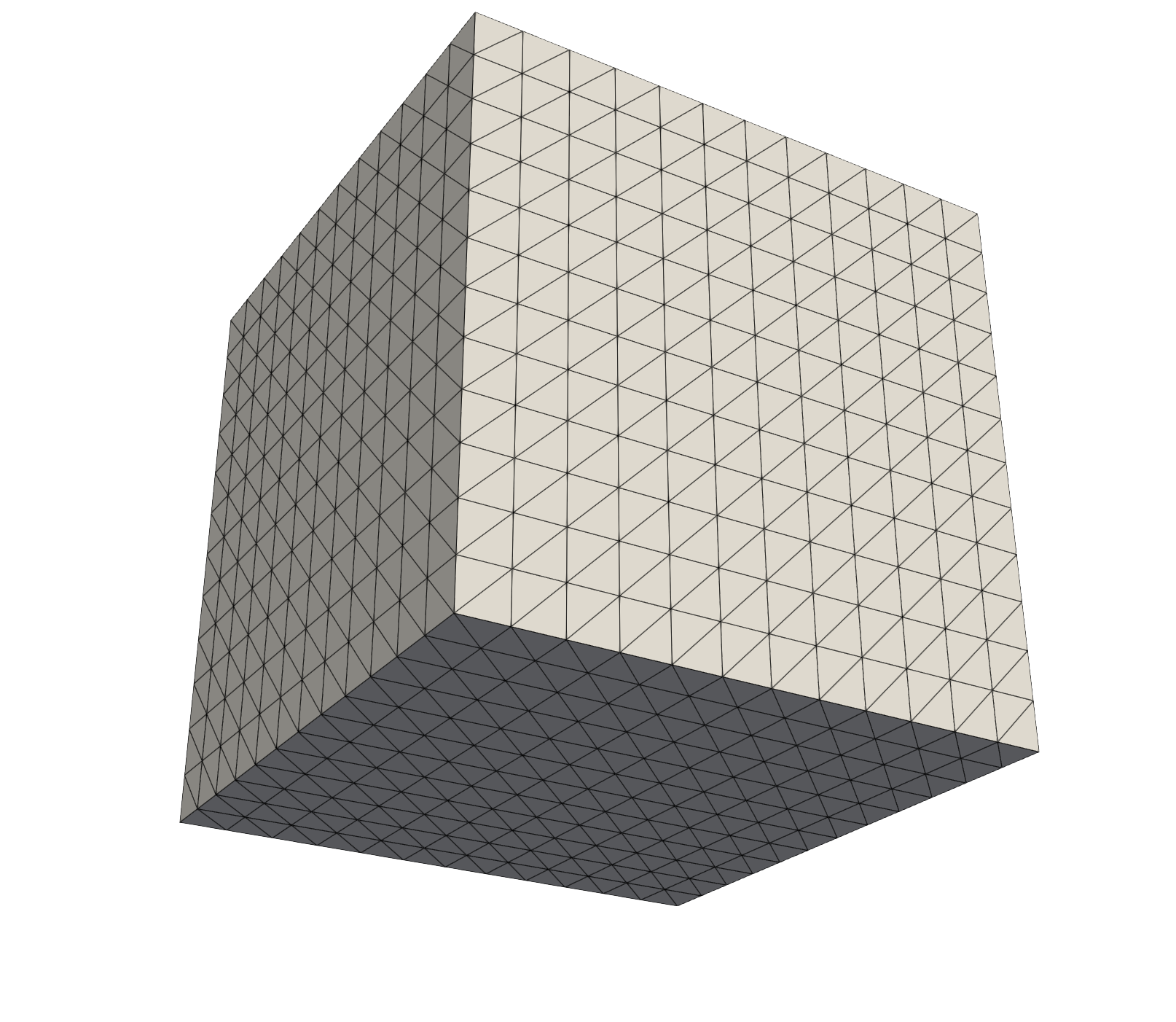}
                   \end{minipage}
		\caption{Examples of the meshes used in the unit cube. The left figure represents a mesh for $N=4$, the figure in for $N=8$  and the right figure for $N=12$.}
		\label{FIG:meshescube}
	\end{center}
\end{figure}

In the second scenario we consider the unitary sphere $$\O_S:=\{(x,y,z)\in\mathbb{R}^3\,:\, x^2+y^2+z^2\leq 1\}.$$ For this case, $N$ represents the mesh resolution such that $N\sim 1/h$. We remark that this test consists into approximate a curved domain by means of  tetrahedral meshes. In Figure \ref{FIG:meshessphere} we present several meshes used in the experiment.
\begin{figure}[H]
		\begin{minipage}{0.32\linewidth}
				\centering\includegraphics[scale=0.075]{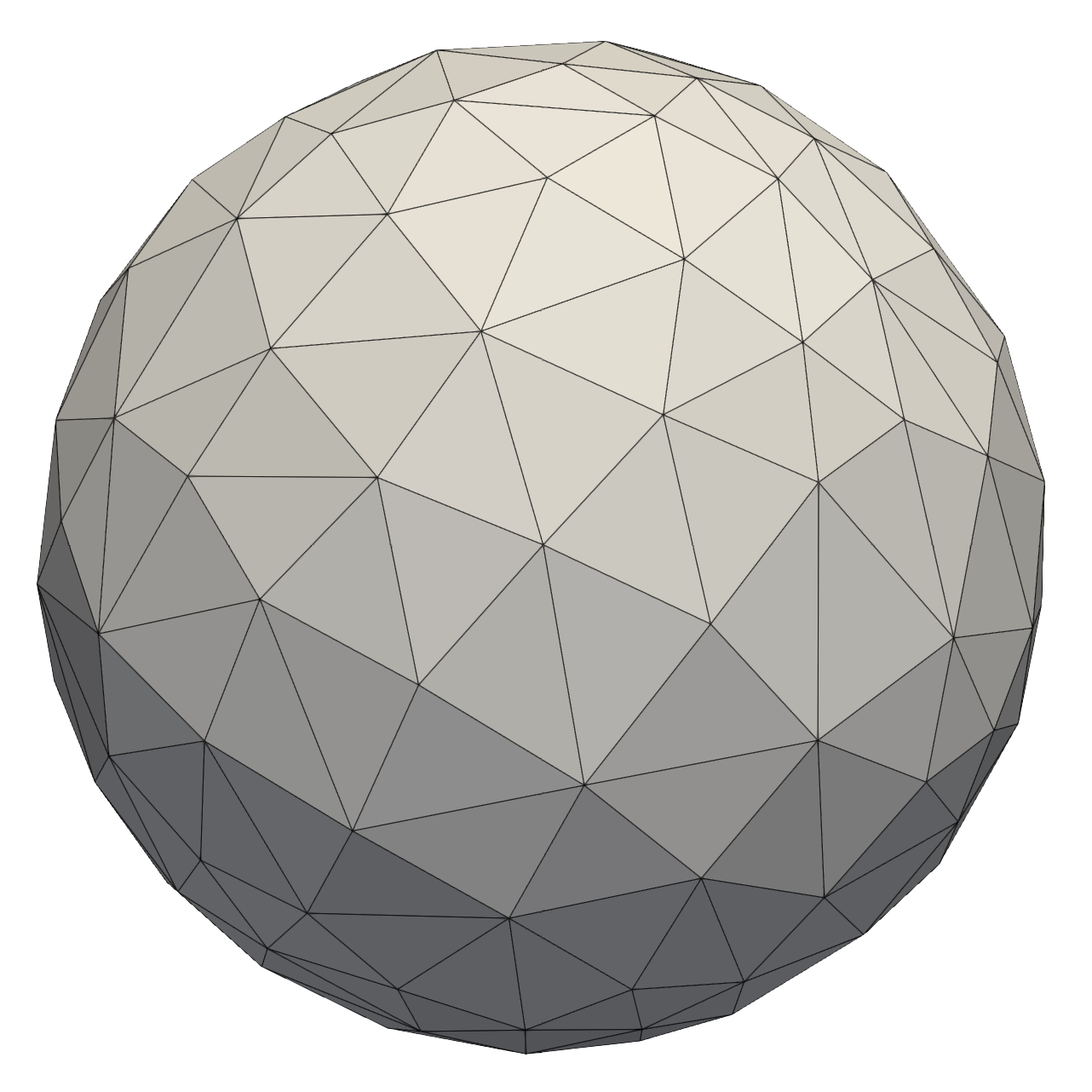}
			\end{minipage}
		\begin{minipage}{0.32\linewidth}
				\centering\includegraphics[scale=0.075]{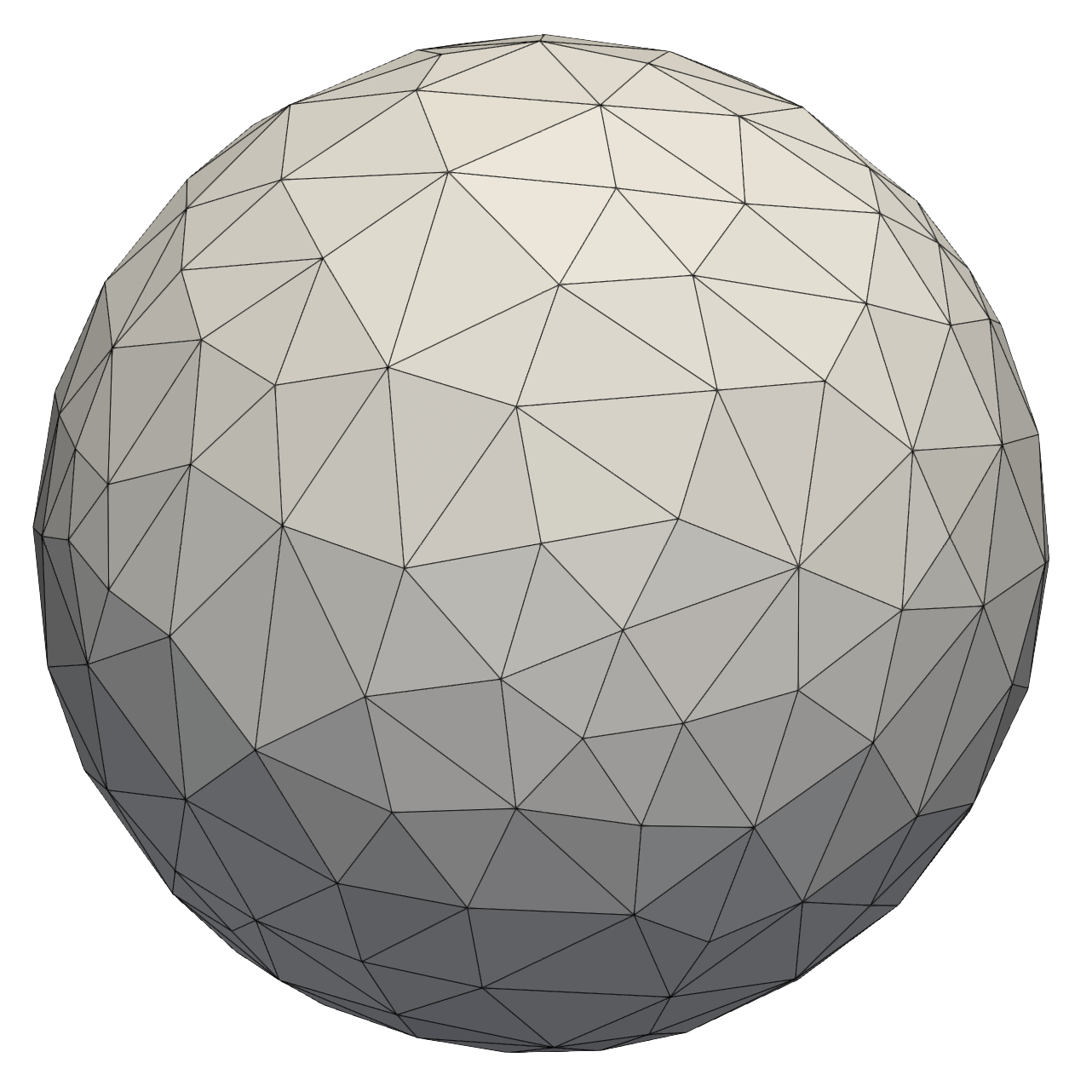}
			\end{minipage}
		\begin{minipage}{0.32\linewidth}
			\centering\includegraphics[scale=0.075]{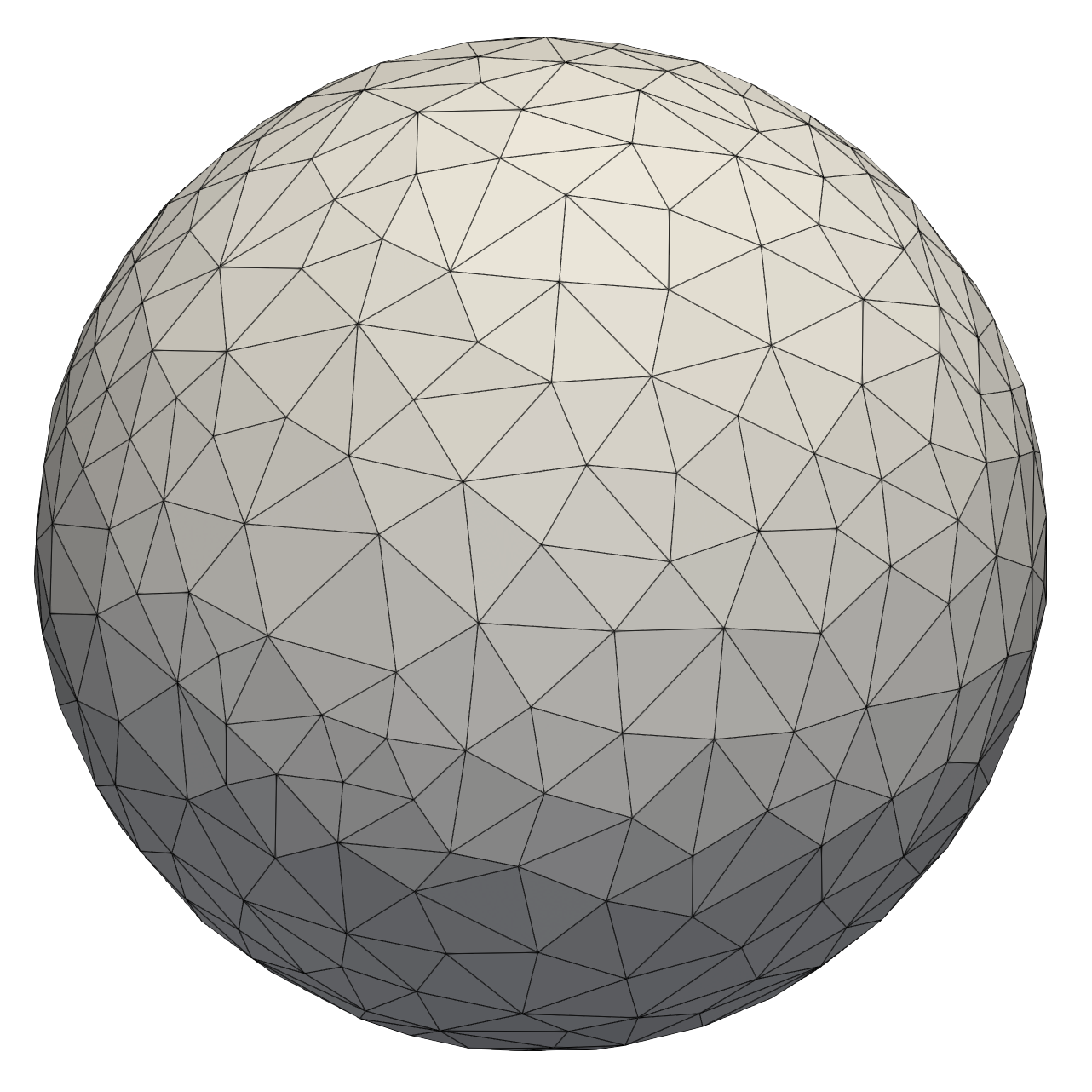}
		\end{minipage}
			\caption{Examples of the meshes used in the unitary sphere. The left figure represents a mesh for $N=8$, the middle figure for $N=10$  and the right figure for $N=14$.}
			\label{FIG:meshessphere}
\end{figure}
For simplicity, both scenarios only consider the lowest order of approximation ($k=0$). 

For the unit cube, the results from using the $[\mathbb{P}_k]^{n}\text{-}\mathbb{P}_k\text{-}\mathbb{RT}_k$   and $[\mathbb{P}_k]^{n}\text{-}\mathbb{P}_k\text{-}\mathbb{BDM}_{k+1}$ numerical schemes are reported in Table \ref{tabla:cube}. In figure \ref{figura:erro-cubo} we present the relative error plot for the approximated eigenvalues compared with the extrapolated ones on each table. It notes that the expected rate of convergence is observed. Together with this, the second and fourth lowest eigenfunctions corresponding to $\bu_h$ and $p_h$ are depicted in Figure \ref{figura:modos-cubo}.

On the other hand, the unit sphere case is described in Table \ref{tabla:sphere}. Here, we show that the proposed methods work perfectly and deliver the expected double order of convergence for both schemes, which is observed in Figure \ref{figura:erro-esfera}. For completeness, in Figure \ref{figura:modos-esfera} we present plots of the approximated velocity fields and pressure fluctuations associated with the second and fourth lowest eigenvalues. As in the circular domain, we observe an equally pressure distribution on the boundary for the second eigenvalue.

 \begin{table}[H]
	{\footnotesize
		\begin{center}
			\begin{tabular}{c| c c c c |c| c }
				\toprule
				Scheme&$N=6$             &  $N=8$         &   $N=10$         & $N=12$ & Order & $\lambda_{extr}$  \\ 
				\midrule
				\multirow{5}{0.19\linewidth}{$[\mathbb{P}_k]^{n}\text{-}\mathbb{P}_k\text{-}\mathbb{RT}_k$  }
				&62.35998&  62.23777&  62.19632 & 62.18064 & 3.29 & 62.15973\\
				&62.69715&  62.42145&  62.31185 & 62.25999 & 2.60 & 62.17340\\
				&62.69715&  62.42145&  62.31185 & 62.25999 & 2.60 & 62.17340\\
				&84.69821&  89.79683&  91.41328 & 91.49049 & 4.33 & 91.99133\\
				&89.14796&  91.13378&  92.15949 & 91.49049 & 5.43 & 91.86755\\
				\midrule
				\multirow{5}{0.19\linewidth}{$[\mathbb{P}_k]^{n}\text{-}\mathbb{P}_k\text{-}\mathbb{BDM}_{k+1}$  }
				&65.90006 &  64.32034 & 63.56386 & 63.14558  & 1.88 & 62.11709\\
				&66.63208 &  64.74326 & 63.83797 & 63.33726  & 1.88 & 62.10739\\
				&66.63208 &  64.74326 & 63.83797 & 63.33726  & 1.88 & 62.10739\\
				&100.1074&  96.65063 & 94.93243 & 93.96083  & 1.75 & 91.35751\\
				&100.1074&  96.65063 & 94.93243 & 93.96083  & 1.75 & 91.35751\\
				\bottomrule                  
			\end{tabular}
	\end{center}}
	\caption{Lowest computed eigenvalues for polynomial degrees $k=0$ in the unit cube domain.}
	\label{tabla:cube}
\end{table}

\begin{figure}[H]
	\centering\includegraphics[scale=0.4]{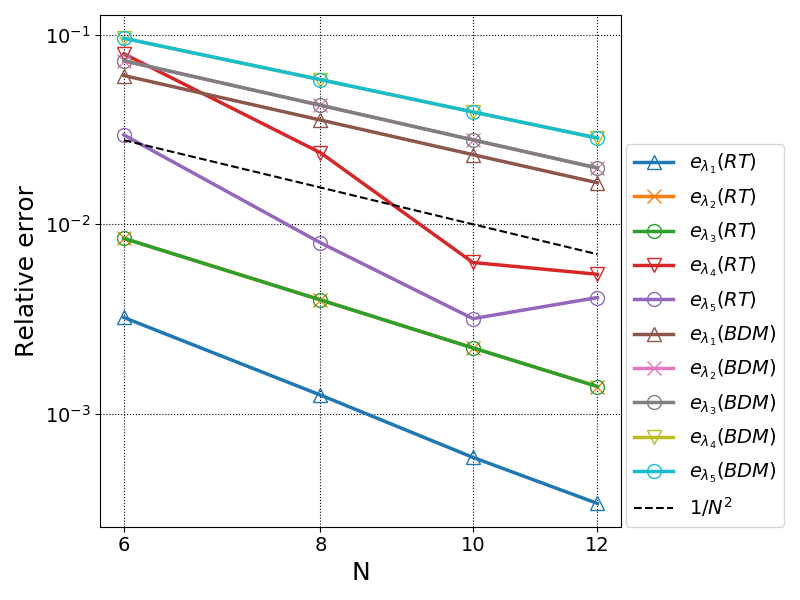}
	\caption{Comparison of the relative error behavior in the unit cube when using $[\mathbb{P}_k]^{n}\text{-}\mathbb{P}_k\text{-}\mathbb{RT}_k$   and $[\mathbb{P}_k]^{n}\text{-}\mathbb{P}_k\text{-}\mathbb{BDM}_{k+1}$   schemes. The experiment considers polynomials of degree $k=0$. }
	\label{figura:erro-cubo}
	\end{figure}


\begin{figure}[H]
	\begin{minipage}{0.45\linewidth}
		\centering\includegraphics[scale=0.1]{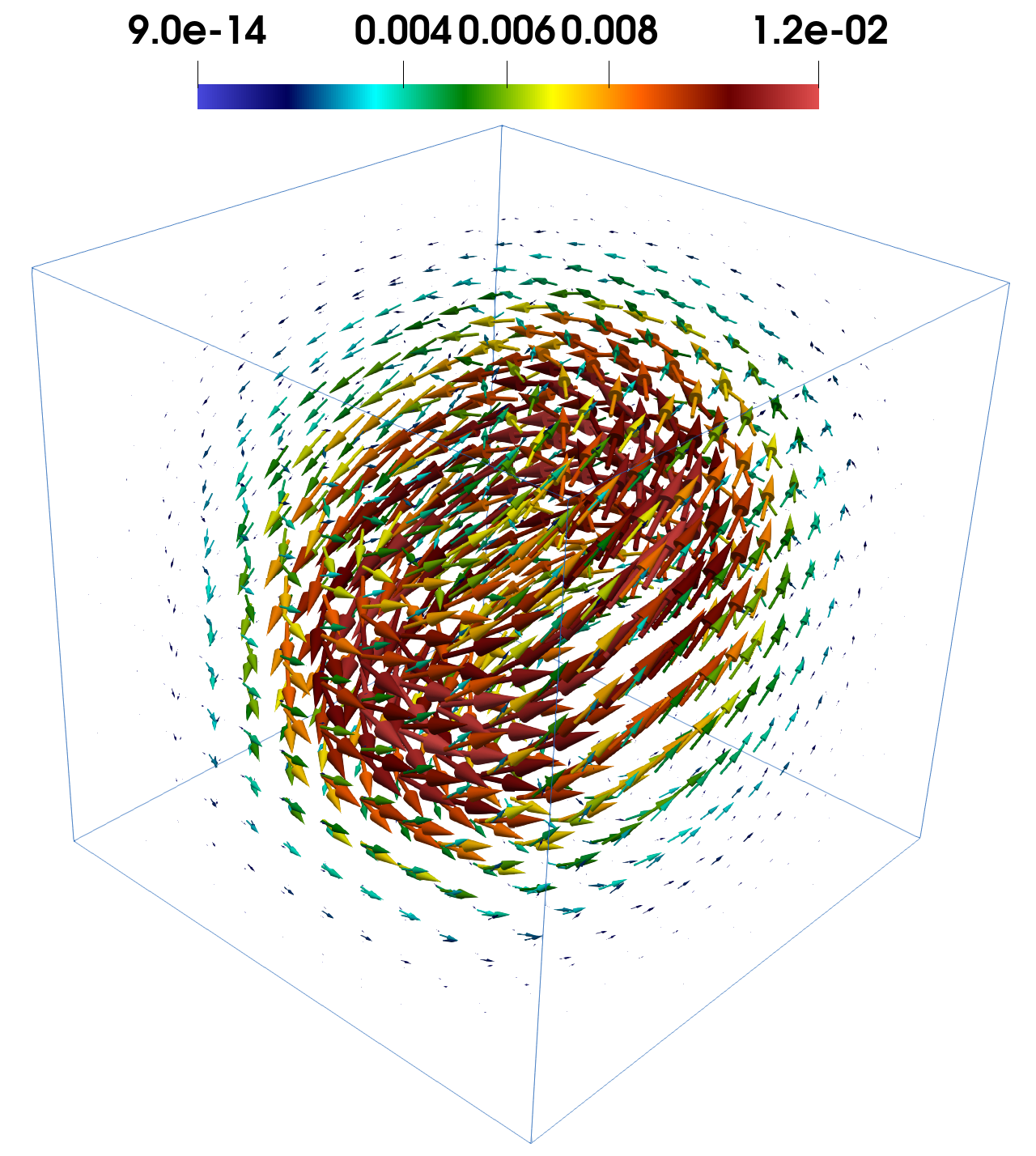}
	\end{minipage}
	\begin{minipage}{0.45\linewidth}
		\centering\includegraphics[scale=0.1]{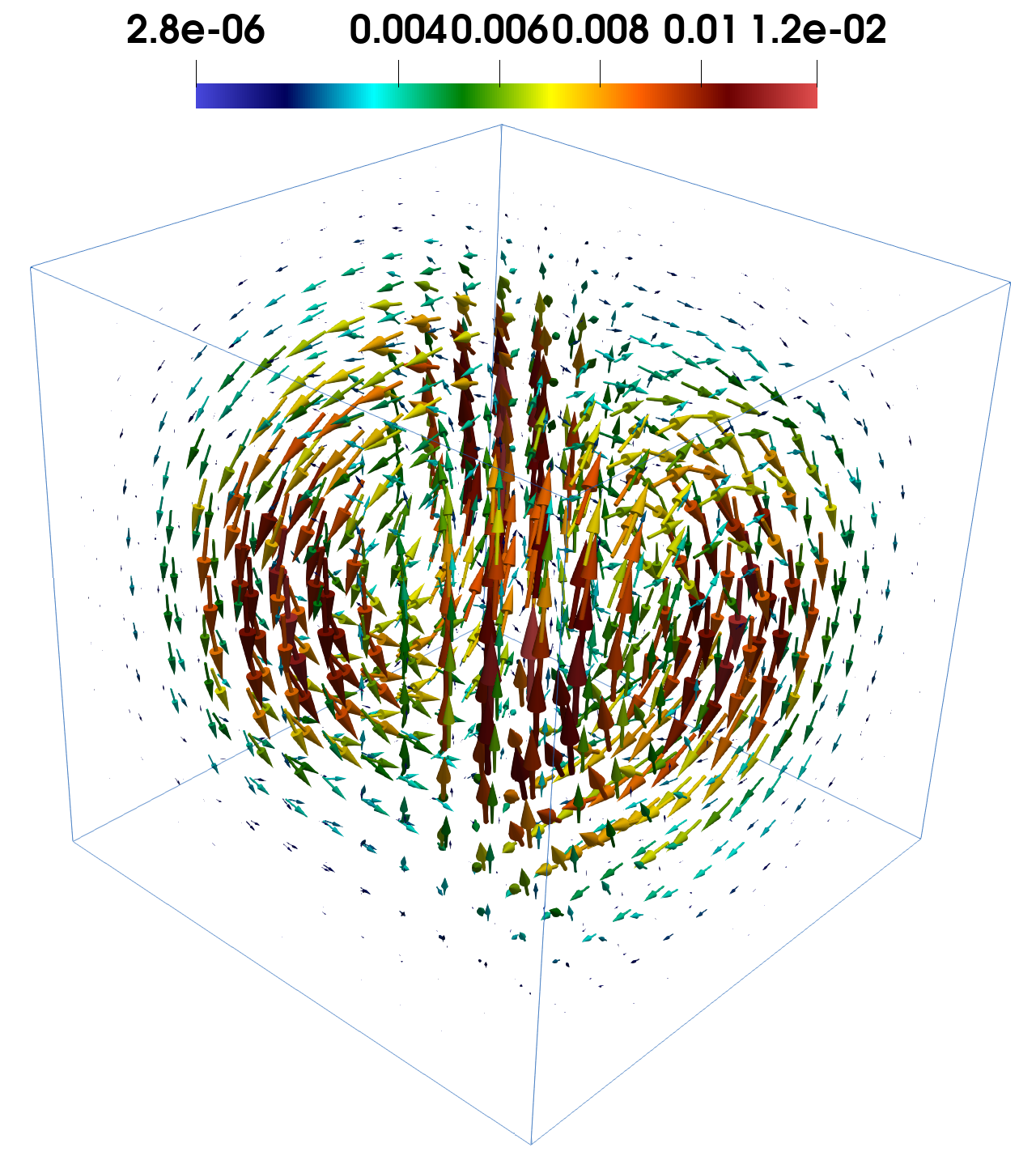}
	\end{minipage}\\
	\begin{minipage}{0.45\linewidth}
	\centering\includegraphics[scale=0.1]{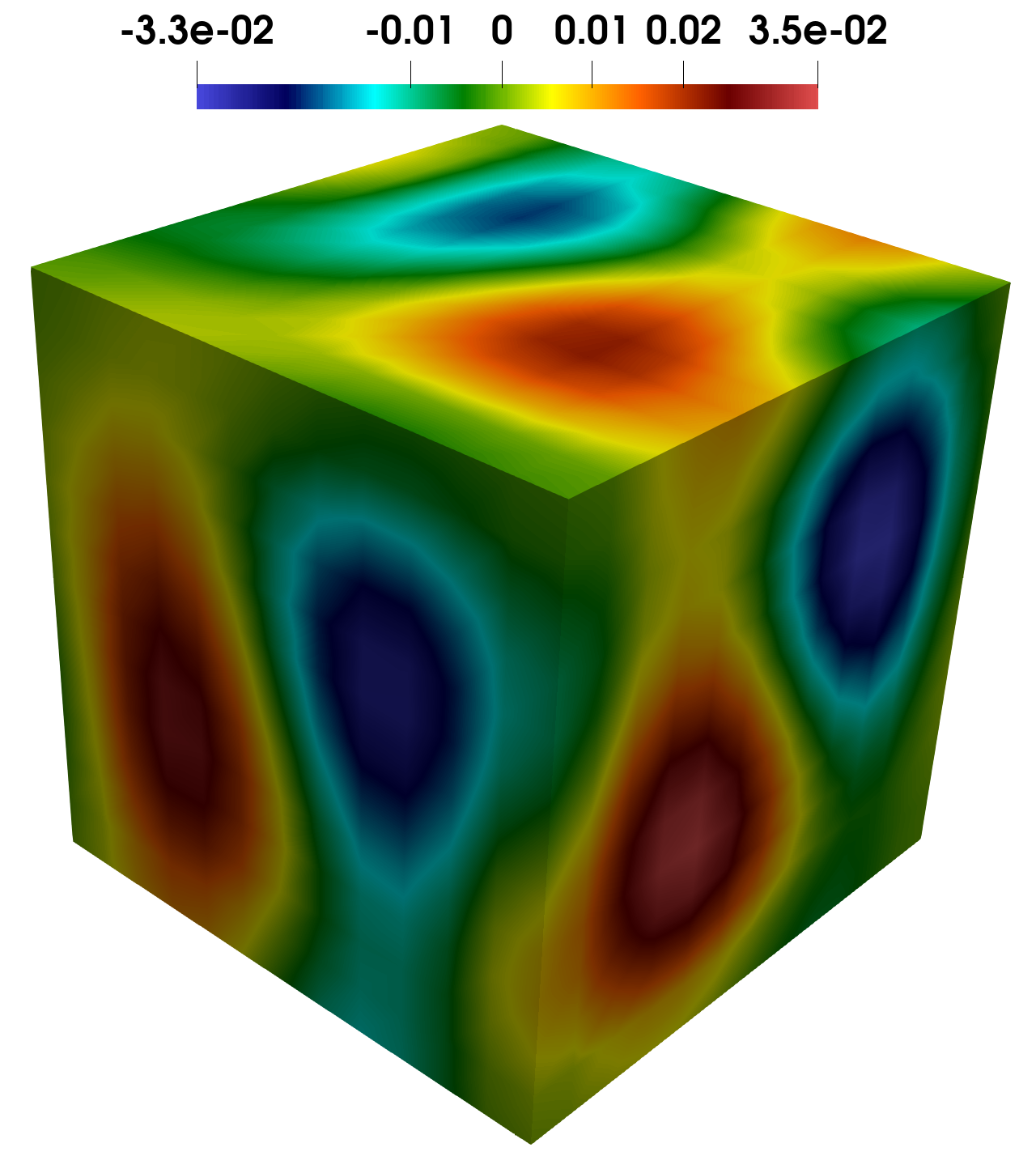}
	\end{minipage}
	\begin{minipage}{0.45\linewidth}
	\centering\includegraphics[scale=0.1]{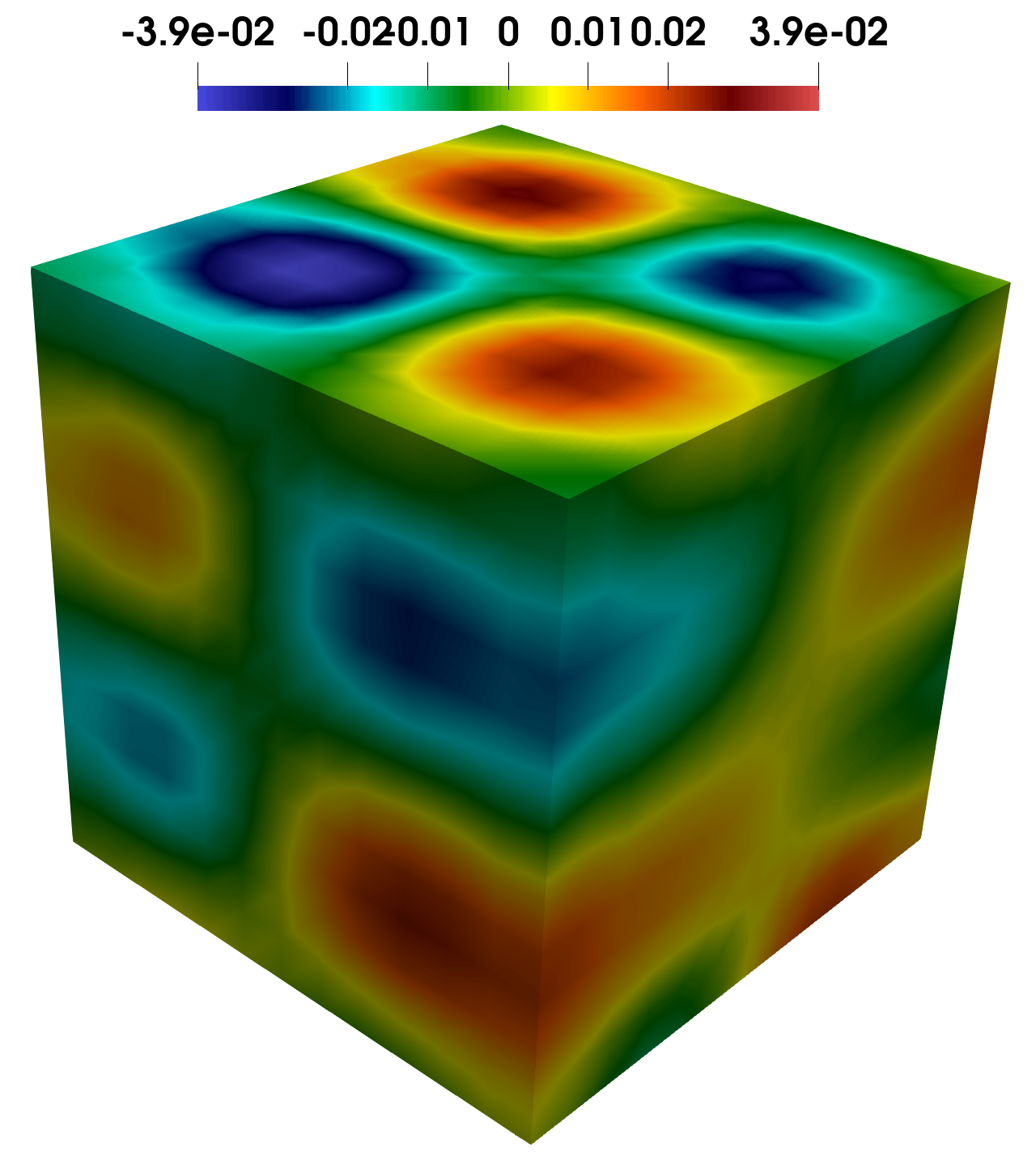}
	\end{minipage}
	\caption{Approximated velocity field  $\bu_h$ (top) and pressure $p$ (bottom), corresponding to the second and fourth lowest eigenvalues in the unit cube domain.}
	\label{figura:modos-cubo}
\end{figure}

 \begin{table}[H]
{\footnotesize
\begin{center}
\begin{tabular}{c|c c c c |c| c }
\toprule
     Scheme&$N=8$             &  $N=10$         &   $N=12$         & $N=14$ & Order & $\lambda_{extr}$  \\ 
	 \midrule
	\multirow{5}{0.19\linewidth}{$[\mathbb{P}_k]^{n}\text{-}\mathbb{P}_k\text{-}\mathbb{RT}_k$  }&21.20335&  20.92074&  20.75254&  20.65936 & 1.83&  20.35082\\
	&21.21249&  20.93014&  20.75649&  20.66213  & 1.76 & 20.32925\\
	&21.26202&  20.93587&  20.76475&  20.66667 &  2.21& 20.42274\\
	&32.88974&  33.16111 & 33.37782 & 33.40790  & 2.30 & 33.63347\\
	&32.92827&  33.19610&  33.38327&  33.42272 &  2.43&  33.61269\\
	\midrule
	\multirow{5}{0.19\linewidth}{$[\mathbb{P}_k]^{n}\text{-}\mathbb{P}_k\text{-}\mathbb{BDM}_{k+1}$  }&22.25623&  21.60098&  21.21005&  20.99626 & 1.85&  20.29349\\
	&22.26758&  21.61050&  21.21633&  21.00327  & 1.85 & 20.29719\\
	&22.31279&  21.64527&  21.24160&  21.00398 &  1.68& 20.42274\\
	&37.33594&  35.95812 & 35.22679 & 34.82632  & 2.24 & 33.81667\\
	&37.35343&  36.01392&  35.23896&  34.83680 &  2.03&  33.62683\\
	\bottomrule                   
\end{tabular}
\end{center}}
\caption{Lowest computed eigenvalues for polynomial degrees $k=0$ in the unit sphere domain.}
\label{tabla:sphere}
\end{table}

\begin{figure}[H]
	\centering\includegraphics[scale=0.4]{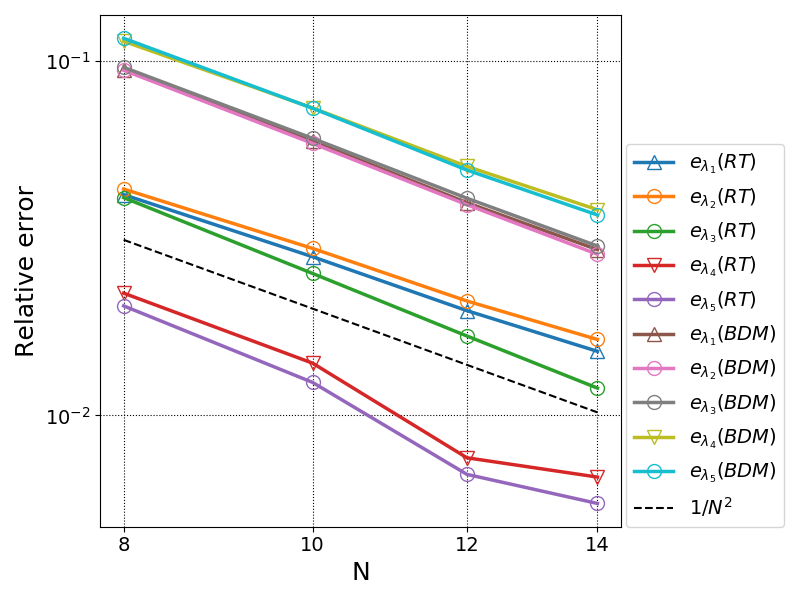}
	\caption{Comparison of the relative error behavior in the unit sphere when using $[\mathbb{P}_k]^{n}\text{-}\mathbb{P}_k\text{-}\mathbb{RT}_k$  and $[\mathbb{P}_k]^{n}\text{-}\mathbb{P}_k\text{-}\mathbb{BDM}_{k+1}$   schemes. The experiment considers polynomials of degree $k=0$. }
	\label{figura:erro-esfera}
\end{figure}


\begin{figure}[H]
	\begin{minipage}{0.45\linewidth}
		\centering\includegraphics[scale=0.1]{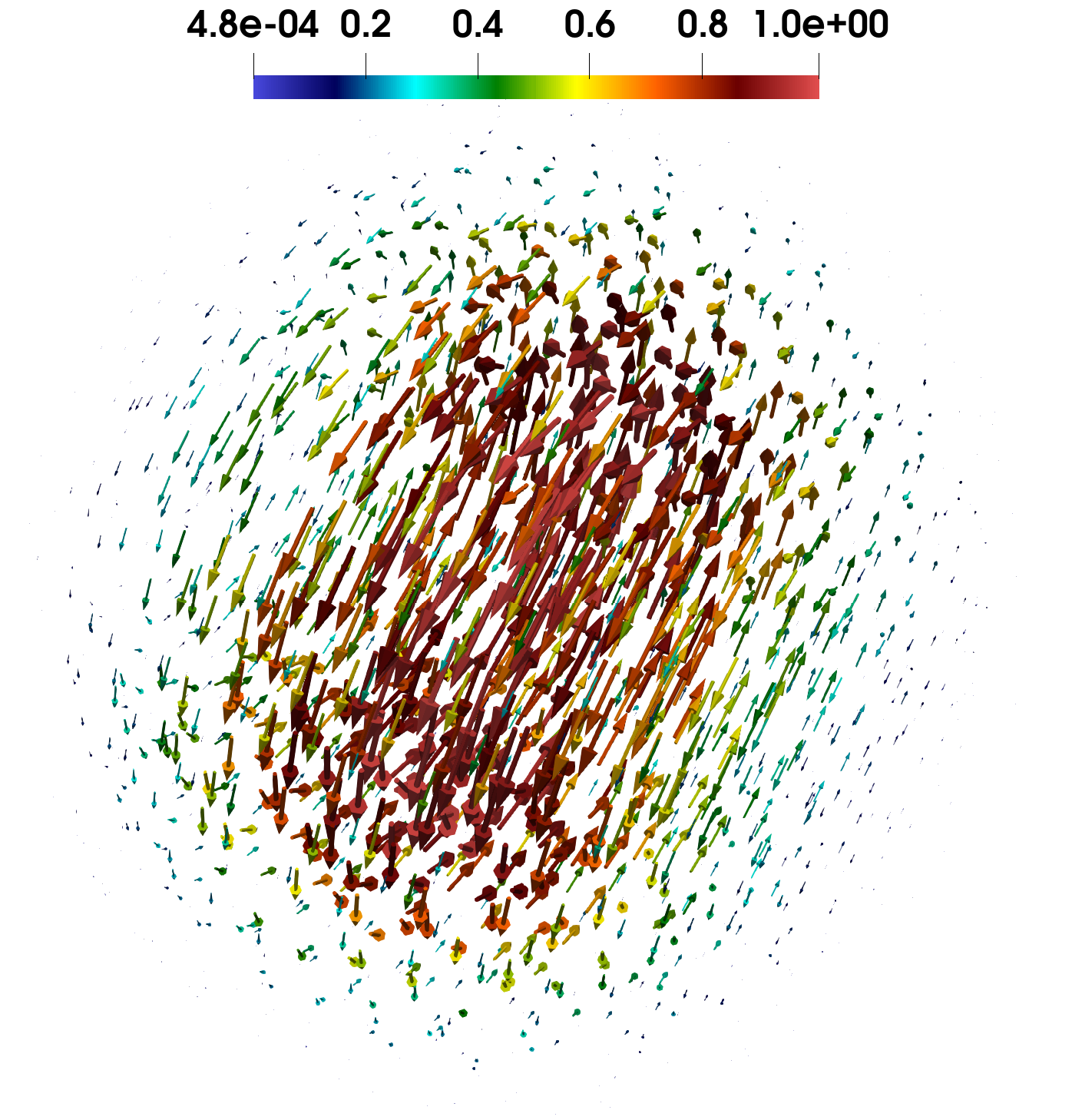}
	\end{minipage}
	\begin{minipage}{0.45\linewidth}
		\centering\includegraphics[scale=0.1]{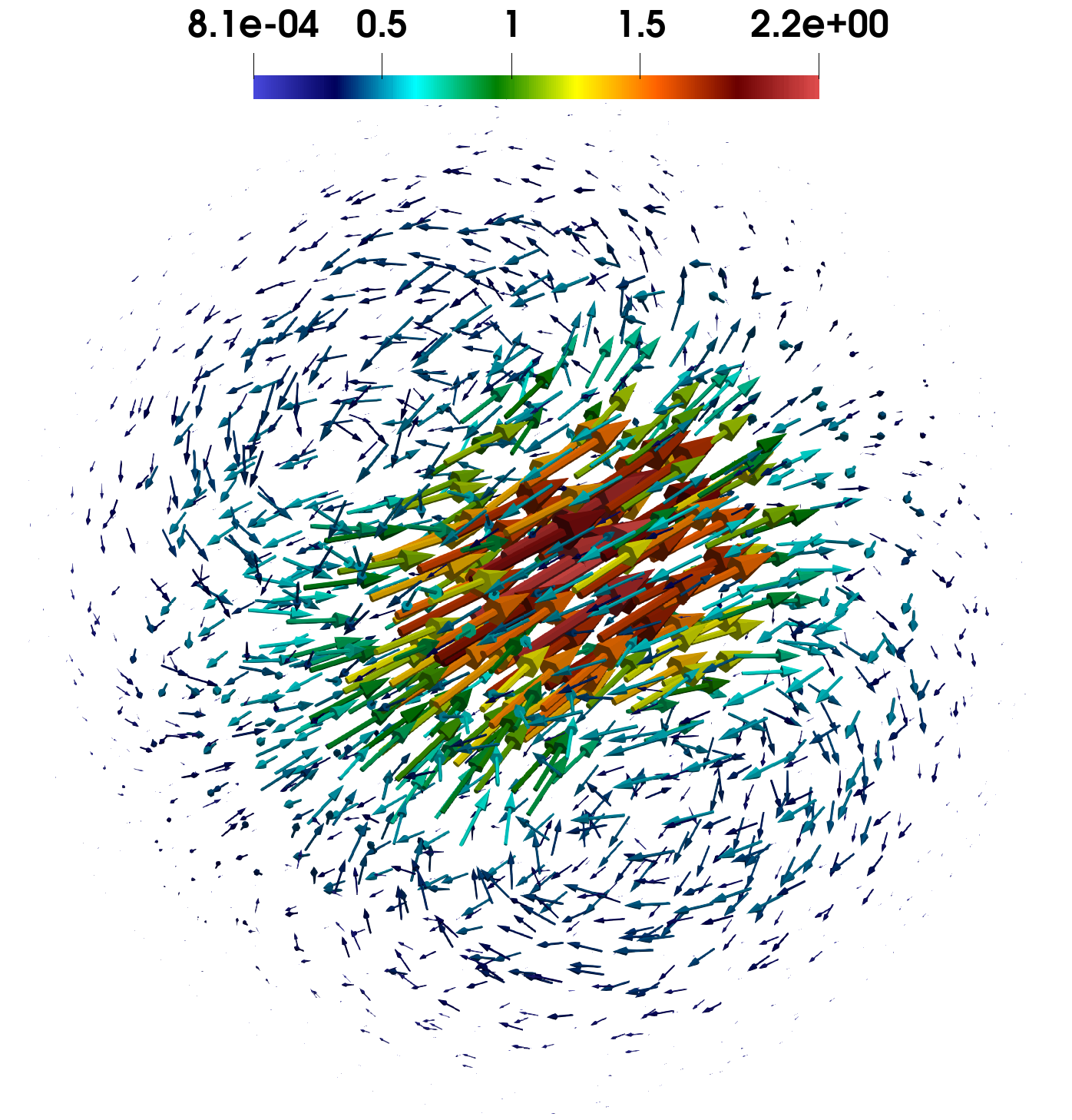}
	\end{minipage}\\
	\begin{minipage}{0.45\linewidth}
	\centering\includegraphics[scale=0.1]{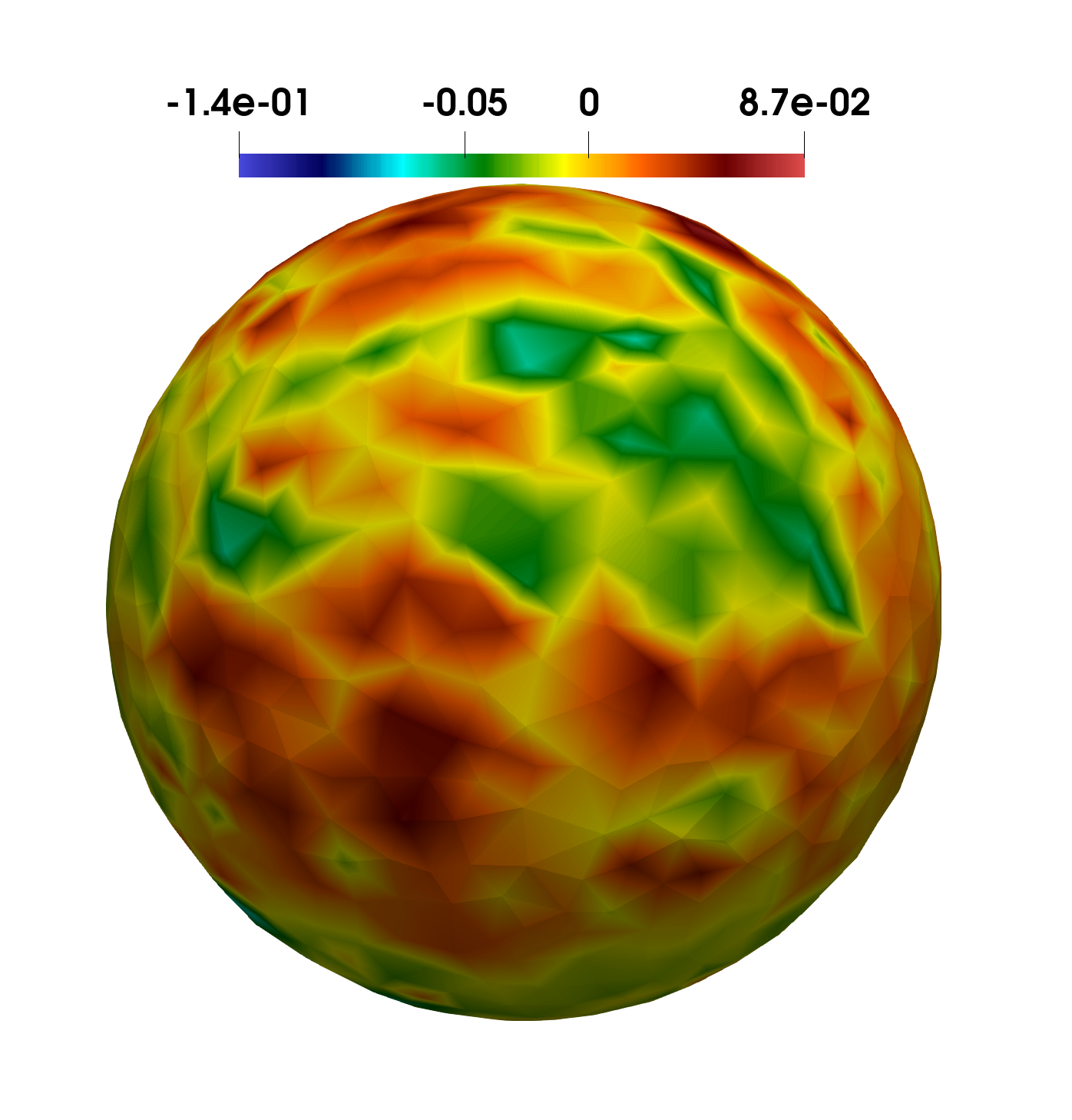}
	\end{minipage}
	\begin{minipage}{0.45\linewidth}
	\centering\includegraphics[scale=0.1]{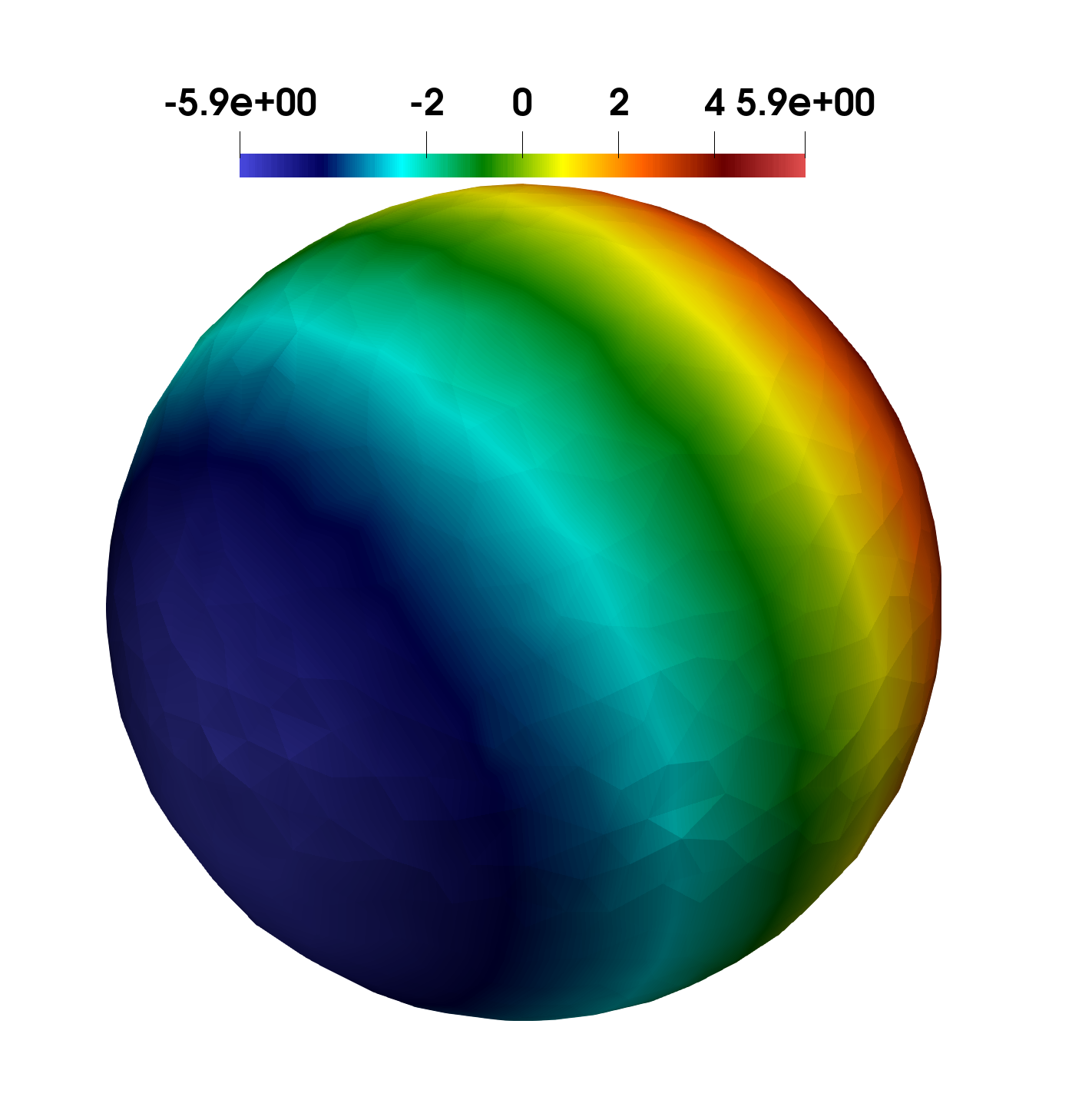}
	\end{minipage}
	\caption{Approximated velocity field $\bu_h$ (top) and pressure $p_h$ (bottom), corresponding to the second and fourth lowest eigenvalues in the unit sphere domain.}
	\label{figura:modos-esfera}
\end{figure}
%

\section{Conclusions}
From our analysis and numerical tests, we derive the following conclusions:
\begin{itemize}
\item The introduction of the pseudo-stress tensor allowed to propose a mixed finite element formulation for a Stokes eigenvalue problem that do not introduce spurious modes, and the analysis was possible thanks to definition of  appropriate solution operators. This allowed to obtain the respective approximation results along with theoretical convergence rates based on well-known finite element spaces.
\item The numerical schemes employed in our study
perform an accurate approximation of the eigenvalues and the associated eigenfunctions in two and three dimensions.
\item For curved domains both  methods works perfectly, even from the fact that we are considering polygonal/polyhedral meshes, depending in the dimension fo the domain.
\item For the lowest order in both numerical schemes ($k=0$) the double order of convergence is clearly
quadratic and, for $k\geq 1$, the$[\mathbb{P}_k]^{n}\text{-}\mathbb{P}_k\text{-}\mathbb{BDM}_{k+1}$   scheme seems to be more stable than the $[\mathbb{P}_k]^{n}\text{-}\mathbb{P}_k\text{-}\mathbb{RT}_k$   when the order fo convergence are computed. This is due the fact that, when $[\mathbb{P}_k]^{n}\text{-}\mathbb{P}_k\text{-}\mathbb{RT}_k$   is considered, the computed eigenvalues are more close between them than the scheme with BDM elements. This phenomenon is observable in the relative error plots.
\item In non convex domains, the results are the expectable due the singularity of the geometry. This is an interesting fact that motivates the analysis of adaptive schemes.
\end{itemize}

%
%
\bibliographystyle{siam}
\footnotesize
\bibliography{bib_LOQ}

\end{document}